\documentclass[a4paper]{amsart}

\usepackage[applemac]{inputenc}
\usepackage[T1]{fontenc}
\usepackage[english]{babel}
\usepackage{babelbib}
\usepackage{enumerate}
\usepackage{url}
\usepackage{amssymb}
\usepackage[cal=boondoxo]{mathalfa}
\usepackage{scalerel}

\DeclareFontFamily{U}{skulls}{}
\DeclareFontShape{U}{skulls}{m}{n}{ <-> skull }{}

\usepackage[hypertexnames=false,%backref=page,%pdftex, 
 	pdfpagemode=UseNone,
 	breaklinks=true,
 	extension=pdf,
 	colorlinks=true,
 	linkcolor=blue,
 	citecolor=red,
 	urlcolor=blue,
 ]{hyperref} 
\usepackage{color}

\usepackage{tikz}
\usepackage{tikz-cd}

\usetikzlibrary{matrix}
\usetikzlibrary{arrows}

\usepackage{graphicx}

\usepackage{cleveref}

\mathcode`A="7041 \mathcode`B="7042 \mathcode`C="7043 \mathcode`D="7044
\mathcode`E="7045 \mathcode`F="7046 \mathcode`G="7047 \mathcode`H="7048
\mathcode`I="7049 \mathcode`J="704A \mathcode`K="704B \mathcode`L="704C
\mathcode`M="704D \mathcode`N="704E \mathcode`O="704F \mathcode`P="7050
\mathcode`Q="7051 \mathcode`R="7052 \mathcode`S="7053 \mathcode`T="7054
\mathcode`U="7055 \mathcode`V="7056 \mathcode`W="7057 \mathcode`X="7058
\mathcode`Y="7059 \mathcode`Z="705A

\newcommand{\bbC}{\mathbb{C}}

\newcommand{\bbF}{\mathbb{F}}
\newcommand{\bbG}{\mathbb{G}}

\newcommand{\bbL}{\mathbb{L}}

\newcommand{\bbN}{\mathbb{N}}

\newcommand{\bbP}{\mathbb{P}}
\newcommand{\bbQ}{\mathbb{Q}}
\newcommand{\bbR}{\mathbb{R}}
\newcommand{\bbS}{\mathbb{S}}

\newcommand{\bbV}{\mathbb{V}}
\newcommand{\bbW}{\mathbb{W}}

\newcommand{\bbZ}{\mathbb{Z}}

\newcommand{\cA}{\mathcal{A}}

\newcommand{\cC}{\mathcal{C}}
\newcommand{\cD}{\mathcal{D}}

\newcommand{\cH}{\mathcal{H}}

\newcommand{\cL}{\mathcal{L}}
\newcommand{\cM}{\mathcal{M}}
\newcommand{\cN}{\mathcal{N}}
\newcommand{\cO}{\mathcal{O}}

\newcommand{\cT}{\mathcal{T}}

\newcommand{\cV}{\mathcal{V}}

\newcommand{\cX}{\mathcal{X}}

\newcommand{\rH}{\textup{H}}

\newcommand{\rL}{\textup{L}}

\newcommand{\rR}{\textup{R}}
\newcommand{\rS}{\textup{S}}
\newcommand{\rT}{\textup{T}}

\newcommand{\frS}{\mathfrak{S}}

\newcommand{\sssT}{\scriptscriptstyle{T}}

\newcommand{\tfstar}{{_{\sssT}f}_*}

\newcommand{\tf}[1]{{_{\sssT}{#1}}}

\renewcommand{\ss}{\textup{ss}}

\newcommand{\reg}{\textup{reg}}

\renewcommand{\top}{\textup{top}}

\newcommand{\ontoo}{\too\hspace{-14pt}\too}

\newcommand{\dashto}{\dashrightarrow}
\newcommand{\into}{\hookrightarrow}
\newcommand{\too}{\longrightarrow}
\renewcommand{\phi}{\varphi}
\renewcommand{\epsilon}{\varepsilon}
\renewcommand{\ker}{\Ker}

\newcommand{\iso}{\simeq}

\newcommand{\Dbc}{{\mathrm{D}^b_c}}

\newcommand{\Perv}{\textup{Perv}}

\newcommand{\Pbar}{\overline{\Perv}}
\newcommand{\PLambda}{\scalerel*{{\rotatebox[origin=c]{180}{$\bbV$}}}{\bbV}}

\newcommand{\intoo}{\lhook\joinrel\longrightarrow}
\newcommand{\sing}{{\mathrm{sing}}}
\DeclareMathOperator{\Vect}{Vect}
\DeclareMathOperator{\Alb}{Alb}

\DeclareMathOperator{\rk}{rk}
\DeclareMathOperator{\pr}{pr}

\DeclareMathOperator{\Spec}{Spec}
\DeclareMathOperator{\Proj}{Proj}

\DeclareMathOperator{\Hom}{Hom}

\DeclareMathOperator{\im}{Im}
\DeclareMathOperator{\Ker}{Ker}

\DeclareMathOperator{\coker}{Coker}

\newcommand{\red}{{\mathrm{red}}}

\DeclareMathOperator{\Gr}{Gr}

\DeclareMathOperator{\End}{End}

\DeclareMathOperator{\Sym}{Sym}

\DeclareMathOperator{\Aut}{Aut}
\DeclareMathOperator{\GL}{GL}
\DeclareMathOperator{\SL}{SL}
\DeclareMathOperator{\SO}{SO}
\DeclareMathOperator{\Sp}{Sp}

\DeclareMathOperator{\Supp}{Supp}
\DeclareMathOperator{\Lie}{Lie}

\DeclareMathOperator{\id}{id}

\DeclareMathOperator{\IC}{IC}

\DeclareMathOperator{\characteristic}{char}

\DeclareMathOperator{\pt}{pt}

\DeclareMathOperator{\codim}{codim}

\DeclareMathOperator{\Alt}{Alt}

\DeclareMathOperator{\cc}{cc}
\DeclareMathOperator{\CC}{CC}
\DeclareMathOperator{\Char}{Char}
\newcommand{\PChar}{\bbP\Char}

\DeclareMathOperator{\DR}{DR}
\DeclareMathOperator{\HM}{HM}

\DeclareMathOperator{\Mbar}{\overline{\HM}}

\renewcommand{\le}{\leqslant}
\renewcommand{\ge}{\geqslant}

\theoremstyle{plain}
\newtheorem{theoremintro}{Theorem}

\newtheorem*{maintheorem-monodromy}{Main theorem (monodromy version)}
\newtheorem*{maintheorem-tannaka}{Main theorem (Tannaka version)}

\newtheorem{theorem}{Theorem}[section]
\newtheorem{lemma}[theorem]{Lemma}
\newtheorem{proposition}[theorem]{Proposition}
\newtheorem{corollary}[theorem]{Corollary}

\newtheorem*{bigmonodromyintro}{Big Monodromy Criterion}

\theoremstyle{definition}

\newtheorem{definition}[theorem]{Definition}
\newtheorem{example}[theorem]{Example}

\newtheorem{remark}[theorem]{Remark}
\newtheorem*{exampleintro}{Example}

\numberwithin{equation}{section}

\newcommand{\marco}[1]{{\color{blue} \sf   M: [#1]}}

\begin{document}

\author[T. Kr\"amer]{Thomas Kr\"amer}
\address{Thomas Kr\"amer \\
	Fakult\"{a}t f\"{u}r Mathematik\\
	Technische Universit\"{a}t Chemnitz\\
	Rei\-chenhainer Stra\ss e 39, 09126 Chemnitz\\
	Germany.}
\email{thomas.kraemer@math.tu-chemnitz.de}

\author[C. Lehn]{Christian Lehn}
\address{Christian Lehn\\ 
Fakult\"at f\"ur Mathematik\\Ruhr-Universit\"at Bochum\\ 
Universit\"ats\-stra\ss e~150\\
Postfach IB 45\\
44801 Bochum, Germany}
\email{christian.lehn@rub.de}

\author[M. Maculan]{Marco Maculan}
\address{Marco Maculan \\
Institut de Math\'ematiques de Jussieu \\
Sorbonne Universit\'e \\
4 place Jussieu \\
75005 Paris \\
France}
\email{marco.maculan@imj-prg.fr}

\title[Exceptional Tannaka groups only arise from cubic threefolds]{Exceptional Tannaka groups only arise\\ from cubic threefolds}

\subjclass[2020]{14K12 (primary), 14D05, 14J70, 18M25, 20G05, 32S60 (secondary).}
\keywords{Subvarieties of abelian varieties, perverse sheaves, Hodge modules, Tannaka categories, exceptional groups, Fano surfaces, cubic threefolds.}

% 14K12 	Subvarieties of abelian varieties 
% 14D05		Structure of families (Picard-Lefschetz, monodromy, etc.)
% 14J70		Hypersurfaces and algebraic geometry
% 14C30 	Transcendental methods, Hodge theory (algebro-geometric aspects) 
% 18M25 	Tannakian categories 
% 20G05 	Representation theory for linear algebraic groups 
% 32S60  	Stratifications; constructible sheaves; intersection cohomology (complex-analytic aspects)

\begin{abstract}
We show that under mild assumptions, the Fano surfaces of lines on smooth cubic threefolds are the only smooth subvarieties of abelian varieties whose Tannaka group for the convolution of perverse sheaves is an exceptional simple group. This in particular leads to a considerable strengthening of our previous work on the Shafarevich conjecture. A key idea is to control the Hodge decomposition on cohomology by a cocharacter of the Tannaka group of Hodge modules, and to play this off against an improvement of the Hodge number estimates for irregular varieties by Lazarsfeld-Popa and Lombardi. 
\end{abstract}

\maketitle

\setcounter{tocdepth}{1}

\tableofcontents

\thispagestyle{empty}

%\vspace{-2em}

\section{Introduction} 

Any subvariety of an abelian variety gives rise to a reductive group via the convolution of perverse sheaves. For smooth subvarieties these Tannaka groups have recently been used to obtain arithmetic finiteness results for varieties over number fields~\cite{LS20,KM} and the big monodromy criterion in~\cite{JKLM}. For both applications one has to exclude exceptional Tannaka groups. Such exceptional groups are known to occur for Fano surfaces of smooth cubic threefolds, which have Tannaka group~$E_6$ by~\cite{KraemerCubicThreefolds}. We show that in the range relevant for the above applications, these Fano surfaces are the only examples with exceptional Tannaka groups; this leads to a considerable strengthening of~\cite{KM, JKLM}. The key 
idea is of a general nature and goes far beyond the application in this paper: We upgrade the Tannaka formalism from perverse sheaves to complex Hodge modules in the sense of Sabbah-Schnell~\cite{MHMProject}. 
In particular, we prove a comparison theorem between the Tannaka groups for perverse sheaves and for Hodge modules which implies that the Hodge decomposition is defined by a cocharacter of the perverse Tannaka group. This puts a strong restriction on the arising Tannaka groups 
that will be useful also for other applications. 

\subsection{Tannaka groups of subvarieties} \label{sec:TannakaGroupIntro} Let~$A$ be a~$g$-dimensional abelian variety over an algebraically closed field~$k$ of characteristic zero, and consider an integral smooth subvariety~$X\subset A$ of dimension~$d<g$. Our results apply both in an analytic and in an algebraic setup: For complex varieties we use perverse sheaves with coefficients in~$\bbF = \bbC$, while for varieties over general fields~$k$ we use~$\ell$-adic perverse sheaves with coefficients in~$\bbF = \overline{\bbQ}_\ell$ for a prime~$\ell$. We define the perverse intersection complex 
\[ \delta_X \;:=\; i_\ast\bbF_X[d] \]
as the pushforward of the constant sheaf under the closed immersion~$i \colon X \into A$, shifted in cohomological degree~$-d$ so that it becomes an object of the abelian category~$\Perv(A, \bbF)$  of perverse sheaves on~$A$ as in \cite{BBDG}. The group law on the abelian variety induces a convolution product on perverse sheaves so that the convolution powers of~$\delta_X$ generate a neutral Tannaka category~$\langle \delta_X \rangle$, see for instance~\cite[sect.~3]{JKLM}.
In what follows, we fix a fiber functor 
\[
 \omega \colon \quad \langle \delta_X \rangle \;\too\; \Vect(\bbF).
\]
The automorphisms of this fiber functor are represented by a reductive algebraic group 
\[G_{X, \omega}\; := \; G_{\omega}(\delta_X)\; \subset \; \GL(\omega(\delta_X)) \] 
which we call the \emph{Tannaka group of the subvariety~$X\subset A$}. If~$X$ is invariant under translation by a nonzero point of~$A$, then it descends to the quotient of~$A$ modulo the subgroup generated by that point. Hence, in what follows we assume that~$X$ is~\emph{nondivisible} in the sense that it is not stable under the translation by any nonzero point of~$A$. In this case, the connected component of the identity~$G_{X,\omega}^\circ \subset G_{X,\omega}$ acts irreducibly on the representation~$\omega(\delta_X)$ by~\cite[prop.~3.3]{JKLM}. We are  interested in its derived group
\[
 G_{X,\omega}^\ast \;:=\; [G_{X,\omega}^\circ, G_{X,\omega}^\circ],
\]
which is a connected semisimple algebraic group. 
The smoothness of~$X$ implies that~$\omega(\delta_X)$ is a minuscule representation~\cite[cor.~1.10]{KraemerMicrolocalI} \cite[cor 5.15]{JKLM} of dimension equal to the absolute value of the topological Euler characteristic of~$X$:
\[ \dim \omega(\delta_X) \; = \; |\chi_{\top}(X)|. \]
This leaves only very few possibilities because the Tannaka group is known to be simple in most cases: Suppose that~$g\ge 3$ and that the normal bundle of~$X\subset A$ is ample. Then by~\cite[th. A]{JKLM}, the  group~$G_{X, \omega}^\ast$ fails to be simple if and only if there exist smooth positive-dimensional subvarieties~$X_1, X_2 \subset A$ such that the sum morphism induces an isomorphism
\[ X_1 \times X_2 \stackrel{\sim}{\too} X = X_1 + X_2.\]
Beyond the standard representations of classical groups, there are only few other minuscule representations of simple algebraic groups: Wedge powers of the standard representation of~$\SL_n$, (half)spin representations of spin groups, and the smallest irreducible representations of the exceptional groups~$E_6$ and~$E_7$. In \cite{JKLM} we completely characterized wedge powers and ruled out spin representations; the main goal of this paper is to understand when exceptional groups occur.

\subsection{Exceptional Tannaka groups} 

The group~$E_6$ does occur as the Tannaka group of the Fano surface of lines on a smooth cubic threefold:

\begin{exampleintro} \label{Ex:CubicThreefoldE6} 
Let~$Y \subset \bbP^4$ be a smooth cubic threefold. The Fano variety of lines on~$Y$ is a smooth projective surface~$X$ whose Albanese variety~$A=\Alb(X)$ has dimension~$5$. Moreover, for any chosen base point the corresponding Albanese morphism is a closed embedding~$X \hookrightarrow \Alb(X)$. In what follows, we fix a base point and view~$X$ as a subvariety of its Albanese variety. By \cite[th.~2]{KraemerCubicThreefolds} then
\[ (G^\ast_{X,\omega}, \omega(\delta_X)) \; \iso \; (E_6, V) \]
where~$V$ is one of the two irreducible representations of~$E_6$ with~$\dim V = 27$. Via the Tannaka formalism, the trilinear form defining $E_6$ corresponds to an irreducible fiber of the sum map $X \times X \times X \to A$  of dimension $3$ parametrizing triples of coplanar lines in $Y$. It is clear from the description of~$X$ as a Fano variety of lines that the classical Gauss map
\[
\quad X \;\intoo\; \Gr_2(\Lie A)
\]
is an embedding, hence $X\subset A$ is nondivisible. The difference morphism
\[
 X\times X \;\longrightarrow\; X - X \;\subset\; A
\]
is generically finite of degree $6$ and its fibers are given by classical ``double-six'' configurations of lines on smooth cubic surfaces \cite[sect. 13]{ClemensGriffiths}. Its image is a theta divisor of the principally polarized abelian variety~$A$ and permits to recover the cubic threefold \cite[th.~13.4]{ClemensGriffiths} \cite{BeauvilleIntermediateJacobian}: The theta divisor has an isolated singularity at the origin and its tangent cone~$C$ there can be identified with the cubic threefold~$Y$ in suitable linear coordinates on~$\Lie A$. In the same coordinates~$Y$ is also identified with the image of the projection
\[ \bbP(T_X) \; \too \; \bbP(\Lie A).\]
\end{exampleintro}

\noindent
The goal of this paper is to show that under mild assumptions, the above is the only example of exceptional Tannaka groups. In particular, in the dimension range relevant for~\cite{KM, JKLM} we obtain (see \cref{Thm:E6VarietyIsFanoSurface}):

\begin{theoremintro} \label{Thm:E6VarietyIsFanoSurfaceIntro} Let~$X \subset A$ be a  smooth irreducible subvariety with ample normal bundle and dimension~$< g/2$. Then the following are equivalent:
\begin{enumerate}
\item $X\subset A$ is nondivisible with Tannaka group~$G_{X,\omega}^\ast \simeq E_6$.\smallskip
\item $X$ is isomorphic to the Fano surface of lines on a smooth cubic threefold, and the canonical morphism~$\Alb(X) \to A$ is an isogeny. 
\end{enumerate}
\end{theoremintro}

The Fano surface~$X$ of lines on a smooth cubic threefold~$Y \subset \bbP^4$ does not always have ample normal bundle in its Albanese variety: This fails for example for the cubic threefold with affine equation
\[ y^2 - f(x) + z^2 t + t^3 \;=\; 0,\]
where~$f(x)$ is a degree~$3$ polynomial without multiple roots. In fact, it is easy to see that the normal bundle of an Albanese embedding~$X \into \Alb(X)$ fails to be ample if and only if there is a hyperplane~$H \subset \bbP^4$ such that the cubic surface~$H \cap Y \subset H$ is a cone over an elliptic curve. Such hyperplanes are in bijection with the vertices of the corresponding cones, which are the \emph{Eckardt points} of~$Y$. There are at most finitely many such Eckardt points and their absence is equivalent to ampleness of the cotangent bundle of~$X$. See \cite[sect. 2]{RoulleauManuscripta} for details.\medskip 

To include cubic threefolds with Eckardt points, we relax the ampleness of the normal bundle to a weaker positivity property: We say that an integral subvariety~$Z \subset A$ of an abelian variety~$A$ is \emph{nondegenerate} if for every surjective morphism of abelian varieties~$\phi \colon A \to B$ we have
\[ \dim \phi(Z) = \min \{ \dim Z, \dim B \}.\]
The Fano surface of lines on any smooth cubic threefold is nondegenerate inside its Albanese variety by~\cite[th.~1]{SchreiederIMRN}, since it is a summand of a theta divisor; in fact loc.~cit.~refers to a stronger notion of nondegeneracy, which implies the above one by \cite[lemma II.12]{Ran}. Our proof of~\cref{Thm:E6VarietyIsFanoSurfaceIntro} proceeds by reduction to the case of surfaces. In this case we only need to assume nondegeneracy, hence we obtain the following result which now includes the Fano surfaces of \emph{all} smooth cubic threefolds  (see \cref{Thm:E6SurfaceIsFanoSurface}):

\begin{theoremintro} \label{Thm:E6SurfaceIsFanoSurfaceIntro} Suppose~$g \ge 5$. For any smooth irreducible surface~$X \subset A$, the following three properties are equivalent:\smallskip
\begin{enumerate}
\item $X\subset A$ is nondivisible and nondegenerate with Tannaka group~$G_{X,\omega}^\ast \simeq E_6$.\smallskip
\item $X\subset A$ is nondivisible and nondegenerate with 
\[ \chi(X, \cO_X)=6,
 \qquad c_2(X)= 27,
\]
the difference morphism~$ X\times X \to X-X$ has generic degree~$\ge 6$, and the sum morphism $X \times X \times X \to A$ has an irreducible fiber of dimension $\ge 3$.
\smallskip
\item $X$ is isomorphic to the Fano surface of lines on a smooth cubic threefold, and the canonical morphism~$\Alb(X) \to A$ is an isogeny. \smallskip
\end{enumerate}
\end{theoremintro}

A different characterization has been obtained by Casalaina-Martin, Popa and Schreieder in \cite[th. 6.1]{CasalainaPopaSchreieder}: They show that if $X$ is a subvariety of minimal cohomology class of a principally polarized abelian fivefold $A$ and $X - X$ is a theta divisor, then $X$ is the Fano surface of lines on a smooth cubic threefold. Note that in both theorems \ref{Thm:E6VarietyIsFanoSurfaceIntro} and \ref{Thm:E6SurfaceIsFanoSurfaceIntro}, the canonical morphism~$\Alb(X) \to A$ need not be an isomorphism: 

\begin{exampleintro} \label{ex:TheFanoSurfaceAsAnE6}
Let~$X$ be the Fano surface of lines on a smooth cubic threefold, and let
\[
 \Theta \;=\; X - X \;\subset\; \Alb(X).
\]
be its theta divisor. The upper bound on the number of 2-torsion points on theta divisors in~\cite{MarcucciPirola}
shows that for any theta divisor on a principally polarized abelian variety, there exists a 2-torsion point outside that theta divisor. Applying this to our case, we can find a point~$a\in \Alb(X)$ of order two with~$a \notin \Theta = X-X$. We then have
\[
 X \cap (X+a) \;=\; \varnothing
\]
and it follows that the isogeny 
\[
 p\colon \quad \Alb(X) \;\longrightarrow\; A \;:=\; \Alb(X)/\langle a \rangle
\]
maps~$X\subset \Alb(X)$ isomorphically onto its image~$\bar{X} := p(X) \subset A$. So this image is still an integral smooth nondivisible subvariety. The pushforward of perverse sheaves under an isogeny does not change the derived connected component of their Tannaka groups~\cite[cor. 3.5]{JKLM}, so we still have 
\[ (G_\omega^\ast(\bar{X}), \omega(\delta_{\bar{X}})) \simeq (E_6, V)
\]
although the canonical morphism~$p\colon \Alb(\bar{X}) \to A$ is not an isomorphism.
\end{exampleintro}

\subsection{The Hodge filtration}

A key ingredient in our proof of theorems~\ref{Thm:E6VarietyIsFanoSurfaceIntro} and~\ref{Thm:E6SurfaceIsFanoSurfaceIntro} is the theory of complex Hodge modules by Sabbah-Schnell~\cite{MHMProject}. To explain the idea, let~$A$ be a complex abelian variety and~$\HM(A)$ the category of direct sums of complex Hodge modules of arbitrary weight on it, see~\cref{sec:HodgeModules} for details. The category~$\HM(A)$ is semisimple~$\bbC$-linear abelian and comes equipped with a faithful exact~$\bbC$-linear functor
\[
 \DR\colon \quad \HM(A) \;\longrightarrow\; \Perv(A, \bbC).
\]
To any~$M\in \HM(A)$ one may attach a Tannaka group~$G_\omega(M)$ in the same way as for perverse sheaves; in~\cref{cor:perverse-versus-hodge} we show a comparison theorem between this group and the Tannaka group~$G_\omega(P)$ of the perverse sheaf~$P=\DR(M)$. In particular, the derived groups of their connected components of the identity are the same:
\[
G_\omega^\ast(M) \;=\; G_\omega^\ast(P).
\]
This allows us to enrich with Hodge theoretic data all the previous Tannaka constructions for perverse sheaves, so for the rest of this section we will work only with Hodge modules. It is now time to be more specific about our choice of fiber functors: Fix a Hodge module~$M\in \HM(A)$, and let~$L$ be a unitary local system of rank one on~$A$ with the property that all perverse subquotients of~$\DR(M)\otimes L$ have their cohomology concentrated in degree zero; there are plenty of such local systems by  generic vanishing for perverse sheaves on abelian varieties~\cite{KWVanishing, SchnellHolonomic}. Then the functor
\[
 \omega \colon \quad \langle M \rangle \;\longrightarrow\; \Vect(\bbC), \qquad N \;\longmapsto\; \rH^0(A, \DR(N)\otimes L)
\]
is a fiber functor. Now~$L$ underlies a complex variation of Hodge structures~$\bbL$, so from the natural identification~$\DR(M)\otimes L = \DR(M\otimes \bbL)$ we obtain a Hodge decomposition 
\[
 \rH^0(A, \DR(M)\otimes L) \;=\; \bigoplus_{(p,q)\in \bbZ^2} \rH^{p,q}(M\otimes \bbL)
\]
We show that this Hodge decomposition is compatible with the convolution product, which allows to enrich~$\omega$ to a tensor functor 
$\omega^H\colon 
\langle M \rangle \rightarrow \Vect_{\bbZ\times \bbZ}(\bbC)$
with values in the category of finite-dimensional bigraded complex vector spaces. This leads to the following result (see~\cref{cor:hodge-cocharacter}):

\begin{theoremintro} 
There is a natural morphism~$\lambda \colon \bbG_m^2 \to G_\omega(M)$ such that the Hodge decomposition
\[
 \omega^H(M) \;\;= \bigoplus_{(p,q)\in \bbZ^2} \rH^{p,q}(M\otimes \bbL)
\]
is the decomposition in weight spaces for the characters~$(p,q)\in \bbZ^2 = \Hom(\bbG_m^2, \bbG_m)$.
\end{theoremintro} 

We call~$\lambda$ the \emph{Hodge cocharacter} of~$M$. The fact that this cocharacter takes values in the subgroup 
\[ G_\omega^\ast(M) \;\subset\; \GL(\omega(M)) \]
imposes strong restrictions on the possible Hodge filtrations if the subgroup is known to be small. This in particular applies to the Tannaka group of smooth subvarieties~$X\subset A$: The perverse intersection complex~$\delta_X\in \Perv(A, \bbC)$ naturally lifts to a complex Hodge module~$\delta_X^H\in \HM(A)$ of weight~$d=\dim X$ and its Hodge numbers are 
\[
 \dim \rH^{p,d-p}(\delta_X^H \otimes \bbL) \;=\; (-1)^{d-p} \chi(X, \Omega^p_X),
\] 
see~\cref{sec:hodge-level-of-IC}. Comparing this with an improvement to the Hodge estimates by Lazarsfeld-Popa~\cite{LazarsfeldPopa} and Lombardi~\cite{Lombardi13} explained in  \cref{sec:HodgeEstimates}, we obtain the following result (see~\cref{prop:hodge-numbers}):

\begin{theoremintro} \label{Thm:hodge-numbers-intro}
Let~$X\subset A$ be a smooth nondivisible irreducible subvariety with ample normal bundle and dimension $d$ such that~$G_{X, \omega}^\ast$ is a simple exceptional group. Then, either\smallskip 
\begin{enumerate} 
\item $G_{X,\omega}^\ast \simeq E_6$,~$|\chi_{\top}(X)| = 27$ and~$d\in \{2,4,6\}$, or\smallskip 
\item $G_{X,\omega}^\ast \simeq E_7$,~$|\chi_{\top}(X)| = 56$ and~$d\in \{3,\dots,15\}$ is odd. \smallskip
\end{enumerate} 
Moreover~$g\le g_{\max}$ for the following upper bound~$g_{\max}$ depending on~$d$:
\begin{align*}
(1) \quad
\begin{array}{c|ccccc}
d & 2 & 4 & 6  \\ \hline
g_{\max} & 7 & 6 & 8 
\end{array}
&&\quad (2) \quad
\begin{array}{c|ccccccc}
d & 3 & 5 & 7 & 9 & 11 & 13 & 15 \\ \hline
g_{\max} & 9 & 10 & 12 & 13 & 15 & 16 & 18
\end{array}
\end{align*}
\end{theoremintro}

In each case we also know the Hodge numbers~$h^p(X):=(-1)^{d-p}\chi(X, \Omega^p_X)$, as they must be among those listed in \cref{prop:E6,,prop:E7}. This essentially reduces the proof of theorem~\ref{Thm:E6VarietyIsFanoSurfaceIntro} to the case of surfaces treated in theorem~\ref{Thm:E6SurfaceIsFanoSurfaceIntro}. For the proof of the latter we apply the representation theory of~$E_6$ to show that the difference morphism 
\[ X\times X \;\longrightarrow\; D \;:=\; X-X \;\subset\; A, \] is generically finite of degree~$\ge 6$ over its image (see~\cref{sec:DifferenceMorphism}), and then use this numerical information to show by direct geometric arguments that the projective tangent cone to~$D\subset A$ at the origin is a smooth cubic threefold whose associated Fano surface of lines is isomorphic to~$X$ (see~\cref{sec:proof-of-main-result}).

\subsection{No country for $E_7$} Although \cref{Thm:hodge-numbers-intro} imposes strong restrictions, it is not enough to rule out completely the existence of subvarieties $X$ with $G_{X, \omega}^\ast \iso E_7$
 in the dimension range $d < g/2$. 
With the $E_6$ case in mind, one might think that a natural candidate for such an $X$ would be the Fano variety of lines on a quartic double solid, that is, a double cover of $\bbP^3$ branched along a smooth quartic. Indeed, the $(-1)$-curves on a del Pezzo surface of degree~$3$ resp.~$2$ give rise to a root system of type $E_6$ resp. $E_7$. A cubic threefold can be thought of as a one parameter family of del Pezzo surfaces of degree $3$ since these are cubic surfaces. Similarly, a quartic double solid is a one parameter family of del Pezzo surfaces of degree $2$ since the anticanonical divisor realizes them as double covers of~$\bbP^2$ branched along smooth quartic curves.\medskip

However, the pieces of the jigsaw puzzle do not quite match together in the $E_7$ case: For instance, lines on a quartic double solid are parametrized by a surface, while \cref{Thm:hodge-numbers-intro} states that the sought-for example must be a threefold. This parity issue is hardwired in the representation theory of~$E_7$ and lies at the heart of the following negative result (see \cref{Cor:NoE7}):

\begin{theoremintro} \label{thmintro:NoE7} For any smooth irreducible subvariety $X\subset A$ with ample normal bundle and dimension~$< g/2$ we have
\[G_{X, \omega}^\ast \; \not \simeq \; E_7.\] 
\end{theoremintro}

The proof of this will be given in~\cref{sec:E7}. It does not require Hodge modules and is largely independent from the rest of this paper: The idea is to control the decomposition of the tensor square~$\delta_X*\delta_X$ by looking at the fibers of the sum morphism $X\times X \to A$ and using Kashiwara's estimate for characteristic varieties of direct images. Once we know how the tensor square decomposes, we can then conclude by a version of Larsen's alternative.

\subsection{Big Monodromy} As an application of the above results we can significantly strengthen the big monodromy criterion in \cite{JKLM}. More precisely, let~$S$ be a smooth complex variety. Let~$A$ be a complex abelian variety of dimension~$g$ and~$\cX \subset A_S := A\times S$ a subvariety such that the projection~$f \colon \cX \to S$ is a smooth morphism with connected fibers of dimension~$d$.  The situation is summarized in the following diagram:
\[
\begin{tikzcd}
&  \cX \ar[d, hook, no head, xshift=-1pt] \ar[d, no head, shorten <=1.2pt, xshift=1pt] \ar[dl, bend right=30, swap, "\pi"] \ar[dr, bend left=30, "f"]  & \\
A & A_S \ar[l, swap, "\pr_A"] \ar[r, "\pr_S"]  & S
\end{tikzcd}
\]
For a character~$\chi \colon \pi_1(A, 0) \to \bbC^\times$ let~$L_\chi$ denote the associated rank one local system on~$A$. Given an~$n$-tuple of such characters~$\underline{\chi} = (\chi_1, \dots, \chi_n)$, we consider the local system
\[ V_{\underline{\chi}} \;:=\; R^d f_{\ast} \pi^* L_{\underline{\chi}} \quad \textup{where} \quad L_{\underline{\chi}} \;:=\; L_{\chi_1} \oplus \cdots \oplus L_{\chi_n}.\]
The fiber at $s$ of this local system comes with a linear action of the group~$\pi_1(S,s)$ via the monodromy representation. With terminology as in loc.~cit.~we  obtain the following strengthening of the big monodromy result in \cite[sect. 1.1]{JKLM}:

\begin{bigmonodromyintro} Suppose~$X:= \cX_{\bar{\eta}} \subset A_{S, \bar{\eta}}$ has ample normal bundle and dimension~$d < (g-1)/2$. If~$\cX$ is symmetric up to translation, assume that~\eqref{Eq:NumericalConditions} below holds. Then the following are equivalent:
\begin{enumerate}
\item $X$ is nondivisible, not constant up to translation, not a symmetric power of a curve and not a product; \smallskip
\item $V_{\underline{\chi}}$ has big monodromy for most 
torsion~$n$-tuples of characters~$\underline{\chi}$.
\end{enumerate}
\end{bigmonodromyintro}

Here~$\cX$ is said to be \emph{symmetric up to translation} if there is a morphism~$a \colon S \to A$ with the property that~$\cX_t = a(t) - \cX_t$ for all~$t \in S(\bbC)$. Only in this case we need to exclude some very specific numerics by assuming that the topological Euler characteristic $e = \chi_{\top}(X)$ of $X$ has absolute value
 \begin{equation} \label{Eq:NumericalConditions}  
|e| \;\neq\;
2^{2m - 1} \quad \text{if~$d \ge (g-1)/4$ and $ m\in \{3, \dots, d\}, m\equiv d$ modulo $2$.}
\end{equation}
This assumption is empty for~$d < (g-1)/4$. Apart from curves, in \cite{JKLM} we had no control on the dimension and irregularity of varieties with exceptional Tannaka group, so we assumed~$|e| \neq 27, 56$ for all~$(d, g)$ in the given range.
Theorems \ref{Thm:E6VarietyIsFanoSurfaceIntro} and~\ref{thmintro:NoE7} now  permit to remove these assumptions completely. For applications outside this range, we only need to exclude the finite list of~$(d, g)$ in \cref{Thm:hodge-numbers-intro}.

\subsection*{Acknowledgements} 
We thank Victor Gonzalez-Alonso, Manfred Lehn, Luigi Lombardi,  Claude Sabbah, Christian Schnell and Mads Villadsen for discussions related to this project. The computer-aided searches in \cref{prop:E6,,prop:E7} were carried out by Mads Villadsen. We thank the referee for the careful reading and for valuable comments, which helped make the paper more accessible. C.~L. was supported by the DFG through the research grants Le 3093/3-2, Le 3093/5-1. T.~K. was supported by the DFG through the research grant Kr 4663/2-1. 

\subsection*{Conventions} A \emph{variety} over a field~$k$ is a separated~$k$-scheme of finite type, and a~\emph{subvariety} is a closed subvariety. For a vector bundle~$\cV$ on a variety we denote by~$\bbP(\cV)$ the projective bundle of lines in~$\cV$.

\section{Tannaka groups of Hodge modules} \label{sec:HodgeModules}

In this section we introduce Tannaka categories of Hodge modules on complex abelian varieties and relate their Tannaka groups to those for perverse sheaves.

\subsection{Complex Hodge modules} For a complex manifold~$X$ and~$w\in \bbZ$, we denote by~$\HM(X, w)$ the category of polarizable complex Hodge modules of weight~$w$ on~$X$ in the sense of~\cite[def.~14.2.2]{MHMProject}. This is a semisimple abelian~$\bbC$-linear category, and it comes with a faithful exact~$\bbC$-linear functor
$
 \DR\colon \HM(X, w) \rightarrow \Perv(X)
$ 
to the category~$\Perv(X)$ of perverse sheaves on~$X$ with complex coefficients. We denote by
\[ \HM(X) \;:=\; \bigoplus_{w\in \bbZ} \HM(X, w) \]
the category whose objects are the formal direct sums~$M=\bigoplus_{w\in \bbZ} M_w$ of Hodge modules~$M_w \in \HM(X, w)$ such that~$M_w = 0$ for all but finitely many weights~$w$, with morphisms
\[
 \Hom_{\HM(X)}(M, N) \;:=\; \bigoplus_{w\in \bbZ} \Hom_{\HM(X, w)} (M_w, N_w). 
\] 
This is again a semisimple abelian~$\bbC$-linear category with a faithful exact~$\bbC$-linear functor
\[
 \DR\colon \quad \HM(X) \;\longrightarrow\; \Perv(X), \quad M \;\longmapsto\; \bigoplus_{w\in \bbZ} \DR(M_w).
\]

\begin{example} \label{hodge-modules-on-a-point}
If~$X=\{\mathrm{pt}\}$ is a point, then~$\HM(\{\mathrm{pt}\})$ is the category~$\Vect_{\bbZ \times \bbZ}(\bbC)$ of finite-dimensional bigraded complex vector spaces
\[
 V \;=\; \bigoplus_{p,q\in \bbZ} V^{p,q}
\]
and~$\HM(\{\mathrm{pt}\}, w)$ is its subcategory of \emph{pure~$\bbC$-Hodge structures of weight~$w$}, by which we mean those bigraded vector spaces that satisfy~$V^{p,q}=0$ for all~$p,q\in \bbZ$ with~$p+q\neq w$. Moreover, 
\[ \DR\colon \quad \HM(\{\mathrm{pt}\}) = \Vect_{\bbZ \times \bbZ}(\bbC) \; \too \; \Perv(\{\mathrm{pt}\})=\Vect(\bbC) \]
is the functor that forgets the grading. 
\end{example}

By a \emph{variation of~$\bbC$-Hodge structures} of weight~$w$ on~$X$ we mean a~$\cC^\infty$-bundle~$\cV$ together with a flat connection and a decomposition into~$\cC^\infty$-subbundles
\[
 \cV \;=\; \bigoplus_{p+q=w} \cV^{p,q}
\]
that satisfies Griffiths transversality. In what follows we only consider \emph{polarizable} variations of~$\bbC$-Hodge structures in the sense of~\cite[def.~4.1.9]{MHMProject}; by theorem~14.6.1 in loc.~cit.~any such polarizable variation of~$\bbC$-Hodge structures can be viewed as an object of~$\HM(X, w + \dim X)$. More generally, for any integral subvariety~$Z\subset X$ and any pure polarizable variation of~$\bbC$-Hodge structures~$\bbV$ of weight~$w$ on an open dense~$U\subset Z^\reg$ there is a unique simple Hodge module 
\[
 M \;=\; \IC_Z(\bbV) \;\in\; \HM(X, w+\dim Z)
\]
with~$M_{\vert U} = \bbV$. Conversely every simple Hodge module~$M\in \HM(X)$ arises like this for a unique germ of a simple pure polarizable variation of~$\bbC$-Hodge structures on a dense open subset of an integral subvariety of~$X$ \cite[th.~16.2.1]{MHMProject}. 

\begin{example} 
For any integral subvariety~$Z\subset X$, the trivial variation~$\bbC_U$ of Hodge structures of rank one and weight zero on the smooth locus~$U=Z^\reg$ defines a pure Hodge module 
\[
 \delta_Z^H \;=\; \IC_Z(\bbC_U)
 \;\in\; \HM(X, d).
\]
of weight~$d=\dim Z$. Its image under the functor~$\DR$ is the perverse intersection complex 
\[ \delta_Z \;:=\; \DR(\delta_Z^H) \in \Perv(X). \]
If~$Z$ is smooth, then this intersection complex is given by~$\delta_Z = \bbC_Z[d]$.
\end{example}

\subsection{Hodge structure on cohomology} \label{sec:hodge-structure-on-cohomology}
Passing from perverse sheaves to Hodge modules will enrich the cohomology groups with~$\bbC$-Hodge structures. Indeed, with the notation of~\cite[def.~12.7.28]{MHMProject}, we have for any proper morphism~$f\colon X\to Y$ direct image functors
\[
 \tfstar^{(k)} \colon \quad \HM(X) \;\longrightarrow\; \HM(Y),
\]
for~$k\in \bbZ$. These lift the perverse direct image functors~${}^p \cH^k\circ Rf_*$ in the sense that the following diagram commutes:
\[
\begin{tikzcd}[column sep = 50pt]
\HM(X) \ar[r, "\tfstar^{(k)}"] \ar[d, swap, "\DR"]
& \HM(Y) \ar[d] \ar[d, "\DR"]
\\
\Perv(X) \ar[r, "{}^p \cH^k\circ \rR f_*"] 
& \Perv(Y)
\end{tikzcd} 
\]
Taking~$f\colon X\to Y=\{\mathrm{pt}\}$, we see that for any~$M\in \HM(X)$ the groups~$\rH^k(X, \DR(M))$ come with a natural bigrading. We call this the \emph{Hodge decomposition} and denote it by
\[
 \rH^k(X, \DR(M)) \;=\; \bigoplus_{p,q\in \bbZ} \rH^{p,q}(M).
\]
Note that this is a pure~$\bbC$-Hodge structure of weight~$w+k$ if~$M\in \HM(X, w)$.

\subsection{Tannaka categories} \label{sec:tannaka}
Now let~$X=A$ be an abelian variety. In this case every perverse sheaf has nonnegative Euler characteristic by~\cite[cor.~1.4]{FraneckiKapranov}, so inside the abelian category of perverse sheaves, those of Euler characteristic zero form a Serre subcategory 
\[
 \rS_{\Perv}(A) \;:=\; \{ P \in \Perv(A) \mid \chi(A, P)=0\} \;\subset\; \Perv(A).
\] 
The objects of this subcategory are called \emph{negligible} perverse sheaves. On the abelian quotient category 
$\Pbar(A)=\Perv(A)/\rS_{\Perv}(A)$ the group law~$\sigma \colon A\times A \to A$ gives rise to  a convolution product
\[
*\colon \quad \Pbar(A) \times \Pbar(A) \;\longrightarrow\; \Pbar(A),
\quad P_1 * P_2 \;:=\; {}^p \cH^0(R \sigma_*(P_1\boxtimes P_2))
\]
making~$\Pbar(A)$ a neutral Tannaka category as recalled in~\cite[sect.~3.1]{JKLM}.~For any fiber functor $\omega\colon \Pbar(A) \to \Vect(\bbC)$, the dimension of objects is the Euler characteristic
\begin{equation} \label{eq:dim}
 \dim_\bbC \,\omega(P) \;=\; \chi(A, P),
\end{equation}
see~\cite[proof of cor.~4.2]{KWVanishing}. The same constructions work also for the category of Hodge modules: Consider the Serre subcategory
\[
 \rS_{\HM}(A) \;:=\; \{ M\in \HM(A) \mid \chi(A, \DR(M)) = 0 \} \;\subset\; \HM(A).
\]
Then the functor~$\DR$ descends to an exact faithful~$\bbC$-linear functor on the abelian quotient category
\[
 \DR\colon \quad \Mbar(A) := \HM(A)/\rS_{\HM}(A) \;\longrightarrow\; \Pbar(A),
\]
and this functor naturally underlies a tensor functor with respect to the convolution product 
\[
 *\colon \quad \Mbar(A) \times \Mbar(A) \;\longrightarrow\; \Mbar(A), 
 \quad M_1 * M_2 \;:=\; \tf{\sigma}_*^{(0)}(M_1\boxtimes M_2).
\] 
So for~$M_1, M_2\in \Mbar(A)$ we have isomorphisms
\[
 \DR(M_1 \ast M_2) \;\stackrel{\sim}{\longrightarrow}\; \DR(M_1)\ast \DR(M_2)
\]
compatible with the associativity, commutativity and unit of the respective tensor categories; note that to avoid sign issues in the commutativity, we will always use right~$\cD$-modules for the definition of~$\HM(A)$. 

\begin{corollary} 
$\Mbar(A)$ is again a neutral Tannaka category.
\end{corollary} 

\begin{proof}
Pick any fiber functor on the neutral Tannaka category~$\Pbar(A)$. Composing it with the faithful exact~$\bbC$-linear functor~$\DR$, which is a tensor functor by the above, we get a fiber functor on~$\Mbar(A)$.
\end{proof}

\subsection{An exact sequence of Tannaka groups}  We want to compare the Tannaka groups for Hodge modules with those for the underlying perverse sheaves. For the rest of this section, we fix a full abelian tensor subcategory~$\cC \subset \Mbar(A)$ and denote by 
\[ \DR(\cC) \subset \Pbar(A) \]
the smallest full abelian tensor subcategory containing all subquotients of~$\DR(M)$ for~$M\in \cC$. Suppose that on this subcategory we are given a fiber functor, i.e.~a faithful exact~$\bbC$-linear tensor functor
\[
 \omega \colon \quad \DR(\cC) \;\longrightarrow\; \Vect(\bbC).
\]
Precomposing with the exact faithful~$\bbC$-linear tensor functor~$\DR\colon \cC \to \DR(\cC)$ we get a fiber functor~$\omega \circ \DR$ on the original category:
\[
\begin{tikzcd}
\cC \ar[rr, "\omega\circ \DR"] \ar[dr, swap, "\DR"] && \Vect(\bbC) \\
& \DR(\cC) \ar[ur, swap, "\omega"] & 
\end{tikzcd} 
\] 
Let~$G_\omega(\cC)=\Aut^\otimes(\omega \circ \DR)$ and~$G_\omega(\DR(\cC))=\Aut^\otimes(\omega)$ be the corresponding Tannaka groups. It follows from the definition of the category~$\DR(\cC)$ that~$\DR$ induces an embedding
\[
 G_\omega(\DR(\cC)) \;\intoo\; G_\omega(\cC)
\]
as a closed subgroup~\cite[prop.~2.21b]{DM82}. In fact this subgroup is normal: To understand this, let
\[
 \cC_{\{0\}} \;:=\; \{  M \in \cC \mid \Supp(M)=\{0\}\} \;\subset\; \cC
\]
be the full subcategory of objects in~$\cC$ supported at the origin; it can be seen as a tensor subcategory of the category of bigraded vector spaces via~\cref{hodge-modules-on-a-point}, hence its Tannaka group
\[ G_\omega(\cC_{\{0\}}) \;:=\; \Aut^\otimes(\omega\circ \DR \vert \cC_{\{0\}}) 
\]
is a quotient of a torus of rank two; in particular, it is a torus.  We then have the following Hodge theoretic analog of the Galois sequence in~\cite[th.~4.3]{JKLM}:

\begin{proposition} \label{perverse-versus-hodge-sequence}
We have an exact sequence
\[
 1 \too G_\omega(\DR(\cC)) \too G_\omega(\cC) \too G_\omega(\cC_{\{0\}}) \too 1.
\]
\end{proposition}

\begin{proof} 
We have an embedding~$\cC_{\{0\}}\hookrightarrow \cC$ as a full abelian tensor subcategory which is stable under subobjects. By Tannaka duality, any such embedding corresponds to an epimorphism~$G_\omega(\cC)\twoheadrightarrow G_\omega(\cC_{\{0\}})$~\cite[prop.~2.21a]{DM82}. The exactness of the sequence of Tannaka groups then follows by the same argument as in the proof of~\cite[th.~4.3]{JKLM}. Indeed, by~\cite[prop.~A.13 and lemma A.4(1)]{DAE20} we only need to show that for any~$M\in \cC$ the following two properties hold:\smallskip
\begin{enumerate} 
\item Every rank one subobject of~$\DR(M)$ is a direct summand in a semisimple object $\DR(N)$ with $N\in \cC$.\smallskip
\item The maximal trivial subobject of~$\DR(M)$ lies in the essential image of the functor~$\DR\colon \langle \cC \rangle \to \DR(\cC)$.\smallskip
\end{enumerate} 
Both properties are trivial in this case: By~\cite[sect.~10]{WeissauerAlmostConnected} or~\cite[lemma~3.5]{KraemerMicrolocalI}, the rank one objects of the tensor category $\cC$ are precisely those whose underlying perverse sheaf is a skyscraper sheaf of rank one,
and the maximal trivial subobject is the direct sum of all skyscraper subsheaves supported at the origin. Hence both claims follow from the fact that for any~$M\in \HM(A)$, every simple subobject of the perverse sheaf~$\DR(M)$ (or equivalently, of the underlying holonomic~$\cD$-module) has the property that a direct sum of twists of that subobject underlies a Hodge submodule of~$M$ \cite[prop.~16.3.1]{MHMProject}.
\end{proof} 

For~$M\in \cC$ with underlying perverse sheaf~$P=\DR(M)$, let~$V=\omega(P)$. Then the groups
\begin{eqnarray*}
 G_\omega(M) &:=& \im(G_\omega(\cC) \to \GL(V)) \\
  G_\omega(P) &:=& \im(G_\omega(\DR(\cC)) \to \GL(V))  
\end{eqnarray*}
can be identified with the Tannaka groups of the full tensor subcategories~$\langle M \rangle \subset \cC$ and~$\langle P \rangle \subset \DR(\cC)$ generated by~$M$ and~$P$ respectively. These Tannaka groups are algebraic groups, by construction they are subgroups of~$\GL(V)$. They are  reductive since all our Hodge modules and perverse sheaves are semisimple. So the derived group of their connected component of the identity is a connected semisimple group which we denote by 
\[ 
G^*_\omega(-) \;:=\; [G^\circ_\omega(-), G^\circ_\omega(-)].
\] 
The groups for Hodge modules and for perverse sheaves are then related as follows:

\begin{corollary} \label{cor:perverse-versus-hodge}
For~$M\in \cC$ consider the perverse sheaf~$P=\DR(M)$. Let~$V=\omega(P)$ as above. Then 
\[
 G_\omega(P) \;\subset\; G_\omega(M) \;\subset\; N_{\GL(V)}(G_\omega(P))
 \quad \text{and} \quad 
 G_\omega^*(P) \;=\; G_\omega^*(M). 
\] 
\end{corollary} 

\begin{proof}
Applying~\cref{perverse-versus-hodge-sequence} to the tensor category~$\cC = \langle M \rangle$ generated by~$M$, we get an exact sequence
\[
 1 \too G_\omega(P) \too G_\omega(M) \too G_\omega(\cC_{\{0\}}) \too 1.
\]
So~$G_\omega(P)\unlhd G_\omega(M)$ is a normal subgroup. Since both are subgroups of~$\GL(V)$, this implies the claim about the containment in the normalizer. For the claim about the derived groups of the connected component, recall that~$G_\omega(\cC_{\{0\}})$ is a torus; so the reductive Lie algebras~$\Lie G_\omega(M)$ and~$\Lie G_\omega(P)$ have the same maximal semisimple Lie subalgebra, hence $G_\omega(M)^\circ$ and $G_\omega(P)^\circ$ share the same derived subgroup.
\end{proof}

\subsection{Character twists}

Any unitary local system~$L$ of rank one on~$A$ naturally underlies a unique pure polarizable variation~$\bbL$ of Hodge structures of weight zero, concentrated in bidegree~$(0,0)$. We then have an endofunctor
\[
 \HM(A) \;\longrightarrow\; \HM(A), \quad M \;\longmapsto\; M\otimes \bbL
\]
which descends to the quotient categories modulo negligibles, where~$M\in \HM(A)$ is called \emph{negligible} if~$\DR(M)$ is a negligible perverse sheaf in the sense of~\cref{sec:tannaka}. The induced functor
\[
 \Mbar(A) \;\longrightarrow\; \Mbar(A)
\]
is a tensor functor with respect to the convolution product: For~$M_1, M_2\in \Mbar(A)$ we have natural isomorphisms
\begin{align*} 
 (M_1*M_2)\otimes \bbL 
 &\;\simeq\; \tf{\sigma}_*^{(0)}(M_1 \boxtimes M_2) \otimes \bbL 
 & 
 \\[0.2em]
 &\;\simeq\; \tf{\sigma}_*^{(0)} ((M_1\boxtimes M_2) \otimes \sigma^{-1}(\bbL))
 & \text{by the projection formula}
 \\[0.2em]
 &\;\simeq\; \tf{\sigma}_*^{(0)} ((M_1\otimes \bbL) \boxtimes (M_2\otimes \bbL))
 & \text{since~$\sigma^{-1}(\bbL)\simeq \bbL \boxtimes \bbL$}
 \\[0.2em]
 &\;\simeq\; (M_1\otimes \bbL) * (M_2\otimes \bbL).
\end{align*}
Let~$\cC$ and~$\omega$ be as in the previous section. Then up to noncanonical isomorphism, the Tannaka group of a Hodge module in~$\cC$ does not change under twists:

\begin{lemma} \label{lem:tannakagroup-of-twist}
	For~$M\in \cC$ and~$\bbL$ a unitary rank one local system with~$M\otimes \bbL\in \cC$, we have
	\[
	G_\omega(M\otimes \bbL) \;\simeq \; G_\omega(M).
	\] 
\end{lemma} 

\begin{proof} 
	As in~\cite[prop.~4.1]{KWVanishing}, twisting by~$\bbL$ gives rise to an equivalence of tensor categories
	\[
	\varphi\colon \quad \langle M \rangle \;\stackrel{\sim}{\too} \; \langle M\otimes \bbL \rangle, \quad N \;\longmapsto\; N\otimes \bbL.
	\]
Both the source and the target of this equivalence are tensor subcategories of $\cC$, hence~$\omega$ restricts to a fiber functor on both of them. The equivalence $\varphi$ need not be compatible with these fiber functors, but over an algebraically closed field of characteristic zero any two fiber functors are noncanonically isomorphic; hence the same holds for the corresponding Tannaka groups. 
\end{proof} 

\section{The Hodge cocharacter}

We now show that the Hodge decomposition on cohomology is induced by a cocharacter of the Tannaka group, and we gather some general estimates for the Hodge level of intersection complexes that will be useful later.

\subsection{The Hodge decomposition} \label{subsec:hodge-decomposition}

Let~$\Perv_0(A) \subset \Perv(A)$ be the full subcategory of all perverse sheaves~$P$ with the property that all the perverse subquotients~$Q$ of~$P$ satisfy 
\[
 \rH^i(A, Q) \;=\; 0 \quad \text{for all} \quad i\;\neq\; 0.
\]
As in~\cite[sect.~4.3]{JKLM}, its image
\[
\Pbar_0(A) \;\subset\; \Pbar(A)
\]
is a full abelian tensor subcategory which is equivalent to~$\Perv_0(A)/S_0(A)$, where $\rS_0(A) \subset \Perv_0(A)$ is the full subcategory  of perverse sheaves~$P$ with the property that all the subquotients~$Q$ of~$P$ satisfy~$H^\bullet(A, Q)=0$. We then get an exact functor
\[
 \omega_0\colon \quad \Pbar_0(A) \;=\; \Perv_0(A)/S_0(A)\;\too\; \Vect(\bbC), \quad Q \;\longmapsto\; H^0(A, Q).
\]
This is a tensor functor. In particular, for all perverse sheaves~$P_1, P_2\in \Pbar_0(A)$ we have natural isomorphisms
\[
 \vartheta_{P_1, P_2}\colon \quad \omega_0(P_1*P_2) \;\stackrel{\sim}{\longrightarrow}\;  \omega_0(P_1) \otimes \omega_0(P_2)
\] 
which are defined as follows: Let~$f\colon A \to \{\pt \}$ be the map to a point, then
\[ \omega_0(P) \;=\; \cH^0(\rR f_*(P)) \quad \text{for all} \quad P\in \Perv_0(A). \]
Now, the complex~$\rR f_\ast (P)$ has cohomology at most in degree~$0$ so we identify it with~$\omega(P)$. With this in mind, using the K\"unneth isomorphism and the fact that $f\circ \sigma = f\times f$, we obtain
\begin{align*} 
 \omega_0(P_1 * P_2) 
 &\;\iso\; Rf_* (P_1 * P_2) 
 \\
 &\;\simeq\; Rf_* R\sigma_* (P_1\boxtimes P_2) 
 \\
 &\;\simeq\; R(f\times f)_* (P_1\boxtimes P_2)
 \\
 &\;\simeq\; Rf_*(P_1)\otimes Rf_*(P_2)
 \\
 &\;\iso\; \omega_0(P_1) \otimes \omega_0(P_2). 
\end{align*} 
These constructions can be upgraded to Hodge modules: Let~$\HM_0(A)\subset \HM(A)$ be the full subcategory of all Hodge modules~$M$ with~$\DR(M)\in \Perv_0(A)$, and denote by $\Mbar_0(A)\subset \Mbar(A)$ its image in the Tannaka category from~\cref{sec:tannaka}. Then~$\omega_0$ induces the functor
\[
 \omega_0^H\colon \quad \Mbar_0(A) \;\longrightarrow\; \Vect_{\bbZ \times \bbZ}(\bbC), \quad M\;\longmapsto\; \rH^0(A, \DR(M))
\]
where the bigrading on the values of the functor is the Hodge decomposition from~\cref{sec:hodge-structure-on-cohomology}:
\[
\rH^0(A, \DR(M)) \;=\; \bigoplus_{p,q\in \bbZ} \rH^{p,q}(M).
\]

\begin{proposition} \label{prop:HodgeTensorFunctor}
For~$M_1, M_2\in \Mbar_0(A)$, the isomorphism~$\vartheta_{\DR(M_1), \DR(M_2)}$ is compatible with the Hodge decomposition. Hence the functor
$\omega_0^H$ 
is a tensor functor.
\end{proposition}

\begin{proof} 
The construction is the same as for perverse sheaves, but working on the level of Hodge modules. For a morphism~$g$ we denote by~$\tf{g}_*$ the pushforward functor of~\cite[def.~12.7.35]{MHMProject}, which replaces the derived pushforward~$Rg_*$ of constructible sheaves.
As above, let~$f\colon A \to \{\pt\}$ be the map to a point, then
\[ \omega_0^H(M) \;=\; \tf{f}^{(0)}_* (M)  
\quad \text{ for all~$M\in \Mbar_0(A)$.}
\]
As above, the complex of bigraded vector space~$\tf{f}_* (M)$ has cohomology at most in degree~$0$ so we identify it with~$\omega_0^H(M)$. 
Let~$\sigma\colon A\times A \to A$ be the group law. Then we have~$f\circ \sigma = f\times f$, hence the Leray spectral sequence~\cite[cor.~12.7.38]{MHMProject} gives
\[
 \tf{f}_* \circ \tf{\sigma}_* = \tf{(f\times f)}_*.
\]
Thus in~$\HM(\{\pt\})=\Vect_{\bbZ\times \bbZ}(\bbC)$ we obtain
\begin{align*} 
 \omega_0^H(M_1 * M_2) 
 &\;\iso\; \tf{f}_* (M_1 * M_2)
 \\[0.2em]
 &\;\simeq\;   \tf{f}_* \tf{\sigma}_* (M_1 \boxtimes M_2) 
 \\[0.2em]
 &\;\simeq\;  \tf{(f\times f)}_* (M_1\boxtimes M_2) 
 \\[0.2em]
 &\;\simeq\; \tf{f}_*(M_1) \otimes \tf{f}_*(M_2)  
 \\[0.2em]
 &\;\iso\; \omega_0^H(M_1) \otimes \omega_0^H(M_2),
\end{align*} 
where in the third step we have used the K\"unneth isomorphism.
\end{proof}

The above can be applied to any Hodge module~$M\in \Mbar(A)$ after twisting by a suitable rank one local system. More precisely, by generic vanishing on abelian varieties~\cite{KWVanishing, SchnellHolonomic} we can find a unitary local system~$L$ of rank one on~$A$ such that 
\[
 M\otimes \bbL \;\in\; \Mbar_0(A).
\]
Fixing $M$ and $\bbL$, we obtain a tensor functor
\[
 \omega^H\colon \quad \langle M\rangle \;\too\; \Vect_{\bbZ\times \bbZ}(\bbC), \quad N \;\longmapsto\; \rH^0(A, N\otimes \bbL)
\]
by precomposing $\omega_0^H$ with the tensor functor $N\mapsto N\otimes \bbL$, see~\cite[prop.~4.1]{KWVanishing}.

\begin{corollary} \label{cor:hodge-cocharacter}
There is a morphism~$\lambda \colon \bbG_m^2 \to G_\omega(M)$ of algebraic groups such that the Hodge decomposition
\[
 \omega^H(M) \;=\; \bigoplus_{(p,q)\in \bbZ^2} \rH^{p,q}(M\otimes \bbL)
\]
is the decomposition in weight spaces for the characters~$(p,q)\in \bbZ^2 = \Hom(\bbG_m^2, \bbG_m)$.
\end{corollary}

In classical Hodge theory the bigrading is given by a representation of~$\bbS_\bbC = \bbG_m^2$ for~$\bbS = \mathrm{Res}_{\bbC/\bbR}(\bbG_m)$. By analogy, we give the following definition:

\begin{definition} The morphism~$\lambda\colon \bbG_m^2 \to G_\omega(M)$ in the above corollary is called the \emph{Hodge cocharacter}.
\end{definition}

The fact that~$\lambda$ factors over the Tannaka group will allow us to find information about Tannaka groups from information about Hodge numbers.

\subsection{Hodge level of intersection complexes} \label{sec:hodge-level-of-IC} Any object of~$\HM(A)$ is a direct sum of simple objects, and any simple Hodge module is the intersection complex~$\IC_Z(\bbV)$ of a variation of Hodge structures~$\bbV$ on an open dense subset of an integral subvariety~$Z\subset A$. We are interested in the twisted Hodge numbers
\[
 h^p(M) \;:=\; h^{p, w-p}(M) \;:=\; \dim \rH^{p, w-p}(M\otimes \bbL)
\]
where the twist is by a unitary local system~$\rL$ of rank one with~$\rH^i(A, \DR(M)\otimes L)=0$ for all~$i\neq 0$; we drop the local system from the notation since it does not affect the Hodge numbers.
If~$\bbV = \bbC_U$ is the trivial variation of Hodge structures on the smooth locus~$U\subset X$ of an integral subvariety~$X\subset A$ of dimension $d$, we also write 
\[
\delta_X^H \;:=\; \IC_X(\bbC_U) \;\in\; \HM(A, d).
\]
The underlying perverse sheaf is the perverse intersection complex~$\delta_X := \DR(\delta_X^H)$ of the subvariety, and we write
\[
 h^{p}(X) \;:=\; h^p(\delta_X^H).
\]
For smooth~$X$ these twisted Hodge numbers have the following properties:

\begin{lemma} \label{lem:hodge-estimate}
If~$X \subset A$ is a smooth subvariety of dimension~$d$, then for all~$p\in \bbZ$ we have 
\[ h^p(X) \;=\; (-1)^{d-p} \chi(X, \Omega^p_X) \;=\; h^{d-p}(X). \]
If moreover~$X\subset A$ has ample normal bundle and~$d<g=\dim A$, then these Hodge numbers satisfy
\[ h^p(X) \;\ge\; 
\begin{cases} 
g-d + 1 & \text{for~$p\in \{0,d\}$}, \\[0.1em] 
2 & \text{for~$p\in \{1,d-1\}$}, \\[0.1em] 
1 & \text{for~$2 \le p \le d-2$.}
\end{cases} 
\]
\end{lemma} 

\begin{proof} 
If~$X$ is smooth, then~$h^p(X) = (-1)^{d-p}\chi(X, \Omega^p_X)$ by~\cite[appendix B]{KM}, thus~$h^p(X)=h^{d-p}(X)$ by Serre duality. 
For the inequalities see~\cref{cor:hodge-number-estimates}.
\end{proof} 

\begin{corollary} \label{cor:dim-via-hodge} For any smooth subvariety~$X\subset A$ with nonzero ample normal bundle, its dimension is equal to the level of the Hodge decomposition:
\[
 \dim X \;=\; \max \{ |2p-d| \colon h^p(X) \neq 0 \}.
\]
\end{corollary}

For perverse intersection complexes of singular subvarieties~$X\subset A$ we still have an upper bound. More generally, for intersection complexes of variations~$\bbV$ of Hodge structures, denote by
\[
 \ell(\bbV) \;:=\; \max \{ |p-q| \colon \bbV^{p,q} \neq 0 \}
\]
the \emph{level} of the Hodge structure. We then obtain:

\begin{lemma} \label{Hodge-level-VHS}
Let~$X\subset A$ be an integral subvariety of dimension~$d$. Then for any pure polarizable variation of~$\bbC$-Hodge structures on an open dense subset~$U\subset X^\reg$ we have
\[
 h^{p,q}(\IC_X(\bbV)) \;=\; 0 \quad \text{for} \quad |p-q| \;>\; \ell(\bbV) + d,
\]
\end{lemma} 

\begin{proof} 
By definition~$h^{p,q}(\IC_X(\bbV))=\dim \rH^{p,q}(\IC_X(\bbV)\otimes \bbL)$, where~$\bbL$ is a unitary local system of rank one. Since
\[
 \IC_X(\bbV) \otimes \bbL \;=\; \IC_X(\bbV \otimes \bbL),
\]
we may replace~$\bbV$ by the variation of~$\bbC$-Hodge structures~$\bbV \otimes \bbL$ which has the same level. Hence in what follows we will assume that~$\bbL$ is the trivial local system.\medskip

Now~$\bbV$ embeds into a variation of pure polarizable~$\bbR$-Hodge structures of the same level. To see this, write~$\cV$ for the flat~$\cC^\infty$-bundle underlying~$\bbV$. Define~$\bar{\bbV}$ to be the variation of~$\bbC$-Hodge structures whose underlying flat~$\cC^\infty$-bundle is the complex conjugate of~$\cV$ and whose component of bidegree~$(p,q)$ is the complex conjugate of~$\cV^{q,p}$.  This defines a variation of~$\bbC$-Hodge structures by~\cite[sect.~4.1.a]{MHMProject}. Note that a polarization on~$\bbV$ induces a natural polarization on~$\bar{\bbV}$. Therefore~$\bbW = \bbV \oplus \bar{\bbV}$ is the complexification of the  sought-for polarizable variation of~$\bbR$-Hodge structures, see also \cite[rem. 4.1.12]{MHMProject}.
Replacing~$\bbV$ by~$\bbW$ we assume in what follows that~$\bbV$ comes from a polarizable variation of~$\bbR$-Hodge structures.
\medskip 

Now let~$k = \max\{ p \in \bbZ \mid F^p \bbV \neq 0 \}$, so that~$F^k \bbV\subset \bbV$ is the smallest piece in the Hodge filtration. If the variation of Hodge structures has weight~$w$, then by Hodge symmetry the Hodge numbers of its fibers can be nonzero at most in bidegrees $(i, w-i)$ with~$w-k \le i \le k$. So the level of the Hodge structure on the fibers is 
\[
 \ell(\bbV) \;=\; 2k - w.
\]
Now by the conventions of~\cite[appendix~A.7]{MHMProject}, the right~$\cD_U$-module~$\cM_U :=\omega_U\otimes \bbV$ associated to the variation of Hodge structures comes with the increasing Hodge filtration 
\[ 
F_p \cM_U \;:=\; \omega_U \otimes F^{-p-d} \bbV 
\]
which starts in degree~$p=-k-d$. By the theory of Hodge modules this filtration naturally extends to an increasing filtration on the right~$\cD_A$-module~$\cM$ underlying the Hodge module~$M=\IC_X(\bbV)$. Moreover, the lowest piece~$F_p \cM$ of this filtration is still in degree~$p=-k-d$, see for instance \cite[sect.~2.4]{SchnellYang} in the context of Saito's Hodge modules. Now consider the de Rham complex
\[
 \DR(\cM) 
 \;=\;
 \bigl[
 \cdots 
 \to \cM \otimes \wedge^2 \cT_A 
 \to \cM \otimes \cT_A 
 \to \cM
 \bigr]
\]
where the last term~$\cM$ is placed in degree zero. It comes with the increasing filtration
\[
 F_p\DR(\cM) 
 \;=\;
 \bigl[
 \cdots 
 \to F_{p-2}(\cM) \otimes \wedge^2 \cT_A 
 \to F_{p-1}(\cM) \otimes \cT_A 
 \to F_p (\cM)
 \bigr]
\]
which by the above starts in degree~$p=-k-d$. Hence on de Rham cohomology the filtration
\[
 F_p \rH^0(A, \DR(M)) \;=\; \im \bigl( \rH^0(A, F_p \DR(M)) \to \rH^0(A, \DR(M) \bigr)
\]
can only start in degree~$p=-k-d$ or higher. Since~$M$ is a pure Hodge module of weight~$w+d$ by~\cite[th.~14.6.1]{MHMProject}, the de Rham cohomology~$\rH^0(A, \DR(M))$ in degree zero is a pure Hodge structure of weight~$w+d$ by th.~14.3.2(2) in loc.~cit. Its classical decreasing Hodge filtration goes at most up to degree $-p = k+d$ by the above. So again by Hodge symmetry its level is
\[
 \ell(\rH^0(A, \DR(M)) \;\le \; 2(k+d) - (w+d) \;=\; (2k-w) + d \;=\; \ell(\bbV) + d,
\]
which concludes the proof.
\end{proof} 

The above then also gives an estimate on the level of the Hodge structures arising from the decomposition theorem applied to intersection complexes:

\begin{lemma} \label{Hodge-level-of-direct-image} 
Let~$Y$ be a smooth variety with a proper morphism~$f\colon Y\to A$. Then for every direct summand~$M\subset \tfstar^{(0)}(\delta_Y^H)$, its twisted Hodge numbers satisfy
\[
 h^{p,q}(M) \;=\; 0 \quad \text{for} \quad |p-q| \;>\; \dim f^{-1}(\Supp M).
\]
\end{lemma}

\begin{proof} 
Recall that by~$S$-decomposability of Hodge modules~\cite[th.~14.2.19(1)]{MHMProject}, we have  
\[ \tfstar^{(0)}(\delta_Y^H) \;=\; \bigoplus_{X\subset A} M_X \]
where~$X\subset A$ runs through all integral subvarieties and~$M_X \subset \tfstar^{(0)}(\delta_Y^H)$ denotes the maximal Hodge submodule with strict support~$X$. It thus suffices to prove the lemma when~$M=M_X$ for some~$X\subset A$. Note that the above strict support decomposition is compatible with the natural functor from Saito's $\bbQ$-Hodge modules to complex Hodge modules, since under the faithful de Rham functor it corresponds to the strict support decomposition of the underlying perverse sheaves. Thus, the summand~$M\subset \tfstar^{(0)}(\delta_Y^H)$ arises from a direct summand in the category of~$\bbQ$-Hodge modules; this will allow us to embed its generic stalk cohomology in the mixed Hodge structure of a singular variety even though the mixed case is not yet written for complex Hodge modules~\cite{MHMProject}. 

\medskip

Since~$M$ has strict support~$X$, we have~$M=\IC_X(\bbV)$ for a polarizable variation~$\bbV$ of~$\bbC$-Hodge structures on an open dense subset of~$X$. Let~$d=\dim X$, then at a general point~$x\in X$ the fiber of the local system~$V$ underlying the variation~$\bbV$ is a direct summand 
\[
 V_x \;=\; \cH^{-d}(\DR(M))_x \;\hookrightarrow\; \cH^{-d}(Rf_*(\delta_Y))_x \;=\; \rH^{\dim Y - d}(f^{-1}(x)).
\]
By the discussion in the first paragraph, this inclusion is induced by an inclusion of mixed~$\bbQ$-Hodge structures. The fibers~$f^{-1}(x)$ may be singular, but since for any singular variety the mixed Hodge structure on its cohomology has level at most the dimension of the variety, our variation of Hodge structures has level 
\[
 \ell(\bbV) \;\le\; \ell(\rH^\bullet(f^{-1}(x)) \;\le\; \dim f^{-1}(x) \;=\;  \dim f^{-1}(X) - d,
\]
where the first inequality comes from the inclusion $V_x \hookrightarrow \rH^{\bullet}(f^{-1}(x))$ of Hodge structures and the last equality uses that~$x\in X$ is general. Since~$M=\IC_X(\bbV)$, \cref{Hodge-level-VHS} gives~$h^{p,q}(M) = 0$ for~$|p-q|> (\dim f^{-1}(X) - d) + d = \dim f^{-1}(X)$.
\end{proof}

\section{Subvarieties with exceptional Tannaka groups}

We now discuss the restrictions which the above results impose on subvarieties with Tannaka group~$E_6$ and~$E_7$. As above we work over the complex numbers.

\subsection{Reduction to semisimple groups} For Hodge modules of nonzero weight the Tannaka group is not semisimple: The weight is given by a central cocharacter of this group. But for fixed weight we can still describe the Hodge decomposition by a cocharacter with values in a connected semisimple group as follows. To fix notation, suppose we are given\smallskip
\begin{itemize} 
\item a finite-dimensional vector space~$V$ over~$\bbC$,\smallskip
\item a connected semisimple subgroup~$G\subset \GL(V)$,\smallskip
\item a cocharacter~$\lambda\colon \bbG_m^2 \to Z\cdot G$ where~$Z=\bbG_m\cdot \id_V \subset \GL(V)$,\smallskip 
\end{itemize} 
such that the weight spaces 
\[ V^{p,q} := \{ v\in V \mid \lambda(x,y)(v)=x^py^q \cdot v \;\; \text{for all} \;\; (x,y) \in \bbG_m^2 \} 
\] 
satisfy Hodge symmetry in the sense that~$\dim V^{p,q} = \dim V^{q,p}$ for all~$p,q$. The Hodge symmetry implies~$\lambda(z,z^{-1}) \in \SL(V)$ for all~$z\in \bbG_m$. On the other hand we have 
\[ G \;=\; (Z\cdot G) \cap \SL(V) \]
since~$G$ is a connected semisimple group. So~$\lambda$ restricts on the antidiagonal to a cocharacter 
\[
 \rho\colon \quad \bbG_m \;\longrightarrow\; G, \quad z \;\longmapsto\; \lambda(z, z^{-1})
\]
with values in the connected semisimple subgroup~$G\subset \GL(V)$. For~$v\in V^{p,q}$ we have
\[
 \rho(z)(v) \;=\; z^{p-q} \cdot v
 \;=\; z^{2p-w} \cdot v
 \quad \text{for the weight} \quad 
 w \;:=\; p+q.
\]
We will apply this in the following situation:

\begin{example} 
Let~$M\in \HM(A,w)$ be a simple Hodge module on a complex abelian variety~$A$. Assume that the Hodge decomposition
\[
 V \;:=\; \rH^0(A, \DR(M\otimes \bbL)) \;=\; \bigoplus_{p+q=w} V^{p,q}
\]
from~\cref{subsec:hodge-decomposition} satisfies Hodge symmetry in the sense that~$\dim V^{p,q} = \dim V^{q,p}$ 
for all~$p,q\in \bbZ$. Recall that this decomposition is given by the weights for the Hodge cocharacter
\[
 \lambda \colon \quad \bbG_m^2 \;\longrightarrow\; G_\omega(M) \;\subset\; \GL(V)
 \quad 
 \text{with} 
 \quad 
 \lambda(z,w)_{\vert V^{p,q}} \;=\; z^p w^q \cdot \id_{V^{p,q}}.
\]
Since~$M$ is a simple Hodge module, the group~$G_\omega(M)$ acts on~$V$ by an irreducible representation. By Schur's lemma its center then acts on~$V$ by scalars. Since any connected reductive group is the almost direct product of its center and its derived group, it follows that
\[
 G_\omega^\circ(M) \;\subset\; Z \cdot G
 \quad \text{for} \quad Z \;:=\; \bbG_m \cdot \id_V\;\subset\; \GL(V).
\]
So we are in the situation described above. 
\end{example}

\subsection{Cocharacters with Hodge properties}  

Let~$G$ be a complex semisimple algebraic group, and let~$V$ be a finite-dimensional representation of~$G$ with finite kernel. Any cocharacter~$\lambda\colon \bbG_m \rightarrow G$
induces a weight space decomposition
\[
 V \;=\; \bigoplus_{n\in \bbZ} V^n(\lambda)
\]
where~$\lambda$ acts on~$V^n(\lambda)$ via the character~$n\in \bbZ= \Hom(\bbG_m, \bbG_m)$. We want to determine all choices of the cocharacter~$\lambda$ such that the Hodge numbers~$\dim V^n(\lambda)$ have the following properties:\smallskip

\begin{enumerate}
\item[(H1)] Hodge symmetry:~$\dim V^n(\lambda) = \dim V^{-n}(\lambda)$ for all~$n\in \bbZ$.\smallskip
\item[(H2)] Gap freeness: There exists an integer~$\ell \ge 0$ such that 
\[
 V^n(\lambda) \;\neq\; 0 
 \quad \Longleftrightarrow \quad 
 n \;=\; 2i - \ell \; \text{with} \; i\;\in\; \{0,1,\dots, \ell\}.
\]
\item[(H3)] Outer Hodge number estimate: For~$\ell$ as above we have
\[ \dim V^{\ell-2}(\lambda) \ge 2 \quad \textup{and} \quad \dim V^\ell(\lambda) \ge 3.\]
\end{enumerate} 
The set of such cocharacters is a finite union of conjugacy classes:

\begin{lemma}
For any given~$(G, V)$, there exist up to~$G$-conjugacy only finitely many cocharacters~$\lambda$ that satisfy the property (H2). 
\end{lemma} 

\begin{proof} 
Fix a maximal torus~$T\subset G$. Any cocharacter is conjugate to one with values in this torus, so it suffices to consider cocharacters~$\lambda \in X_\ast(T) = \Hom(\bbG_m, T)$. The weight space decomposition for~$\lambda$ is then obtained as follows: Consider the weight space decomposition
\[
 V \;=\; \bigoplus_{\chi \in X^\ast(T)} V(\chi)
\]
for our maximal torus, where the action of the torus on the subspace~$V(\chi)\subset V$ is given by~$\chi \in X^\ast(T)=\Hom(T, \bbG)$. Then the weight~$n$ space for the cocharacter~$\lambda$ is the direct sum
\[
 V^n(\lambda) \;=\; \bigoplus_{\langle \lambda, \chi \rangle = n} V(\chi)
\]
where the sum is over all characters~$\chi \in X^\ast(T)$ with~$\langle \chi, \lambda \rangle = n$. To make use of this, let
\[
 S \;:=\; \{ \chi \in X^\ast(T) \mid V(\chi) \ne 0\}
\]
be the set of characters that enter in the given representation. Condition (H2) then implies
$|\langle \lambda, \chi \rangle| \le \ell < \dim V$
for all $\chi \in S$.
This condition leaves only finitely many choices of~$\lambda \in X_\ast(T)=\Hom(\bbG_m, T)$ in the cocharacter lattice, since in the dual character lattice the occurring weights~$\chi \in S$ span a subgroup
$
 \langle S \rangle_\bbZ \subset X^\ast(T)
$
of finite index: Indeed, if the subgroup spanned by those weights had smaller rank than~$X^\ast(T)$, then we could find~$\mu\in X_\ast(T)\smallsetminus \{0\}$ with~$\langle \mu, \chi\rangle = 0$ for all~$\chi \in S$. But then the image of
$
 \mu\colon \bbG_m \to G
$
would be a subgroup of positive dimension acting trivially on the representation~$V$ which is impossible since~$V$ has finite kernel.
\end{proof}

The above proof allows to enumerate all cocharacters~$\lambda$ with (H1), (H2), (H3) for a given~$(G, V)$ by an exhaustive computer search. Below, we will list the results for the exceptional groups~$E_6$ and~$E_7$ in their minimal representations.

\subsection{Cocharacters of \texorpdfstring{$G=E_6$}{G=E{6}}} \label{sec:FiltrationsE6} 

Consider the simply connected group~$G=E_6$. Fix a system of positive roots for a maximal torus~$T\subset G$. We label the fundamental dominant weights
$
 \varpi_1, \dots, \varpi_6 \in X^\ast(T)
$
as in~\cite[ch. VI, sect.~4.12]{Bou02}, and the dual basis by~$\varpi_1^\vee, \dots, \varpi_6^\vee \in X_\ast(T)$. An exhaustive computer search gives:

\begin{proposition}  \label{prop:E6}
Let~$V$ be a~$27$-dimensional irreducible representation of~$E_6$, then for any
$
 \lambda \in X_\ast(E_6)
$
with (H1), (H2), (H3), the numbers
$h^i(\lambda) := \dim V^{2i-\ell}(\lambda)$ for~$0\le i\le \ell$
are given by one of the rows of the following table: 
\[
\begin{array}{c|c|c|c|c|c|c}
h^0 & h^1 & h^2 & h^3 & h^4 & h^5 & h^6  \\
\hline
6 & 15 & 6 &&&&\\
3 & 6 & 9 & 6 & 3 && \\
3 & 3 & 3 & 9 & 3 & 3 & 3  \\
\end{array}
\]
Moreover, up to conjugacy the only cocharacter giving the Hodge numbers~$6,15,6$ in the first row of the table is the cocharacter~$\lambda = \varpi_1^\vee$.
\end{proposition}

\iffalse
\begin{table}[h!]
\[
\begin{array}{c|c|c|c|c|c|c|c|c|c|c}
h^0 & h^1 & h^2 & h^3 & h^4 & h^5 & h^6 & h^7 & h^8 & h^9 & h^{10} \\
\hline
\hline
6 & 15 & 6 &&&&&&&&\\
3 & 6 & 9 & 6 & 3 &&&&&& \\
2 & 3 & 6 & 5 & 6 & 3 & 2 &&&& \\
3 & 3 & 3 & 9 & 3 & 3 & 3 &&&& \\
2 & 2 & 3 & 4 & 5 & 4 & 3 & 2 & 2 && \\
2 & 2 & 2 & 1 & 4 & 5 & 4 & 1 & 2 & 2 & 2
\end{array}
\]
\bigskip
\caption{\label{table:E6} Cocharacters and Hodge numbers of~$E_6$}
\end{table}
\fi

\begin{remark} \label{rem:adjoint-hodge-numbers}
On the adjoint representation~$W=\Lie E_6$, the cocharacter~$\lambda=\varpi_1^\vee$ induces a grading of length~$\ell = 4$ with 
\[
 \dim W^{2i-\ell}(\lambda) \;=\; 
 \begin{cases} 
 1 & \text{for~$i=0,4$}, \\ 20 & \text{for~$i=1,3$}, \\ 36 & \text{for~$i=2$}.
 \end{cases} 
\]
This information will be useful in the proof of~\cref{Prop:DegreeOfDifferenceMorphism} below.
\end{remark} 

\subsection{Cocharacters of \texorpdfstring{$G=E_7$}{G=E{7}}}

Now take the simply connected group~$G=E_7$, then as above an exhaustive computer search gives:

\begin{proposition} \label{prop:E7}
Let~$V$ be a~$56$-dimensional irreducible representation of~$E_7$, then for any $\lambda \in X_\ast(E_7)$ with (H1), (H2), (H3) whose associated Hodge decomposition has odd length~$\ell$, the numbers
$h^i(\lambda) =\dim V^{2i-\ell}(\lambda)$
for~$0\le i\le \ell$ are given by one of the rows of the following table:
\[
\begin{array}{c|c|c|c|c|c|c|c|c|c|c|c|c|c|c|c}
 h^0 & h^1 & h^2 & h^3 & h^4 & h^5 & h^6 & h^7 & h^8 & h^9 & h^{10} & h^{11} & h^{12} & h^{13} & h^{14} & h^{15} \\
\hline
 7 & 21 & 21 & 7 &&&&&&&&&&&&\\
 6 & 7 & 15 & 15 & 7 & 6&&&&&&&&&&\\
 6 & 6 & 1 & 15 & 15 & 1 & 6 & 6 &&&&&&&& \\
 5 & 3 & 10 & 10 & 10 & 10 & 3 & 5&&&&&&&&\\
 5 & 2 & 6 & 5 & 10 & 10 & 5 & 6 & 2 & 5 &&&&&& \\
 4 & 3 & 3 & 12 & 6 & 6 & 12 & 3 & 3 & 4 &&&&&&\\
 5 & 2 & 5 & 1 & 5 & 10 & 10 & 5 & 1 & 5 & 2 & 5 &&&&\\
 4 & 2 & 3 & 5 & 8 & 6 & 6 & 8 & 5 & 3 & 2 & 4 &&& \\
 4 & 2 & 2 & 5 & 1 & 8 & 6 & 6 & 8 & 1 & 5 & 2 & 2 & 4 &\\
 4 & 2 & 2 & 4 & 1 & 1 & 8 & 6 & 6 & 8 & 1 & 1 & 4 & 2 & 2 & 4 \\
\end{array}
\]
\end{proposition}

\iffalse
\begin{table}[h!]
\[
\begin{array}{c|c|c|c|c|c|c|c|c|c|c|c|c|c|c|c} 
 h^0 & h^1 & h^2 & h^3 & h^4 & h^5 & h^6 & h^7 & h^8 & h^9 & h^{10} & h^{11} & h^{12} & h^{13} & h^{14} & h^{15} \\
\hline \hline
 7 & 21 & 21 & 7 &&&&&&&&&&&&\\
 6 & 7 & 15 & 15 & 7 & 6&&&&&&&&&&\\
 6 & 6 & 1 & 15 & 15 & 1 & 6 & 6 &&&&&&&& \\
 5 & 3 & 10 & 10 & 10 & 10 & 3 & 5&&&&&&&&\\
 5 & 2 & 6 & 5 & 10 & 10 & 5 & 6 & 2 & 5 &&&&&& \\
 4 & 3 & 3 & 12 & 6 & 6 & 12 & 3 & 3 & 4 &&&&&&\\
 5 & 2 & 5 & 1 & 5 & 10 & 10 & 5 & 1 & 5 & 2 & 5 &&&&\\
 4 & 2 & 3 & 5 & 8 & 6 & 6 & 8 & 5 & 3 & 2 & 4 &&& \\
 4 & 2 & 2 & 5 & 1 & 8 & 6 & 6 & 8 & 1 & 5 & 2 & 2 & 4 &\\
 4 & 2 & 2 & 4 & 1 & 1 & 8 & 6 & 6 & 8 & 1 & 1 & 4 & 2 & 2 & 4 \\
\end{array}
\]
\caption{\label{table:E7} Cocharacters and Hodge numbers of~$E_7$}
\end{table}
\fi

\subsection{Hodge numbers of exceptional subvarieties} 

Putting together the above, we obtain the following restriction on smooth subvarieties~$X\subset A$ of complex abelian varieties~$A$ such that the Tannaka group
\[
 G_{X,\omega}^\ast \;:=\; G_\omega^\ast(\delta_X)
\]
of their perverse intersection complex is a simple exceptional group:

\begin{proposition} \label{prop:hodge-numbers}
Let~$A$ be an abelian variety of dimension~$g$, and~$X\subset A$ a smooth nondivisible irreducible subvariety of dimension~$d$ with ample normal bundle such that~$G_{X, \omega}^\ast$ is a simple exceptional group. Then, either\smallskip 
\begin{enumerate} 
\item $G_{X,\omega}^\ast \simeq E_6$,~$\dim V_X = 27$ and~$d\in \{2,4,6\}$, or\smallskip 
\item $G_{X,\omega}^\ast \simeq E_7$,~$\dim V_X = 56$ and~$d\in \{3,\dots,15\}$ is odd. \smallskip
\end{enumerate} 
Moreover~$g\le g_{\max}$ for the following upper bound~$g_{\max}$ depending on~$d$:
\begin{align*}
(1) \quad
\begin{array}{c|ccccc}
d & 2 & 4 & 6  \\ \hline
g_{\max} & 7 & 6 & 8 
\end{array}
&&\quad (2) \quad
\begin{array}{c|ccccccc}
d & 3 & 5 & 7 & 9 & 11 & 13 & 15 \\ \hline
g_{\max} & 9 & 10 & 12 & 13 & 15 & 16 & 18
\end{array}
\end{align*}
\end{proposition}

\begin{proof} 
Since~$X\subset A$ is a smooth subvariety, the Tannaka group~$G_{X,\omega}^\ast$ acts on~$V_X$ via a minuscule representation~\cite[cor.~1.10]{KraemerMicrolocalI}. For the simple exceptional groups the only minuscule representations are the irreducible representations of dimension~$27$ for~$E_6$ and of dimension~$56$ for~$E_7$. The latter representation admits a nondegenerate bilinear form which is alternating, so in this case~$X$ must have odd dimension by~\cite[sect.~1.2]{JKLM}; hence for~$E_7$ we only need to consider cocharacters whose associated Hodge decomposition has odd length~$d$. Also, when~$X$ is a divisor,  the topological Euler characteristic of~$X$ is divisible by~$g!$, hence cannot be~$27$ or~$56$ up to sign. Therefore, we can assume~$d \le g - 2$. Since~$X\subset A$ has ample normal bundle, \cref{lem:hodge-estimate} and~\cref{cor:dim-via-hodge} imply that the Hodge cocharacter induces a cocharacter of~$G_{X, \omega}^\ast$ satisfying properties (H1), (H2) and (H3). \Cref{lem:hodge-estimate} also implies the following bound on the codimension:
\[ g - d + 1 \; \le \; h^d(X) \; := \; \chi(X, \Omega^d_X).\]
So it follows that the Hodge numbers
$
 h^p(X) = (-1)^{d-p} \chi(X, \Omega^p_X)
$
must be among those in~\cref{prop:E6,,prop:E7}. 
\end{proof} 

\begin{corollary} \label{cor:hodge-numbers}
If in the above proposition we moreover assume that~$d < g/2$, then the topological Euler characteristic~$\chi_\top(X)=(-1)^d \dim V_X$ and the Hodge numbers~$h^p(X)=(-1)^{d-p} \chi(X, \Omega^p_X)$ are among one of the following cases: 
\[
\begin{array}{c|c|c|c|c|c|c|c}
G_X^\ast & \dim V_X & d & g & h^0 & h^1 & h^2 & h^3 \\ \hline 
E_6 & 27 & 2 & 5, 6, 7 & 6 & 15 & 6 &   \\
E_7 & 56 & 3 & 7, 8, 9 & 7 & 21 & 21 & 7 
\end{array} \medskip
\]
\end{corollary} 

In the rest of this paper we will show that the~$E_6$ case in the above table occurs only when~$X$ is the Fano surface of lines on a smooth cubic threefold, and~$A$ is isogenous to its Albanese variety (which has~$g=5$).

\subsection{Surfaces with group~$E_6$ and the difference morphism} 
\label{sec:DifferenceMorphism}

Let us take a closer look at the first row of the table in~\cref{cor:hodge-numbers}. The corollary in particular says that if~$X\subset A$ is a smooth nondivisible irreducible subvariety with ample normal bundle and dimension~$<g/2$ such that~$G_{X,\omega}^\ast \simeq E_6$, then~$X$ must be a surface with 
\[ 
\chi(X, \cO_X) =6, \qquad \chi(X, \Omega^1_X) = -15, \qquad \chi(X, \Omega^2_X) = 6.
\]
We also know the Chern numbers: The top Chern class is~$c_2(X) = \dim \omega(\delta_S) = 27$, so~$c_1(X)^2 = 45$ by Noether's formula. Moreover, by the representation theory of~$E_6$ there is a unique one-dimensional subrepresentation inside $V_X\otimes V_X\otimes V_X$. Since by~\cite[sect.~10]{WeissauerAlmostConnected} or~\cite[lemma~3.5]{KraemerMicrolocalI} the one-dimensional representations of the Tannaka group correspond to skyscraper sheaves, it follows that there a unique skyscraper direct summand inside $\delta_X*\delta_X*\delta_X$. By base change this means that the sum morphism $X^3 \to A$ has a unique three-dimensional fiber and that this fiber is irreducible. To show that~$X$ is the Fano surface of lines on a smooth cubic threefold (in which case the above fiber parametrizes coplanar triples of lines), we will need one more piece of numerical information: The degree of the difference morphism
\[
 d\colon \quad X \times X \;\longrightarrow\; X - X \;\subset\; A.
\]
The representation theory of the exceptional group~$E_6$ implies:

\begin{proposition} \label{Prop:DegreeOfDifferenceMorphism} 
Let~$A$ be an abelian variety of dimension~$g\ge 5$ and~$X\subset A$ a smooth nondegenerate nondivisible surface with~$G_{X,\omega}^\ast \simeq E_6$. Then the difference morphism~$d \colon X \times X \to X-X$ is generically finite of degree~$\deg(d)\ge 6$.
\end{proposition}

\begin{proof} 
Let~$D:=X-X$. Since~$X$ is nondegenerate, the morphism~$d \colon X\times X \to D$ is generically finite, hence over some open dense subset~$U\subset D^\reg$ it restricts to a finite \'etale cover of degree~$\deg(d)$. The direct image~$\delta_X*\delta_{-X} = Rd_*(\delta_{X\times X})$ therefore restricts on~$U$ to a local system of rank~$\deg(d)$, placed in cohomological degree~$-\dim D = -4$. By adjunction that local system contains the trivial local system of rank one as a direct summand. The decomposition theorem~\cite[th.~6.2.5]{BBDG} therefore shows that
\begin{equation} \label{eq:geometry1} \tag{1a}
 \delta_X * \delta_{-X} \;\simeq\; \delta_D \oplus P 
\end{equation}
where~$P$ is a semisimple perverse sheaf on~$D$ which restricts on~$U \subset D$ to a local system of rank~$\deg(d)-1$, placed in cohomological degree~$-4$. Passing to the clean characteristic cycles of these perverse sheaves in the sense of~\cite[sect.~5.3]{JKLM}, we get
\begin{align} \label{eq:geometry2} \tag{1b}
 \cc(\delta_D) &\;=\; \Lambda_D + \sum_{Z} m_Z(\delta_D) \cdot \Lambda_Z, \\ \label{eq:geometry3} \tag{1c}
 \cc(P) &\;=\; (\deg(d)-1) \cdot \Lambda_D + \sum_Z m_Z(P)\cdot  \Lambda_Z
\end{align}
with certain integers~$m_Z(\delta_D), m_Z(P)\ge 0$. On the right-hand side of these equations the sums are taken over all integral proper subvarieties~$Z\subsetneq D$, and~$\Lambda_Z \subset T^*A$ denotes the conormal variety to~$Z$ inside~$A$. Hence we only need to show that the conormal variety~$\Lambda_D$ enters in~$\cc(\delta_X * \delta_{-X})$ with multiplicity at least six.

\medskip 

To interpret this in terms of representation theory, we need some notation. Fix a maximal torus~$T\subset E_6$. We use the fundamental weights~$\varpi_1, \dots, \varpi_6 \in X^\ast(T)$ from~\cref{sec:FiltrationsE6}. The dominant integral weights are precisely those of the form
\[
 \lambda \;=\; a_1 \varpi_1 + \cdots + a_6\varpi_6 
 \quad \text{with integers} \quad a_1, \dots, a_6 \;\ge\; 0,
\]
and for any such weight we will denote by~$V_\lambda$ the irreducible representation of~$E_6$ of highest weight~$\lambda$. The classes of these irreducible representations form a basis for the representation ring~$\rR(E_6)$, the Grothendieck ring of the tensor category of all finite-dimensional representations of~$E_6$. We identify the representation ring~$\rR(E_6)$ with the Weyl group invariants in the group algebra~$\bbZ[X^\ast(T)]$ by the map sending a representation~$V$ to its character
\[
 \characteristic(V) \;\in\; \bbZ[X^\ast(T)].
\]
In these terms the representation ring~$\rR(E_6)$ has a basis consisting of the Weyl group orbits of the dominant integral weights. For a weight~$\lambda \in X^\ast(T)$ we denote by~$[\lambda]\in \rR(E_6)$ its Weyl group orbit. The character of a representation is a single Weyl group orbit only if the representation is minuscule, so the basis given by Weyl group orbits differs from the one given by the irreducible representations. In what follows we are interested in the irreducible representations that appear in the tensor product of the 27-dimensional representation~$V=V_{\varpi_1}$ with its dual~$V^\vee = V_{\varpi_6}$:
\begin{equation} \label{eq:rep1} \tag{2a}
V\otimes V^\vee
\;\simeq\; V_{\alpha} \oplus V_{\beta} \oplus \bf{1}
\quad \text{with} \quad 
\begin{cases} 
\alpha \;=\; \varpi_2, \\
\beta \;=\; \varpi_1 + \varpi_6.
\end{cases} 
\end{equation}
Here~$V_{\alpha}$ is the adjoint representation of~$E_6$ and~$V_{\beta}$ is an irreducible representation of dimension 650. Their characters are
\begin{align} \label{eq:rep2} \tag{2b}
 \characteristic(V_{\alpha}) &\;=\; \phantom{0}6\cdot [0] + 1 \cdot [\alpha] , \\ \label{eq:rep3} \tag{2c}
 \characteristic(V_{\beta}) &\;=\; 20 \cdot [0] + 5\cdot [\alpha] + 1 \cdot [\beta].
\end{align}
Now we come back to geometry. Fix an isomorphism~$G_{X,\omega}^\ast \simeq E_6$. By \cref{cor:perverse-versus-hodge} we have~$G_\omega^\ast(\delta_X^H) = G_{X,\omega}^\ast$, so we may work as well with Hodge modules. Since~$\omega(\delta_X^H)$ is a minuscule representation, it must be one of the two irreducible representations of dimension 27, which are dual to each other. So we may assume~$\omega(\delta_X^H)\simeq V$. Then~\cref{eq:rep1} translates to
\begin{equation} \label{eq:tannaka1} \tag{3a}
 \delta_X^H * \delta_{-X}^H \;\simeq\; M_{\alpha} \oplus M_{\beta} \oplus \delta_0^H
\end{equation} 
where for~$\lambda \in \{\alpha, \beta\}$, we denote by~$M_\lambda \in \langle \delta_X^H\rangle$ the simple Hodge module with associated representation~$\omega(M_\lambda) \simeq V_\lambda$. The characteristic cycles of these Hodge modules can be computed as follows: By~\cite[th.~5.11, lemma~5.14]{JKLM} there is a natural way to attach to any Weyl group orbit~$[\lambda]$ an effective Lagrangian cycle~$\Lambda_{\lambda}$
on the cotangent bundle~$T^* A$ such that if we extend this map additively to the representation ring~$\rR(E_6)$, then it sends the character of any representation to the characteristic cycle of the corresponding perverse sheaf.
Thus~\cref{eq:rep2,,eq:rep3} translate to
\begin{align} \label{eq:tannaka2} \tag{3b}
 \cc(M_{\alpha}) &\;=\;\phantom{0}6\cdot \Lambda_0 + 1 \cdot \Lambda_{\alpha} , \\ \label{eq:tannaka3} \tag{3c}
 \cc(M_{\beta}) &\;=\; 20 \cdot \Lambda_0 + 5\cdot \Lambda_{\alpha} + 1 \cdot \Lambda_{\beta} ,
\end{align}
where all summands on the right-hand side are effective Lagrangian cycles. 

\medskip 

Comparing~\cref{eq:geometry1} and~\cref{eq:tannaka1}, we see that~$\delta_D^H\simeq M_\lambda$ for some~$\lambda \in \{\alpha, \beta\}$. We claim that 
\[
  \delta_D^H \;\simeq\; M_{\alpha}. 
\] 
Indeed, if~$\delta_D^H \simeq M_{\beta}$, then~$\cc(M_\beta)$ would contain the irreducible component~$\Lambda_D$ with multiplicity one. By~\cref{eq:tannaka3} this component would then have to enter in the summand~$\Lambda_{\beta}$, since the other summands enter with higher multiplicity. Then by~\cref{eq:geometry2} the other summand~$\Lambda_{\alpha}$ could involve only conormal varieties~$\Lambda_Z$ with \[\dim(Z)<\dim(D)=4.\] But then by~\cref{eq:tannaka2} the cycle~$\cc(M_{\alpha})$ would also only involve conormal varieties~$\Lambda_Z$ with~$\dim(Z)<4$. This would imply
\[
 \dim \Supp M_{\alpha} \;<\; \dim \Supp D \;=\; 4,
\]
hence~$\dim d^{-1}(\Supp M_{\alpha}) < 4$ because~$d \colon X\times X \to D$ is dominant and~$X\times X$ is irreducible. By~\cref{Hodge-level-of-direct-image} then
\[
 h^{p,q}(M_{\alpha}) \;=\; 0 \quad \text{for} \quad |p-q| \;>\; 3.
\]
But~$\omega(M_\alpha)\simeq V_\alpha$ is the adjoint representation of~$E_6$, and by~\cref{rem:adjoint-hodge-numbers} the Hodge cocharacter induces on the adjoint representation a Hodge decomposition of level~$4$, a contradiction. This proves that~$\delta_D^H\simeq M_{\alpha}$ as claimed. But then by~\cref{eq:tannaka2} the conormal variety~$\Lambda_D$ enters as a component in the cycle~$\Lambda_{\alpha}$. It then follows that the cycle
\[ \cc(\delta_X^H * \delta_{-X}^H) \;=\; 26\cdot \Lambda_0 + 6\cdot \Lambda_{\alpha} + 1 \cdot \Lambda_\beta
\]
contains~$\Lambda_D$ with multiplicity at least six and we are done.
\end{proof}

The above numerical results in fact characterize Fano surfaces of smooth cubic threefolds, as we will see in the classification of subvarieties with Tannaka group~$E_6$ in~\cref{Thm:E6SurfaceIsFanoSurface}. In particular, the degree of the morphism~$d\colon X\times X \to X-X$ is precisely~$6$. However, for this will not use any representation theory; instead, we will show by direct geometric arguments that any smooth projective surface with the above numerical properties is isomorphic to the Fano surface of a smooth cubic threefold. As a preparation for this, we gather in the next two sections some general facts about cubic hypersurfaces and Fano varieties of lines.

\section{Nonnormal cubic hypersurfaces}

In this section, we give a classification of nonnormal cubic hypersurfaces over an arbitrary algebraically closed field~$k$. 

\subsection{Statement}

In order to understand the Fano variety of lines on a singular cubic threefold it will be useful to have a complete classification of the nonnormal ones. Examples of these are cones over nonnormal cubic plane curves and cones over nonnormal cubic surfaces, where by a cone we mean the following: 

\begin{definition} Let~$\pi \colon V \to W$ be a linear map of finite-dimensional~$k$-vector spaces and~$Y \subset \bbP(W)$ a subvariety. If~$\pi$ is surjective but not an isomorphism, then the scheme-theoretic closure in~$\bbP(V)$ of the preimage of~$Y$ under the induced morphism $\bbP(V) \smallsetminus \bbP(\ker \pi) \to \bbP(W)$ is called the \emph{cone} over~$Y$ with respect to~$\pi$.
\end{definition}

The result of the classification of nonnormal cubic threefolds will be that apart from cones, the only one is the following:

\begin{example} \label{Ex:TwistedPlane} A subvariety~$X \subset \bbP^4$ is said to be a \emph{twisted plane} if it is defined by the equation
\[  x_0^2 x_1+x_0 x_2 x_4+x_3 x_4^2=0 \]
for some homogeneous coordinates~$x_0, \dots, x_4$ on~$\bbP^4$ \cite[sect.~2.2.9]{RuledVarieties}. The singular locus of~$X$ is the plane~$x_0=x_4= 0$, thus~$X$ is nonnormal.
\end{example}

For cubic surfaces we correct some arguments proposed by Dolgachev \cite{Dol12} in the discussion  after theorem~9.2.1. To this end, let~$\pi \colon S \to \bbP^2$ be the blow-up in a point~$o$, with exceptional divisor~$E \subset S$. Consider the closed embedding
\[ S \; \intoo \; \bbP(V^\vee) \quad \textup{where} \quad V = \rH^0(S, \cO_S(2)(-E)),\]
and let~$\cO_{S}(i) := \pi^\ast \cO_{\bbP^2}(i)$ for~$i \ge 0$. The purpose of this section is to prove the following result, which is probably well-known but for which we could not find a complete proof in the literature:

\begin{proposition} \label{Prop:NonNormalCubicHypersurfaces} Let~$X \subset \bbP^{n+1}$ be an integral nonnormal cubic hypersurface.

\begin{itemize}
\item If~$n = 1$, then~$X$ is a nodal or a cuspidal curve. \smallskip
\item If~$n = 2$, then~$X$ is either a cone over a nonnormal cubic integral plane curve or a projection of~$S \subset \bbP(V^\vee)$ from a point not in~$S$. \smallskip
\item If~$n = 3$, then~$X$ is either a cone over a nonnormal cubic surface or a twisted plane as in \cref{Ex:TwistedPlane}. \smallskip
\item If~$n \ge 4$, then~$X$ is a cone over a nonnormal integral cubic threefold.
\end{itemize}
\end{proposition}

\noindent
The rest of the section will be devoted to the proof of this statement.

\subsection{Reduction to surfaces}
We begin with the following:

\begin{lemma}\label{lemma singular locus nonormal cubic}
Let~$X\subset \bbP^{n+1}$ be a nonnormal integral cubic hypersurface with~$n \ge 1$, then~$X^\sing$ is a linear subspace of dimension~$n-1$.
\end{lemma}
\begin{proof}
We argue by induction on~$n \ge 1$. For~$n=1$ the only nonnormal integral cubic plane curves are the nodal and the cuspidal plane cubics, which both have only one singular point. Suppose~$n \ge 2$. By Serre's criterion of normality the singular locus of~$X$ has dimension~$n - 1$. Let~$Y \subset X^\sing$ be an irreducible component of dimension~$n-1$ and suppose that~$Y$ is not a linear subspace. Pick a point~$y \in Y$ such that the line through~$y$ and a general point~$y' \in Y \smallsetminus \{ y \}$ is not contained in~$Y$. As~$\deg X =3$ and~$Y \subset X^\sing$, every such line is contained in~$X$. An open subset of~$X$ is thus covered by lines going through~$y$. It follows that~$X$ is a cone with vertex~$y$ over an integral cubic hypersurface~$Z\subset \bbP^n$, thus~$Y^\sing$ is a cone with vertex~$y$ over~$Z^\sing$. In particular we have~$\dim Z^\sing=n-2$, hence the cubic hypersurface~$Z$ is not normal. By the induction hypothesis,~$Z^\sing$ is a linear subspace of dimension~$n-2$, thus~$Y^\sing$ is a linear subspace of dimension~$n-1$.
\end{proof}

\begin{lemma} \label{Lemma:ClassificationNonnormalCubics} Let~$X\subset \bbP^{n+1}$ be a nonnormal integral cubic hypersurface with~$n \ge 3$. 
\begin{enumerate}
\item If~$n = 3$, then~$X$ is either a cone over a nonnormal cubic surface or the twisted plane from \cref{Ex:TwistedPlane}. \smallskip
\item If~$n \ge 4$, then~$X$ is a cone over a nonnormal integral cubic threefold.
\end{enumerate}
\end{lemma}

\begin{proof} By \cref{lemma singular locus nonormal cubic}, the singular locus is a linear subspace of dimension~$n - 1$, which we may assume to have equations~$x_0 = x_1 = 0$. In such coordinates, the hypersurface~$X$ is defined by an equation of the form
\[
x_0^2 f_1 + x_0x_1f_2 + x_1^2f_3 = 0
\]
for some linear forms~$f_1,f_2,f_3$. Let~$m+2$ be the dimension of the linear span of the linear forms~$x_0$,~$x_1$,~$f_1$,~$f_2$ and~$f_3$. If~$n > m$, then~$X$ is the cone over a cubic hypersurface in~$\bbP^{m+1}$ which is necessarily integral and nonnormal. Since~$m \le 3 \le n$ the equality~$n = m$ holds only for~$m = 3$ in which case we may choose~$f_1,f_2,f_3$ as coordinate functions and obtain that~$X$ is the twisted plane in \cref{Ex:TwistedPlane}.
\end{proof}

With the above lemma, the proof of \cref{Prop:NonNormalCubicHypersurfaces} will be reduced at the end of the next section to the study of nonnormal cubic surfaces given below.

\subsection{Nonnormal cubic surfaces} \label{sec:NonnormalCubicSurfaces}

In order to conclude the proof of \cref{Prop:NonNormalCubicHypersurfaces} we need to describe nonnormal cubic surfaces which are not cones over nonnormal plane cubic curves. To do this, for a point~$v \in \bbP(V^\vee) \smallsetminus S$ let~$W_v \subset V$ be the subspace of linear forms vanishing at~$v$. Let
\[ p_v \colon \quad \bbP(V^\vee) \smallsetminus \{ v \} \; \too \; \bbP(W_v^\vee)\]
be the corresponding projection. The geometry of the projection~$p_v \colon S \to p_v(S)$ depends on the image of~$v$ via the projection with center~$E$ that we denote by
\[ q \colon \quad \bbP(V^\vee)\smallsetminus E \; \too \; \bbP(V'^\vee) \quad \textup{where} \quad V' := \rH^0(S, \cO_{S}(2)(-2E)).\]
More precisely, the line bundle~$\cL := \cO_{S}(1)(-E)$ is globally generated on~$S$, hence the restriction of~$q$ to~$S \smallsetminus E$ extends to a morphism~$q \colon S \to \bbP(V'^\vee)$. The latter is  the composite morphism
\[ S \; \too \; \bbP(\rH^0(S, \cL)^\vee) \; \intoo \; \bbP(\rH^0(S, \cL^{\otimes 2})^\vee) = \bbP(V'^\vee)\]
where the second map is the second Veronese embedding. Note that the first map can be identified with the projection~$S \to E$ onto the exceptional divisor. Let
\[ C \; := q(S) \;  \subset \; \bbP(V'^\vee)\]
be the smooth conic which is the image of the previous composite morphism. For any~$v' \in \bbP(V'^\vee)$, the closure of~$q^{-1}(v')$ in~$\bbP(V')$ is a plane~$P_{v'}$ which intersects~$S$ in the following way:
\begin{itemize}
\item If~$v' \in C$, then $S \cap P_{v'} = L \cup E$ where~$L \subset S$ is the strict transform of a line in~$\bbP^2$ passing through the point~$o$. \smallskip
\item If~$v' \not \in C$, then $S \cap P_{v'} = E$,  because~$C = q(S)$ and~$E$ is the center of the projection~$q$.
\end{itemize}
Recall that the strict transform of a line~$L \subset \bbP^2$ is mapped to a line if~$L$ contains the point~$o$, and to a smooth conic otherwise. With this in mind, we have:

\begin{proposition} \label{prop:ProjectionHirzebruchSurface} The image~$X_v := p_v(S) \subset \bbP(V'^\vee)$ is a nonnormal cubic surface and its singular locus~$X_v^\sing \subset \bbP(V'^\vee)$ is a line. The morphism~$p\colon S \to X_v$ is finite birational and on the ramification locus~$B \subset S$ it is of degree~$2$. More precisely:
\begin{itemize}
\item If~$q(v) \in C$, then~$B = L \cup E$ where~$L \subset S$ is the strict transform of a line in~$\bbP^2$ passing through~$o$ and~$p_v \colon L \cup E \to X_v^\sing$ identifies the two lines. \smallskip
\item If~$q(v) \not \in C$, then~$B$ is a smooth conic and~$q \colon B \to X_v^\sing$ is a degree~$2$ map.
\end{itemize}
\end{proposition}
\begin{proof} The nonnormality of~$X_v$ follows from the description of~$p_v \colon S \to X_v$. Indeed, if~$X_v$ were to be normal, then~$p_v$ would be an isomorphism because~$p_v$ is finite birational and~$S$ is smooth. But we will show that~$p_v$ is not an isomorphism, thus the surface~$X_v$ cannot be normal. The line bundle~$\cO_{S}(2)(-E)$ has self-intersection~$3$; thus~$S \subset \bbP(V^\vee)$ has degree~$3$ and the same is true for~$X_v \subset \bbP(W_v^\vee)$. By \cref{lemma singular locus nonormal cubic} the singular locus~$X_v^\sing$ is a line. \medskip

Suppose~$q(v) \in C$. In this case,~$B = P \cap S$ where~$P$ is the closure of~$q^{-1}(q(v))$. Indeed, as recalled above,~$P$ meets~$S$ in~$L \cup E$ where~$L$ is the strict transform of a line in~$\bbP^2$ passing through~$o$. When~$x \in L \cup E$, the line~$D$ joining~$x$ and~$v$ is contained in the plane~$P$. Therefore the scheme-theoretic intersection~$D \cap (L \cup E)$ is the cycle~$[x] + [y]$ where~$y \in E$ is equal to~$x$ if and only if~$x$ is the intersection point of~$L$ and~$E$. This shows that~$p_v \colon L \cup E \to X_v$ identifies the lines~$L$ and~$E$. To conclude it suffices to prove that~$p_v \colon S \smallsetminus B \to X_v$ is injective and takes values in~$X_v \smallsetminus p_v(B)$. Indeed, we show that given~$x \in S \smallsetminus (L \cup E)$ the line~$D$ joining~$x$ and~$v$ meets~$S$ only in~$x$. To see this, first notice that~$D \cap E = \varnothing$: Otherwise, the line~$D$ would lie in the closure~$P'$ of~$q^{-1}(q(x))$ and we would have~$q(x) = q(v)$.  Now~$D \cap S$ is contained in the preimage of~$q(D) \cap C$. The line~$q(D)$ meets the smooth conic~$C$ in the two distinct points~$q(x)$ and~$q(v)$, thus
\[ D \cap S \; \subset \; q^{-1}(q(D) \cap C) = L \cup L'\]
where~$P' \cap S = L' \cup E$ and~$L' \neq L$ is the strict transform of a line in~$\bbP^2$ passing through~$o$. The lines~$D$ and~$L$ do not meet (otherwise~$v$ would belong to~$S$) and~$D$ meets~$L'$ only in~$x$.\medskip

Suppose~$q(v) \not \in C$. We claim that there is a unique plane~$P \subset \bbP(V^\vee)$ containing~$v$ and the strict transform~$Q$ of a line in~$\bbP^2$ not passing through~$o$. For the uniqueness, let~$P$ and~$P'$ distinct be planes in~$\bbP(V^\vee)$ meeting~$S$ in the strict transform~$Q$ and~$Q'$ respectively of distinct lines in~$\bbP^2$ not passing through~$o$. Then~$P \cap P'$ is the meeting point of~$Q$ and~$Q'$, and in particular it is contained in~$S$. To see this, notice that~$P \cup P'$ is not contained in any hyperplane. If it were contained in some hyperplane~$H$ say, then we would have
\[ Q \cup Q' \; = \; (P \cup P') \cap S \; \subset \; H \cap P.\]
But by construction~$H \cap S$ is the strict transform of a (possibly singular) conic in~$\bbP^2$ that \emph{passes} through~$o$. It follows that~$P \cap P' = \{ v \}$ is contained~$S$, which contradicts the hypothesis~$v \not \in S$. For the existence, note that every line~$L$ in~$\bbP^2$ determines a plane~$P_L \subset \bbP(V^\vee)$ and these planes cover~$\bbP(V^\vee)$. Moreover, by construction the points in the cone in~$\bbP(V^\vee)$ over~$C$ are exactly the points lying in~$P_L$ for some~$L$ passing through~$o$. Since~$v$ is not one of those, we obtain the wanted plane. \medskip

We are now ready to conclude the proof. Let~$P$ be the plane containing~$v$ and the strict transform~$B$ of a line in~$\bbP^2$ not passing through~$o$. Then~$B \subset P$ is a smooth conic and any line through~$v$ meets~$B$ in two points (taking in account multiplicities). This shows that~$B \to X_v$ is of degree~$2$ onto its image. To conclude it suffices to show that~$p_v \colon S \smallsetminus B \to X_v$ takes values in~$X_v \smallsetminus p_v(B)$ and is injective. To do this, let~$x, y \in S$ be distinct points such that~$p_v(x) = p_v(y)$ and let~$D$ be the line joining~$x$,~$y$ and~$v$. Note that~$x$ and~$y$ cannot lie both in the exceptional divisor, otherwise we would have~$D = E$ hence~$v \in S$. It follows that~$x$ and~$y$ have distinct images under the blow-up map~$\pi \colon S \to \bbP^2$. Let~$L \subset \bbP^2$ be the line joining~$\pi(x)$ and~$\pi(y)$, and let~$Q \subset S$ be its strict transform. If the blow-up point~$o$ were to belong to~$L$, then~$Q$ would be mapped to a line in~$\bbP(V^\vee)$ containing~$x$ and~$y$, implying~$D = Q$ and~$v \in S$ again. Therefore~$L$ does not contain~$o$, so~$Q$ is embedded as a smooth conic in~$\bbP(V^\vee)$. To conclude, notice that the plane~$P'$ determined by~$Q$ must be~$P$: Otherwise, as argued above,~$P$ and~$P'$ would meet in a single point which belongs to~$S$; but~$v \in P \cap P'$ so that would imply~$v \in S$ once more. Finally, we have~$ P = P'$ hence~$Q = B$.
\end{proof}

\begin{proof}[{Proof of \cref{Prop:NonNormalCubicHypersurfaces}}]
We choose homogeneous coordinates~$x_0,x_1,x_2$ on~$\bbP^2$ and identify~$S$ with the blow-up of~$\bbP^2$ in~$o = [1:0:0]$. Then a basis of~$V$ is given by the monomials $u_{ij} = x_i x_j$ with $0 \le i \le j \le 2$ and $(i,j) \neq (0,0)$, and~$S$ is cut out by the vanishing of the polynomials $h_1= u_{02}u_{11} - u_{01}u_{12}$, $h_2 =u_{01}u_{22}- u_{02}u_{12}$ and $ h_3 =u_{11}u_{22}- u_{12}^2$.
With this notation the vector space~$V'$ is spanned by~$u_{11}$,~$u_{12}$ and~$u_{22}$ and the conic~$C$ has equation~$h_3 = 0$. If~$v \in \bbP(V)$ is defined by the vanishing of~$u_{01}$,~$u_{12}$,~$u_{22}$ and~$u_{02} + u_{11}$, then~$q(v) \in C$ and
\begin{equation} \label{eq:EqPointInConic} X_v : \quad 
 u_{12}^3 + u_{22}^2 u_{01} - u_{12}u_{22}(u_{02} + u_{11}) 
 = 0.
 \end{equation}
If instead~$v$ is defined by the vanishing of~$u_{ij}$ for~$(i,j) \neq (1,2)$, then~$q(v) \not \in C$ and 
\begin{equation} \label{eq:EqPointNotInConic}
X_v: \quad 
u_{02}^2u_{11} - u_{01}^2u_{22} 
= 0.
\end{equation}
The discussion preceding theorem 9.2.1 in \cite{Dol12} shows that up to a linear change of variables, any nonnormal integral cubic surface~$X \subset \bbP^3$  which is not a cone over a nonnormal plane cubic curve has equation \eqref{eq:EqPointInConic} or \eqref{eq:EqPointNotInConic}. Together with \cref{Lemma:ClassificationNonnormalCubics} and \cref{prop:ProjectionHirzebruchSurface} this concludes the proof.
\end{proof}

\section{On Fano varieties of lines}

We now prove some facts about Fano varieties of lines that will be used for \cref{Thm:E6SurfaceIsFanoSurfaceIntro}. Let~$k$ be an algebraically closed field. For a finite-dimensional~$k$-vector space~$V$, the Fano variety of lines on a subvariety~$X \subset \bbP(V)$ is the~$k$-scheme~$F_X$ whose points with values in a~$k$-scheme~$S$ are rank~$2$ vector bundles~$\cV 
\subset V \otimes_k \cO_S$ on~$S$ with locally free quotient such that~$\bbP(\cV) \subset \bbP(V) \times S$ is contained in~$X \times S$. Equivalently,~$F_X$ is the Hilbert scheme parametrizing subvarieties of~$X$ with Hilbert polynomial~$P(z) = z + 1$ with respect to line bundle~$\cO(1)$. 

\subsection{Rational normal scrolls} For an integer vector~$a = (a_1, \dots, a_r) \in \bbZ^r$ we put~$\cO_{\bbP^1}(a)  =  \cO_{\bbP^1}(a_1) \oplus \cdots \oplus \cO_{\bbP^1}(a_r)$. Let
\[p_a \colon \quad \bbP(a) = \bbP(\cO_{\bbP^1}(-a)) \; \too \; \bbP^1\] 
be the projective bundle of lines in~$\cO_{\bbP^1}(-a)$. If 
$ 0 \le a_1 \le \cdots \le a_r$ with $a_r > 0$, then the line bundle~$\cO_{\bbP(a)}(1)$ on~$\bbP(a)$ is generated by global sections. This defines a morphism
\[ f_a \colon \quad \bbP(a) \;\too\; \bbP(V(a)^\vee)
\quad \text{with} \quad V(a) \; = \; V(a_1) \oplus \cdots \oplus V(a_r) 
 \]
where $V(a_i) = \rH^0(\bbP^1, \cO_{\bbP^1}(a_i))$. 
The morphism~$f_a$ is birational and embeds each fiber $F$ of~$p_a$ in such a way that $f_a^\ast \cO(1) \iso \cO_F(1)$. Moreover,~$f_a$ is a closed immersion if and only if~$a_i > 0$ for all~$i = 1, \dots, r$. The image of~$f_a$ is called the \emph{rational normal scroll} of type~$a$ and will be denoted~$S(a)$. The rational normal scroll~$S(a)$ is smooth in the following cases \cite[p.~6]{EisenbudHarrisMinimalDegree}:
\begin{itemize}
\item $a_i > 0$ for all~$i$, in which case~$f_a \colon \bbP(a) \to S_a$ is an isomorphism; \smallskip
\item $a = (0, \dots, 0, 1)$, in which case~$S(a)$ is a projective~$r$-plane in~$\bbP(V(a)^\vee)$.\smallskip
\end{itemize}
Let~$F(a) = F_{S(a)}$ be the Fano variety of lines in~$S(a)$, and denote by~$\Gr_2(\cO_{\bbP^1}(-a))$ the Grassmannian of planes in~$\cO_{\bbP^1}(-a)$. 

\begin{proposition} \label{prop:FanoVarietiesOfScrolls} 
Suppose~$a_i > 0$ for all~$i = 1, \dots, r$. Then $F(a)=F(a)_v\sqcup F(a)_h$ with
\begin{itemize}
\item $F(a)_v = \Gr_2(\cO_{\bbP^1}(-a))$, \medskip
\item $F(a)_h = \bbP(W^\vee)$ where $W = V(1)^n$ and $n = |\{ i : a_i = 1\}|$.\smallskip
\end{itemize}
Moreover~$F(a)_h \subset \Gr_2(V(a)^\vee)$ is linearly embedded via the Pl\"ucker embedding. 
\end{proposition}

\begin{proof} Since~$f_a$ is a closed embedding, a line in~$S(a)$ is the image of a closed embedding~$\epsilon \colon \bbP^1 \into \bbP(a)$ such that the line bundle~$\epsilon^\ast \cO_{\bbP(a)}(1)$ has degree~$1$. The morphism~$g := p_a \circ \epsilon \colon \bbP^1\to \bbP^1$ is either dominant or constant. We say that the corresponding line is \emph{horizontal} in the first case, and \emph{vertical} in the second. The Fano variety of lines decomposes correspondingly as
\[ F(a) = F(a)_h \sqcup F(a)_v \]
where the subscript~$h$ resp.~$v$ stands for horizontal resp. vertical. Vertical lines are parametrized by the relative Grassmannian~$\Gr_2(\cO_{\bbP^1}(-a))$. For horizontal lines, the condition~$\deg \epsilon^\ast \cO_{\bbP(a)}(1) = 1$ implies that the morphism~$g$ is an isomorphism, hence up to precomposing~$\epsilon$ by~$g^{-1}$, we may assume~$g = \id$. It follows that horizontal lines in~$S(a)$ are given by degree~$1$ line bundles on~$\bbP^1$ which are quotients of~$\cO_{\bbP^1}(a)$. Such quotients are necessarily quotients of~$\cO_{\bbP^1}(a_1, \dots, a_n) = \cO_{\bbP^1}(1)^n$ where~$n$ is as in the statement. Finally, twisting by~$\cO_{\bbP^1}(-1)$, this discussion shows that vertical lines in~$S(a)$ are given by surjective morphisms~$\cO_{\bbP^1}^n \to \cO_{\bbP^1}$. This identifies~$F(a)_h$ with~$\bbP(W^\vee)$ for $W = V(1)^n$. For the second claim write~$V := V(a) = W \oplus V'$ with
\[  V' = V(a_{n+1}) \oplus \cdots \oplus V(a_r).\] The Pl\"ucker embedding of~$\bbP(W^\vee) \into \bbP(\Alt^2 V^\vee)$ is given by~$L \mapsto L \otimes \Alt^2(H^\vee)$ which is clearly linear.
\end{proof}

\begin{example} \label{Ex:SmoothScrollsOfDegree3} The case when~$S(a)$ is smooth with~$a_1 + \cdots + a_r = 3$ will be relevant to us. Because of the previous characterization, there are three possibilities:
\begin{itemize}
\item $a = (3)$. In this case~$S(3) \subset \bbP^3$ is the twisted cubic, i.e.~the triple Veronese embedding of the projective line. Hence $F(3) = \varnothing$.
\item $a = (1,2)$. In this case~$S(1,2) \subset \bbP^4$ is the blow-up~$\pi \colon S \to \bbP^2$ of~$\bbP^2$ in a point and it is embedded via the line bundle~$\pi^\ast \cO_{\bbP^2}(2)(-E)$, where~$E \subset S$ is the exceptional divisor. Thus
\[
 F(1, 2)_v = \{ \textup{fibers of the projection~$S \to E$}\}, \qquad
 F(1,2)_h = \{ E \}.
\]
\item $a = (1,1,1)$. In this case~$S(1,1,1)$ is~$\bbP^1 \times \bbP^2$ embedded in~$\bbP^5$ via the Segre embedding. Then, the Pl\"ucker embedding
\[ F(1,1,1)_v  \into \bbP(\Alt^2 \cO_{\bbP^1}(-1,-1,-1))= \bbP(2,2,2) \cong \bbP^1 \times \bbP^2\]
is an isomorphism and~$F(1,1,1)_h = \bbP^2$.
\end{itemize}
\end{example}

\subsection{The Fano variety of a cone} Let~$\pi \colon V \to W$ be a surjective linear map between finite-dimensional~$k$-vector spaces. For~$i = 0, 1, 2$ consider the following closed subvariety of the Grassmannian of planes in~$V$:
\[ \Gr^i_2(V, \pi)  \; := \; \{ E \subset V \mid \dim (E \cap \ker \pi) \ge i \} \; \subset \; \Gr_2(V).\]
We have~$\Gr^0_2(V, \pi) = \Gr_2(V)$ and~$\Gr^2_2(V, \pi) = \Gr_2(\ker \pi)$. For~$i = 0, 1$ the induced morphism
\[ \pi_i \colon \quad  \Gr^i_2(V, \pi) \smallsetminus \Gr^{i+1}_2(V, \pi) \; \too \; \Gr_{2-i}(W), \quad E \; \longmapsto \; \pi(E)\]
is surjective and smooth. Moreover, given a plane~$E \subset W$, the closure of the fiber of~$\pi_0$ in the corresponding point~$[E] \in \Gr_2(W)$ is~$\Gr_2(\pi^{-1}(E))$.

\medskip

  Let~$Y \subset \bbP(W)$ be a nonempty subvariety and~$X \subset \bbP(V)$ the cone over~$Y$. Consider the respective Fano varieties of lines~$F_X \subset \Gr_2(V)$ and~$F_Y \subset \Gr_2(W)$. At the set-theoretical level we have 
\[ F_X \quad = \quad \pi_0^{-1}(F_Y) \; \sqcup \; \pi_1^{-1}(Y) \; \sqcup \; \Gr_2(\ker \pi). \]  
At the topological level, note that~$\Gr_2(\ker \pi)$ is contained in the closure of any fiber of~$\pi_1$. For the other components, we have:

\begin{proposition} \label{Prop:FanoOfACone} With the notation above,
\begin{enumerate}
\item If~$Y$ is covered by lines, then~$\pi_0^{-1}(F_Y)$ is dense in~$F_X$ and~$\pi_0$ sets up a bijection between the irreducible components of~$F_X$ and those of~$F_Y$. \smallskip
\item If~$Y' \subset Y$ is an irreducible component that is not covered by lines, then the closure of~$\pi_1^{-1}(Y')$ is an irreducible component of~$F_{X}$.  \smallskip
\item If~$Y$ contains no lines, then~$\pi_1$ sets up a bijection between the irreducible components of~$F_X$ and those of~$Y$.
\end{enumerate}
\end{proposition}

\begin{proof} (1) The subset~$\pi_1^{-1}(Y) \subset F_X$ parametrizes lines contained in the closure of fibers of~$X \dashto Y$. By hypothesis any~$y \in Y$ belongs to some line~$L \subset Y$. Therefore lines in the closure of fiber of~$X \dashto Y$ in~$y$ are contained in 
$ \Gr_2(\pi^{-1}(E)) = \overline{\pi_0^{-1}([E])}$
where~$L = \bbP(E)$ and~$[E] \in \Gr_2(W)$ is the point defined by~$E$. \smallskip

(2) Lines in the closure of fiber of~$X \dashto Y$ in a point~$y \in Y'$ belonging to no line contained in~$Y$ form a subset of~$\pi_1^{-1}(Y')$ not lying in the closure of~$\pi_0^{-1}(F_Y)$. \smallskip

(3) is clear.
\end{proof}

It will be useful to combine \cref{Prop:FanoOfACone} with the following remark: For each irreducible component~$F' \subset F_Y$ and~$Y' \subset Y$, we have
\begin{align*}
\dim \pi_0^{-1}(F') \; &= \; \dim F' + 2(\dim V - \dim W),\\
 \dim \pi_1^{-1}(Y') \; &= \; \dim Y' + 2(\dim V - \dim W - 1).
\end{align*}

\begin{example} \label{Ex:ConesOverEllipticCurves} Suppose~$\dim W = 3$ and that~$Y$ is an integral cubic plane curve. Then the preceding statement says that the Fano variety of lines in the cone~$X$ is irreducible of dimension~$2(\dim V -4)+1$.
\end{example}

\begin{example} \label{Ex:ConesOverCubicSurfacesWithFinitelyManyLines} Suppose~$\dim W = 4$ and that~$Y \subset \bbP(W)$ is an integral cubic surface with only finitely many lines. This is the case if~$Y$ is smooth, or if it has only isolated singularities and is not a cone over an elliptic curve in~$\bbP^2$ \cite{BruceWall, Sakamaki}. Then~$F_X$ is pure of dimension~$2(\dim V - 4)$ and its irreducible components (with the reduced structure) are the following: the closure of~$\pi_1^{-1}(Y)$ in~$\Gr_2(V)$ and, for each line in~$Y$, one component isomorphic to~$\bbP^{2(\dim V - 4)}$ linearly embedded with respect to the Pl\"ucker embedding.
\end{example}

\begin{example} \label{Ex:ConesOverNonnormalCubicSurfaces} Suppose~$\dim W = 4$ and that~$Y \subset \bbP(W)$ is a nonnormal integral cubic surface that is not a cone over a cubic plane curve. Borrow notation from \cref{sec:NonnormalCubicSurfaces}. 
By \cref{Prop:NonNormalCubicHypersurfaces} the cubic surface~$Y$ is obtained from 
the blow-up~$S$ of~$\bbP^2$ in a point,
\[ S = S(1, 2) \; \subset \; \bbP(V^\vee) \quad \textup{where} \quad V = V(1,2), \]
by projecting it from a point~$v \in \bbP(V^\vee)$ not in~$S$. Recall from \cref{Ex:SmoothScrollsOfDegree3} that 
\[ F_{S} \; = \; F_{S, h} \sqcup F_{S, v} \qquad \textup{with} \quad F_{S, h} = 
\{ \textup{pt} \}, \quad F_{S,v} \cong \bbP^1. \]
It follows from \cref{prop:ProjectionHirzebruchSurface} that such a projection induces a morphism between Fano varieties of lines~$f \colon F_S \to F_Y$. The morphism~$f$ is a closed embedding once restricted to irreducible components. Moreover, following the case distinction in \cref{prop:ProjectionHirzebruchSurface} there are two possibilities:
\begin{itemize}
\item if~$q(v) \in C$, then~$F_{Y, \red} = f(F_{S, v})$ has one irreducible component; \smallskip
\item otherwise~$F_{Y, \red} = f(F_{S, v}) \sqcup f(F_{S, h}) \sqcup \{ Y^\sing \}$ has three irreducible components (recall by \cref{lemma singular locus nonormal cubic} that the singular locus~$Y^\sing$ of~$Y$ is a line).
\end{itemize}
The surface~$Y$ is covered by lines, thus by \cref{Prop:FanoOfACone}, the morphism~$\pi_{0}$ sets up a bijection between the irreducible components of the Fano variety of~$X$ and that of~$Y$. Continuing with the above distinction, we have the following cases:
\begin{itemize}
\item if~$q(v) \in C$, then~$F_X$ is irreducible of dimension~$1 + 2 (\dim V - 4)$; \smallskip
\item otherwise~$F_X$ has one irreducible component of dimension~$1 + 2 (\dim V - 4)$ and two irreducible components isomorphic to~$\bbP^{2 (\dim V - 4)}$ linearly embedded with respect to the Pl\"ucker embedding.
\end{itemize}
\end{example}

\subsection{The Fano variety of singular cubic hypersurfaces} Let~$X \subset \bbP^{n+1}$ be an integral cubic hypersurface and~$F = F_X$ its Fano variety of lines. Let~$\cO_F(1)$ be the restriction to~$F \subset \Gr_2(V)$ of the line bundle defining the Pl\"ucker embedding.

\begin{proposition}\label{Prop:HilbertPolynomial} If~$n = 3$ and~$\dim F \le 2$, then $F$ is a locally complete intersection of pure dimension $2$, and for all~$i \ge 0$,
\[ \chi(F, \cO_F(i)) =  \textstyle 45 \binom{i+1}{2} - 45 i + 6.\]
\end{proposition}

\begin{proof} With the notation in \cite{AltmanKleiman}, theorem 3.3~(iv) therein states that~$F$ is the zero locus of a regular section of the locally free sheaf~$\Sym^3 Q$. Note that the statement of theorem 3.3~(iv) gives a somewhat weaker statement: for each~$x \in F$ at which~$F$ has dimension~$\le 2$, the subvariety~$F$ is the zero subscheme of a section~$s_x$ of~$\Sym^3 Q$ which is regular at~$x$. Anyway, it is clear from its proof that the section~$s_x$ can be taken to be the same for all points~$x \in F$. When~$X$ has isolated singularities, the identity in the statement is equation (1.21.1) in \cite{AltmanKleiman}, whose proof relies on corollary 1.4, lemma 1.14 and prop. 1.15 (i) in loc. cit.. However, the proof of these facts only uses that~$F$ is the zero locus of a regular section of~$\Sym^3 Q$, hence holds under our more general assumptions.
\end{proof}

The main goal of this section is to complement \cite{AltmanKleiman} by describing precisely the cases where~$F$ is not of dimension~$2$. When~$X$ is a cone over a cubic surface with finitely many lines~$F$ has dimension~$2$; see \cref{Ex:ConesOverCubicSurfacesWithFinitelyManyLines}. On the other hand, in \cref{Ex:ConesOverNonnormalCubicSurfaces,,Ex:ConesOverEllipticCurves} we saw that the Fano variety of a cone over a nonnormal cubic surface or over an integral cubic plane curve has dimension~$3$. The twisted plane in \cref{Ex:TwistedPlane} also furnishes such an example:

\begin{example} \label{Ex:ConesOverTwistedPlane} Let~$X \subset \bbP^4$ be the twisted plane from \cref{Ex:TwistedPlane}, given by
\[  x_0^2 x_1+x_0 x_2 x_4+x_3 x_4^2=0. \]
The discussion in \cite[sect.~2.2.9]{RuledVarieties} shows that the Fano variety of lines~$F = F_X$ on~$X$ has three irreducible components:
\begin{itemize}
\item one~$3$-dimensional component is given by the lines in the~$1$-dimensional family of planes in~$X$ with equations, for~$[t_0 : t_1] \in \bbP^1$,
\[
t_1 x_0 - t_0 x_4 = 0, \qquad t_0^2 x_1 + t_0 t_1 x_2 + t_1^2 x_3 = 0;
\]
\item one~$2$-dimensional component is the dual of the plane~$X^\sing$ and is linearly embedded via the Pl\"ucker embedding; \smallskip
\item one~$2$-dimensional component obtained as the closure of the image of the morphism~$\bbP^2 \smallsetminus \{ [0:1:0], [0:0:1]\} \to \Gr_2(5)$ associating to~$[t_0 : t_1: t_2]$ the line joining the points
\[ [t_0^2 : t_1 t_2 : - t_0 t_2 : 0 : t_0 t_1], \quad [0 : t_1^2 : - 2 t_0 t_1 : t_0^2 : 0] \quad \in \; \bbP^4. \]
This component has degree~$4$ with respect to the Pl\"ucker embedding.
\end{itemize}
\end{example}

\noindent
The main result of this section states that these are the only cases of~$3$-dimensional Fano varieties. The argument goes through in any dimension and yields:

\begin{proposition} \label{Prop:ExpectedDimensionCubicHypersurface} For~$n = 3$ we have~$\dim F \le 3$ and equality holds if and only if~$X$ is either
\begin{itemize}
\item the twisted plane in \cref{Ex:ConesOverTwistedPlane}, or 
\item a cone over a nonnormal integral cubic surface, or 
\item a cone over a smooth cubic plane curve.
\end{itemize}
For~$n \ge 4$ we have~$\dim F \le 2n-3$ and if equality holds then~$X$ is a cone.
\end{proposition}

This is the combination of the following two lemmas with the classification of nonnormal cubic hypersurfaces in \cref{Prop:NonNormalCubicHypersurfaces}.

\begin{lemma} \label{lemma:DimensionFanoOfLinesThroughAFixedPoint} Let~$x \in X$ and~$F_{X, x} \subset F_X$ the subvariety of lines passing through~$x$. If~$X$ is not a cone, then
$ \dim F_{X, x} \le n-2$.
\end{lemma}

\begin{proof} Suppose there is an irreducible component~$Z \subset F_{X, x}$ of dimension~$\ge n-1$. Consider the universal line~$\pi \colon P \to Z$ and the evaluation morphism~$f \colon P \to X$. The morphism~$f \colon P\smallsetminus f^{-1}(x) \to X$ is injective because there is a unique line passing through two distinct points. As~$P \neq f^{-1}(x)$, this implies
\[ n = \dim X \ge \dim P\smallsetminus f^{-1}(x) = \dim Z + 1 \ge n,\]
hence~$\dim Z = n-1$ and~$f$ is surjective. It follows that the lines in~$Z$ cover~$X$ contradicting the assumption that~$X$ is not a cone.
\end{proof}

\begin{lemma}\label{lemma dim ysing at most one}
If~$\dim X^\sing\le n-2$, then either~$\dim F \le 2n - 3$ or~$X$ is a cone.
\end{lemma}

\begin{proof} Suppose that~$X$ is not a cone. By \cite[th.~4.2]{AltmanKleiman} the variety~$F$ is smooth of dimension~$2n - 4$ at the point corresponding to any line not meeting the singular locus. Therefore, it remains to treat the case of lines going through singular points. Since~$\dim X^\sing\le n-2$ by hypothesis, \cref{lemma:DimensionFanoOfLinesThroughAFixedPoint} implies that~$F$ at a line going through a singular point has dimension~$\le 2(n-2)$. This concludes the proof. 
\end{proof}

\begin{proof}[{Proof of \cref{Prop:ExpectedDimensionCubicHypersurface}}] By \cref{lemma dim ysing at most one} it remains to treat the case when~$X$ is not normal. We then conclude by \cref{Prop:NonNormalCubicHypersurfaces} and \cref{Ex:ConesOverNonnormalCubicSurfaces,,Ex:ConesOverEllipticCurves,,Ex:ConesOverTwistedPlane}.
\end{proof}

\subsection{A construction} \label{sec:Construction} Let~$X$ be a smooth projective variety over~$k$,~$V$ a finite-dimensional~$k$-vector space and 
\begin{equation} \label{eq:SESDatum} 0 \too \cT \too V \otimes_k \cO_X \too \cN \too 0 \end{equation}
be a short exact sequence of vector bundles over~$X$ with $\rk \cT = 2$. Consider the composite morphism 
\[ \pi \colon \quad \bbP(\cT) \intoo \bbP(V) \times X \too \bbP(V).\]
Let~$Y \subset \bbP(V)$ be the image of~$\bbP(\cT)$ and~$F = F_Y$ the Fano variety of lines of~$Y$. The short exact sequence \eqref{eq:SESDatum} defines a morphism
\[ \gamma \colon \quad X \; \too \;  F.\]
Assuming the generic finiteness of~$\pi$, the main result of this section states that the generic line of~$Y$ parametrized by~$X$ cannot be deformed into a line that is not of the same nature:

\begin{proposition} \label{Prop:TheFlatnessArgument} Suppose~$\dim V > \dim X + 2$ and~$\pi$ generically finite. Then there is a nonempty open subset~$X' \subset X$ such that for every~$x \in X'$ the tangent map of~$\gamma$ is surjective:
\[ \rT_x \gamma \colon \quad \rT_x X \; \ontoo \; \rT_{\gamma(x)} F.\]
If~$\gamma$ is generically finite, then the image of~$\gamma$ is an irreducible component of~$F$.
\end{proposition}

\begin{proof} 
Let~$Y' \subset Y$ be a nonempty open subset over which~$\pi$ is finite flat. We claim that~$X' = p(\pi^{-1}(Y'))$ does the job, where~$p \colon \bbP(\cT) \to X$ is the projection. To show this, given~$x \in X'$, let us consider a tangent vector to~$\gamma(x)$ in~$F$. By interpreting tangent vectors as points of~$F$ with values in~$R = k[\epsilon]$, this consists in a free~$R$-submodule~$\cV \subset V \otimes_k R$ of rank~$2$ with free cokernel and such that~$\bbP(\cV)$ is contained in~$Y_{R}$. Moreover, saying that such a tangent vector lies over~$\gamma(x)$ means that
\[ \bbP(\cV)_{\red} = \bbP(\cT_x) \subset \bbP(L).\]
Finally, the fact that~$x$ belongs to~$X'$ boils down to saying that~$\bbP(\cV)_{\red}$ meets~$Y'$. It follows that the fiber product
\[ P:= \bbP(\cT)_{R} \times_{Y_{R}} (\bbP(\cV) \cap V_{R}) \]
is nonempty and flat over~$R$. Better, an irreducible component of~$P_{\red}$ is the open subset
\[ \bbP(\cV)_{\red}  \cap Y' \; = \; \bbP(\cT_x) \cap Y' \quad \subset \; \bbP(\cT_x).\]
In particular, the open subset~$P \smallsetminus \overline{P \smallsetminus \bbP(\cT_x)}$ is nonempty and flat over~$R$. Consider its scheme-theoretic closure~$P'$ in~$\bbP(\cT)_{R}$. Then~$P'$ is flat over~$R$ and~$P'_{\red} = \bbP(\cT_x)$. It follows that the morphism~$\pi_{\rvert P'} \colon P' \to \bbP(\cV)$ is an isomorphism. Indeed the morphism~$\pi_{\rvert P'}$ is finite hence the line bundle~$\cL = \pi^\ast \cO_{\bbP(\cV)}(1)$ is ample on~$P'$. Since~$\pi$ induces an open immersion~$P \smallsetminus \overline{P \smallsetminus \bbP(\cT_x)} \into \bbP(\cV)$, for each~$d \ge 0$ the morphism
\[ \pi^\ast \colon \rH^0(\bbP(\cV), \cO(d)) \too \rH^0(P', \cL^{\otimes d})\]
is injective. By Nakayama's lemma, it is actually an isomorphism because it is an isomorphism modulo~$\epsilon$. Therefore, it induces an isomorphism
\[ \pi \textstyle \colon \quad P' = \Proj \Bigl( \bigoplus_{d \ge 0} \rH^0(P', \cL^{\otimes d}) \Bigr) \;  \stackrel{\sim}{\too} \; \bbP(\cV) = \Proj \Bigl( \bigoplus_{d \ge 0} \rH^0(\bbP(\cV), \cO(d)) \Bigr). \]
Now~$p \colon P' \to X$ factors through a morphism~$\Spec k[\epsilon] \to X$ with the unique point of~$\Spec  k$ being sent to~$x$: Indeed, the coherent sheaf of~$\cO_{X}$-algebras~$\cA = p_\ast \cO_{P'}$ is set-theoretically supported only on~$x$, so~$\cA$ is the skyscraper sheaf on~$x$ with value~$\rH^0(X, p_\ast \cO_{P'}) = \rH^0(P', \cO_{P'}) = k[\epsilon]$. Seeing this~$R$-valued point of~$X$ as a tangent vector to~$X$ in~$x$ gives the desired lifting and concludes the proof.
\end{proof}

As a sanity check, note that for~$X = \bbP^1$ the hypotheses of \cref{Prop:TheFlatnessArgument} are never fulfilled, hence this is not in contradiction to \cref{prop:FanoVarietiesOfScrolls}.

\section{Subvarieties with Tannaka group \texorpdfstring{$E_6$}{E6}}
\label{sec:proof-of-main-result}

Let~$A$ be an abelian variety of dimension~$g$ over an algebraically closed field~$k$ of characteristic zero. Putting together the results from the previous sections, we now complete our classification of smooth subvarieties with Tannaka group~$E_6$.

\subsection{Main result}
 
As recalled in the introduction, the Fano surfaces of lines on cubic threefolds are examples of subvarieties~$X\subset A$ with Tannaka group~$E_6$. The main result of this section is that these are the only possible examples among nondegenerate surfaces:

\begin{theorem} \label{Thm:E6SurfaceIsFanoSurface} Suppose~$g \ge 5$. For any smooth irreducible surface~$X \subset A$, the following three properties are equivalent:\smallskip
\begin{enumerate}
\item $X\subset A$ is nondivisible and nondegenerate with Tannaka group~$G_{X,\omega}^\ast \simeq E_6$.\smallskip
\item $X\subset A$ is nondivisible and nondegenerate with 
\[ \chi(X, \cO_X)=6, \qquad c_2(X)= 27,\]
the difference morphism~$ X\times X \to X-X$ has generic degree~$\ge 6$, and the sum morphism $X \times X \times X \to A$ has an irreducible fiber of dimension $\ge 3$.
\smallskip
\item $X$ is isomorphic to the Fano surface of lines on a smooth cubic threefold, and the canonical morphism~$\Alb(X) \to A$ is an isogeny. \smallskip
\end{enumerate}
\end{theorem}

Combining this result with~\cref{cor:hodge-numbers}, we obtain the following classification of subvarieties with Tannaka group~$E_6$:

\begin{corollary} \label{Thm:E6VarietyIsFanoSurface} Let~$X \subset A$ be a  smooth irreducible subvariety with ample normal bundle and dimension~$< g/2$. Then the following are equivalent:
\begin{enumerate}
\item $X\subset A$ is nondivisible with Tannaka group~$G_{X,\omega}^\ast \simeq E_6$.\smallskip
\item $X$ is isomorphic to the Fano surface of lines on a smooth cubic threefold, and the canonical morphism~$\Alb(X) \to A$ is an isogeny.
\end{enumerate}
\end{corollary}

The rest of the section is devoted to the proof of \cref{Thm:E6SurfaceIsFanoSurface}. By the Lefschetz principle and \cite[cor. 4.4]{JKLM} we may assume~$k = \bbC$. In this case we already know from~\cref{sec:DifferenceMorphism} that (1)~$\Rightarrow$ (2), and (3)~$\Rightarrow$ (1) is clear from the example in \cref{Ex:CubicThreefoldE6} of the introduction. So the only remaining point is (2)~$\Rightarrow$ (3), which gives a numerical characterization of Fano surfaces of smooth cubic threefolds.

\subsection{Strategy of the proof}\label{section strategy of the proof}

Before coming to the details, let us outline the main steps in the proof of (2)~$\Rightarrow$ (3). The idea is to recover the cubic threefold as in the example in \cref{Ex:CubicThreefoldE6}. By assumption, the subvariety~$X\subset A$ is a surface with
\[ \chi(X, \cO_X) = \chi(X, \Omega^2_X) = 6, 
\qquad c_2(X) = 27. \]
By Noether's formula we then also know~$c_1(X)^2 = 45$ and by the Hirzebruch-Riemann-Roch theorem we have~$\chi(X, \Omega^1_X) = -15$.
Now we can use the short exact sequence
\begin{equation} \label{eq:SESNormalE6}
 0 \too \cT = T_X \too \Lie A \otimes_k \cO_X \too \cN = \cN_{X/A} \too 0,
 \end{equation}
 to compute the Chern classes of the normal bundle~$\cN$ and find:
\[ c_1(\cN)^2 = 45, \qquad c_2(\cN) = c_1(X)^2 - c_2(X) = 18.\]
Borrow notation from \cref{sec:Construction} with the short exact sequence \eqref{eq:SESDatum} being \eqref{eq:SESNormalE6} and~$V = \Lie A$. Since~$X$ is nondivisible, it is not invariant by translation under any nonzero abelian variety, thus the Gauss map~$\gamma\colon X \to \Gr_2(\Lie A)$ is a finite morphism \cite[prop. 3.1]{Deb95}. By interpreting~$c_2(\cN)$ as the top Segre class of~$T_X$, the positivity of~$c_2(\cN)$ implies that the morphism~$\pi \colon \bbP(T_X) \to \bbP(\Lie A)$ is generically finite onto its image.
The bootstrap result is the computation of the degree of~$\pi$. To do this, we use that the difference morphism
\[
 d\colon \quad X \times X \;\longrightarrow\; X - X \;\subset\; A
\]
is generically finite of degree~$\ge 6$ by assumption. 
It is then a completely general fact that~$\deg \pi \ge \deg d$ within this framework. More precisely:

\begin{lemma} \label{Prop:DegreeOfPiE6} Let~$Z \subset A$ be an integral smooth subvariety, and suppose that the difference morphism~$d_Z \colon Z \times Z \to A$ and the morphism~$\pi_Z \colon \bbP(T_Z) \to \bbP(\Lie A)$ are generically finite onto their images. Then
\[ \deg \pi_Z \ge \deg d_Z.\]
\end{lemma}

In particular we have~$\deg \pi \ge 6$. On other hand, by construction the degree of~$\pi$ has to divide~$c_2(\cN) = 18$. Therefore, there are only three possibilities:
\[ \deg \pi \in \{ 6, 9, 18 \}. \]
This constrains also the degree of~$\gamma$, as it must divide the degree of~$\pi$. Moreover, the pull-back on~$X$  of the line bundle~$\cO(1)$ on~$\bbP(\Alt^2 \Lie A)$ via the composite of~$\gamma$ and of the Pl\"ucker embedding is~$\Alt^2 (T_X^\vee) = \Omega^2_X$. It follows that the degree of~$\gamma$ needs to divide also~$c_1(\Omega^2_X) = c_1(X)^2 = 45$. 
As a consequence, we have
\begin{equation} \label{Eq:DegreeOfGaussMapE6}  \deg \gamma \in \{ 1,3, 9\}. \end{equation}
As in \cref{sec:Construction} let~$Y \subset \bbP(\Lie A)$ be the image of~$\pi$. Then~$Y$ has dimension~$3$ and degree
\[ \deg Y = c_2(\cN) / \deg \pi \in \{1, 2,  3\}.\]
On the other hand, the nondegeneracy of~$X$ implies that~$X$ generates the abelian variety~$A$. In particular~$Y$ is not contained in any hyperplane, which rules out the possibility~$\deg Y = 1$. Furthermore, the classical lower bound (see \cite[prop. 0]{EisenbudHarrisMinimalDegree} for instance)
\[\deg Y \ge 1 + \codim Y\]
implies that~$Y$ has codimension~$1$ or~$2$. Moreover, if~$\codim Y = 2$, then equality holds in the above inequality. In loc.cit. subvarieties for which the above inequality is an identity are completely classified. We rule out the codimension~$2$ by comparing the Fano variety in each of these cases with the Fano variety~$F$ of lines  on~$Y$:

\begin{proposition}\label{Prop:ImageOfPiIsACubicHypersurface} We have~$g = 5$ and~$Y$ is a cubic hypersurface.
\end{proposition}

In particular, $\deg \pi = 6$ and hence $\deg \gamma \in \{1,3\}$. It remains to show that~$Y$ is smooth. For this, we will show $F$ is a smooth surface and apply~\cite[th.~4.2]{AltmanKleiman}. We begin with:

\begin{proposition} \label{Prop:GaussMapIsBirational} The Fano variety $F$ is an integral surface and the finite morphism $\gamma \colon X \to F$ is birational.
\end{proposition}

To conclude, a comparison of the Hilbert polynomials with respect to the Pl\"ucker embedding of~$X$ and of the Fano variety~$F$ of lines on~$Y$ will show:

\begin{proposition} \label{Prop:E6CubicHypersurfaceIsSmooth}The hypersurface~$Y$ is smooth and~$\gamma \colon X \to F$ is an isomorphism.
\end{proposition}

The rest of this section is devoted to the proof of \cref{Prop:DegreeOfPiE6,,Prop:ImageOfPiIsACubicHypersurface,,Prop:GaussMapIsBirational,,Prop:E6CubicHypersurfaceIsSmooth}.

\subsection{The projection to the tangent space: Proof of \cref{Prop:DegreeOfPiE6}} 
Let us write~$d = d_Z$ and~$\pi = \pi_Z$. 
Let~$D:= Z - Z$ be the image of the difference morphism. Consider the blow-up~$\tilde{A}$ of~$A$ in~$0$, the blow-up~$\tilde{D}$ of~$D$ in~$0$ and the blow-up~$\tilde{P}$ of~$P = Z \times Z$ in the diagonal. The normal bundle of the diagonal in~$Z \times Z$ is by definition the tangent bundle of~$Z$. It follows that the exceptional divisor of~$\tilde{P} \to P$ is~$\bbP(T_Z)$. The exceptional divisor of~$\tilde{A} \to A$ is~$\bbP(\Lie A)$ and that of~$\tilde{D} \to D$ is the tangent cone~$C$ of~$D$ at~$0$. Let~$\tilde{d} \colon \tilde{P} \to \tilde{D}$ be the morphism induced by the difference map. The situation is summarized in the following commutative diagram:

\begin{equation} \label{eq:DiagramBlowUpE6}
\begin{tikzcd}
\bbP(T_Z) \ar[d, hook] \ar[r] & \ar[d, hook] C \ar[r, hook] & \bbP(\Lie A) \ar[d, hook]\\
\tilde{P} \ar[d] \ar[r, "\tilde{d}"] & \tilde{D} \ar[d] \ar[r, hook] & \tilde{A} \ar[d]  \\
P = Z \times Z \ar[r, "d"] & D = Z - Z \ar[r, hook]& A
\end{tikzcd}
\end{equation}

\begin{lemma} \label{Lemma:StrictTransformOfDifferenceMorphism} The composite morphism~$\bbP(T_Z) \to \bbP(\Lie A)$ is~$\pi$.
\end{lemma}

\begin{proof} We reduce to the case~$Z = A$ by noticing that~$d$ is the restriction to~$Z \times Z$ of the difference morphism~$A \times A \to A$ and~$\pi$ is the restriction to~$\bbP(T_Z)$ the projection \[\bbP(T_A) = \bbP(\Lie A) \times A \too \bbP(\Lie A).\]
Write~$d$ as the composite of the automorphism~$(x, y) \mapsto (x-y, y)$ of~$A \times A$ and the second projection~$\pr_2 \colon A \times A \to A$. In this way we reduce to the situation where we replace~$d$ by the second projection~$\pr_2$ and~$\tilde{P}$ by the blow-up of~$A \times A$ in~$A \times \{ 0 \}$. In this case the result to prove is that the morphism
$A \times \bbP(\Lie A) \to \bbP(\Lie A)$ induced by~$\pr_2$ is the second projection. This is of course clear and concludes the proof.
\end{proof}

To conclude the proof of \cref{Prop:DegreeOfPiE6}, notice first that the top-left square in diagram \eqref{eq:DiagramBlowUpE6} is cartesian. Indeed, this boils down to the fact that the scheme-theoretic fiber of~$d$ in~$0$ is the diagonal~$\Delta_Z \colon Z \to Z \times Z$. To see this, notice that the difference morphism~$A \times A \to A$ is smooth and the scheme-theoretic preimage of~$0$ is the diagonal~$\Delta_A \colon A \into A \times A$ with its reduced structure; the claim then follows because the following square is cartesian:
\[
\begin{tikzcd}
Z \ar[r, "\Delta_Z"] \ar[d] & Z \times Z \ar[d] \\
A \ar[r, "\Delta_A"] & A \times A
\end{tikzcd}
\]

\begin{lemma} Let~$f \colon S \to T$ be a proper surjective morphism between varieties and~$T' \subset T$ a subvariety whose scheme-theoretic preimage~$S' = S \times_T T'$ is integral. If the morphisms~$f$ and~$f_{\rvert S'} \colon S' \to T'_{\red}$ are generically finite, then
\[ \deg (f_{\rvert S'}) \ge \deg (f). \]
\end{lemma}

\begin{proof} 
This follows from upper semicontinuity of the rank of $f_\ast \cO_S$ on the locus in~$T$ where $f$ is finite.
\end{proof}

We apply the preceding lemma with~$S = \tilde{P}$,~$T = \tilde{D}$,~$f = \tilde{d}$ and~$T' = C$. The hypotheses are fulfilled because~$\tilde{d}$ is generically finite, the scheme-theoretic preimage of~$C$ is~$\bbP(T_Z)$, and the restriction of~$\tilde{d}$ to~$\bbP(T_Z)$ is~$\pi$ by \cref{Lemma:StrictTransformOfDifferenceMorphism} hence generically finite by hypothesis. Moreover, the morphism~$\tilde{d}$ is generically finite of same degree of~$d$ because blow-ups are birational maps. 
\qed

\subsection{The cubic hypersurface: Proof of \cref{Prop:ImageOfPiIsACubicHypersurface}} By hypothesis~$X$ is nondegenerate, thus it generates the abelian variety~$A$. It follows that~$Y$ cannot be contained in any hyperplane. Therefore the lower bound~$\deg Y \ge 1 + \codim Y$ holds by \cite[prop. 0]{EisenbudHarrisMinimalDegree}. As already argued, we have~$\deg \pi \in \{3, 9 \}$ hence
\[ \deg Y = c_2(\cN) / \deg \pi \in \{ 2, 3\}.\]

We start by excluding the case~$\deg Y = 2$. In this case we have~$\codim Y= 1$ hence~$g = 5$ and~$Y$ is a~$3$-dimensional quadric. To rule out the case, let~$F$ be the Fano variety of lines in~$Y$ and recall by \cref{Prop:TheFlatnessArgument} that the image of~$\gamma \colon X \to F$ is a~$2$-dimensional irreducible component of~$F$. The rank~$r$ of the quadric~$Y$ is necessarily~$3$,~$4$ or~$5$ because~$Y$ is integral. We proceed now case by case:

\begin{itemize}
\item If~$r = 3$, then~$Y$ is a cone over a smooth plane conic~$Q$ and~$F$ is irreducible of dimension~$\dim Q + 2(5 - 3- 1) = 3$ by \cref{Prop:FanoOfACone}. Contradiction. \smallskip
\item If~$r = 4$, then~$Y$ is a cone over the rational normal scroll~$S(1,1) = \bbP^1 \times \bbP^1$ in~$\bbP^3$ via the Segre embedding \cite[p.~285]{HarrisAlgebraicGeometry}. By \cref{prop:FanoVarietiesOfScrolls} the Fano variety of lines in~$S(1,1)$ is~$\bbP^1 \sqcup \bbP^1$, so $F$ has two irreducible components of dimension~$\dim \bbP^1 + 2(5 - 4) = 3$ by \cref{Prop:FanoOfACone}. Contradiction. \smallskip
\item If~$r = 5$, then~$Y$ is smooth quadric threefold in~$\bbP^4$ and~$F \cong \bbP^3$; see for instance \cite[ex.~22.6]{HarrisAlgebraicGeometry}. Contradiction.
\end{itemize}
The above discussion shows
\[ \deg Y = 3.\]
The inequality~$\deg Y \ge 1 + \codim Y$ implies that~$Y \subset \bbP(\Lie A)$ has codimension~$1$ or~$2$.  It remains to rule out the codimension~$2$ case. If this is the case, the variety~$Y$ is of minimal degree, hence by \cite[th.~1]{EisenbudHarrisMinimalDegree} it is either a cone over a smooth rational scroll or a cone over the Veronese surface, that is,~$\bbP^2$ embedded in~$\bbP^5$ by the second Veronese embedding. Now taking a cone with respect to some linear projection preserves the codimension. This rules out at once the case of the Veronese surface since it has codimension~$3$ while~$Y$ has codimension~$2$. Suppose~$Y$ is a smooth rational normal scroll~$S = S(a)$ or a cone over such a scroll. By \cref{Prop:TheFlatnessArgument} the scheme-theoretic image of the finite morphism~$\gamma \colon X \to F$ is one of the irreducible components of~$F$. A case by case distinction shows that this is a contradiction. More precisely, the tuple~$a = (a_1, \dots, a_r)$ satisfies~$1 \le a_1 \le \cdots \le a_r$ and~$a_1 + \cdots + a_r = 3$. This leaves us with the~$3$ possibilities described in detail in \cref{Ex:SmoothScrollsOfDegree3}:
\[ a = (3), (1,2), (1,1,1).\]

Suppose~$\dim S = 1$. In this case the scroll~$S$ is the image of the triple Veronese embedding~$\bbP^1 \into \bbP^3$. As~$S$ contains no lines, \cref{Prop:FanoOfACone} implies that~$F$ is irreducible of dimension~$3$. But~$X$ dominates a component of~$F$. Contradiction. \medskip

Suppose~$\dim S = 2$. In this case the Fano variety is~$\bbP^1 \sqcup \{ \textup{pt}\}$. By \cref{Prop:FanoOfACone} the component of~$F$ over~$\bbP^1$ has dimension~$\dim \bbP^1 + 2(6-5) = 3$ hence~$X$ cannot dominate it. Instead again by \cref{Prop:FanoOfACone} the component of~$F$ over the singleton is a projective plane embedded linearly via the Pl\"ucker embedding. If~$X$ were to dominate it, we would have~$\deg \gamma = 45$, contradicting \eqref{Eq:DegreeOfGaussMapE6}.\medskip

Suppose~$\dim S = 3$. In this case~$Y = S$ and~$F \iso (\bbP^1 \times \bbP^2) \sqcup \bbP^2$. It follows that~$X$ has to dominate the component isomorphic to the projective plane. But this component is linearly embedded via the Pl\"ucker embedding. As above, if~$X$ were to dominate this component, then~$\deg \gamma = 45$ which contradicts \eqref{Eq:DegreeOfGaussMapE6}.  \medskip
\qed

\subsection{The Fano surface and the birationality of the Gauss map: Proof of \cref{Prop:GaussMapIsBirational}}
By \cref{Prop:ExpectedDimensionCubicHypersurface} we have~$\dim F \le 3$, with equality if and only if~$Y$ is either the twisted plane in \cref{Ex:ConesOverTwistedPlane} or a cone over a nonnormal cubic surface or a cone over a smooth cubic plane curve. These cases are excluded by looking at the Fano variety:
\begin{itemize}
\item If~$Y$ is a cone over an integral cubic curve, then~$F$ is irreducible and has dimension~$3$; see \cref{Ex:ConesOverEllipticCurves}. \smallskip
\item If~$Y$ is a cone over a nonnormal cubic surface that is not itself a cone over a nonnormal cubic curve, then the~$2$-dimensional irreducible components of~$F$ are linearly embedded via the Pl\"ucker embedding; see \cref{Ex:ConesOverNonnormalCubicSurfaces}. This would imply~$\deg \gamma = 45$, which contradicts \eqref{Eq:DegreeOfGaussMapE6}. \smallskip 
\item If~$Y$ is the twisted plane, then by \cref{Ex:ConesOverTwistedPlane} the~$2$-dimensional irreducible components of~$F$ have degree~$1$ and~$4$ with respect the Pl\"ucker embedding; this again contradicts \eqref{Eq:DegreeOfGaussMapE6}.
\end{itemize}
The above shows that we can assume~$\dim F = 2$. To conclude that $F$ is integral, it suffices to show that the Gauss map $\gamma \colon X \to F$ is birational onto its image: Indeed, assuming this, let $\cL$ be the line bundle on~$F\subset \Gr_2(\Lie A)$ which is the restriction of~$\cO(1)$ on~$\bbP(\Alt^2 \Lie A)$ via the Pl\"ucker embedding. Then $\cL$ and~$\gamma^\ast \cL$ have self-intersection
\[ \cL.\cL = 45 = \gamma^\ast \cL.\gamma^\ast \cL.\]
The first identity holds by \cref{Prop:HilbertPolynomial}, the second is $c_1(X)^2 = 45$ since $\gamma^\ast \cL$ is the canonical bundle of~$X$ by definition of~$\gamma$. Since~$\gamma$ is birational onto its image and $\cL$ is ample, this forces~$F$ to be irreducible. \Cref{Prop:TheFlatnessArgument} then implies that $F$ is generically reduced, hence satisfies condition~$R_0$ because its singular locus has codimension~$\ge 1$. But \cref{Prop:HilbertPolynomial} implies that $F$ is a locally complete intersection, thus satisfies condition $S_1$. Therefore, $F$ is reduced by \cite[5.8.5]{EGAIV2}.

\medskip

We now show that $\gamma \colon X \to F$ is birational onto its image. To do this, we mimic the relation between incidence divisors and the canonical bundle of the Fano surface of a smooth cubic threefold~\cite[p. 326]{ClemensGriffiths}. For any $x \in X$, consider the line $L_x := \bbP(T_x X)$ in $\bbP(\Lie A)$. Inside $X\times X$ consider the locus of couples~$(x, y)$ with~$L_x\cap L_y \neq \varnothing$ and let
\[ D \; \subset \; X \times X\]
be the union of its top dimensional components.
For $x\in X$, let 
$D_x = p(q^{-1}(x))$ be the fiber of the second projection
$q\colon D\to X$, considered as a subvariety of $X$ via the first projection $p\colon D\to X$. By~\cite[lemma II.12]{Ran}, the nondegeneracy of~$X$ implies $D_x \neq X$. 
On the other hand, since $\pi\colon \bbP(T_X) \to Y$ is generically finite of degree $6$ and $\deg(\gamma) \in \{ 1, 3 \}$, there are at least two distinct lines $L_x$, $L_{x'}$ through a generic $y \in Y$. Thus, $D_x$ is a divisor on $X$ for generic $x \in X$. 

\begin{lemma} \label{Lemma:FormulaLinesCanonical} \label{Lemma:ExistenceGoodTriples}
Let $U\subset X$ be the open subset where $q\colon D \to X$ is flat.
\begin{enumerate} 
\item There exist $x, y, z \in U$ such that the lines $L_x,L_y,L_z\subset \bbP(\Lie A)$ are pairwise distinct and span a plane.\smallskip
\item For any such $x, y, z$, the sum $D_{x} + D_{y} + D_{z}$ is a canonical divisor on $X$.
\end{enumerate} 
\end{lemma}

\begin{proof} 
(1) By assumption, there exists an integral fiber $I\subset X^3$ of dimension $\ge 3$ of the sum morphism $X^3 \to A$. We first claim that for generic $(x, y, z)\in I$, the three lines~$L_x, L_y,L_z$ are pairwise distinct: Otherwise the fiber $I$ would be contained in the preimage of the big diagonal $\Delta \subset \Gr_2(\Lie A)^3$ under the morphism
\[ \gamma^3\colon\quad X^3 \too \Gr_2(\Lie A)^3 .\]
Now $\Delta$ has three irreducible components which are permuted under the action of the symmetric group $S_3$. Since $I$ is irreducible and stable under the action of $S_3$, it must lie in the preimage of the intersection the irreducible components of $\Delta$. This intersection is the small diagonal of $\Gr_2(\Lie A)^3$ and its preimage via~$\gamma^3$ has dimension $2$ because $\gamma$ is finite and $\dim X = 2$, thus contradicting $\dim I \ge 3$. \medskip

We next claim that for any $(x,y,z)\in I$ the lines $L_x, L_y, L_z$ are coplanar.  Since coplanarity is a closed condition, it suffices to show this for generic~$(x,y,z)\in I$. The tangent map of the sum morphism $X^3 \to A$ has generic rank~$\le 3$ on~$I$ because it contracts $I$. Since the sum map on $A$ induces the sum on tangent spaces, this implies 
$ \dim (T_x X + T_y X + T_z X ) \le 3$ for generic $(x,y, z) \in I$.
Passing to projective spaces this means that $L_x,L_y,L_z \subset \bbP(\Lie A)$ are coplanar.\medskip

Finally, we claim that $I \subset X^3$ dominates $X$ via each projection: Since $I$ is irreducible and stable under permutations, we would otherwise have 
$I = C^3$ for a curve $C \subset X$, but then the image of $I$ under the sum map would be $C+C+C \subset A$, which is not a point.

\medskip 

(2) Let $P$ be the plane in $\bbP(\Lie A)$ spanned by $L_x$, $L_y$ and $L_z$ and $H \subset \Gr_2(\Lie A)$ the subvariety of points in $\Gr_2(\Lie A)$ corresponding to lines meeting $P$. The subvariety $H$ is a hyperplane section of the Pl\"ucker embedding. The nondegeneracy of~$X$ implies $\gamma(X)$ meets any such hyperplane properly \cite[lemma II.12]{Ran}, hence~$\gamma^{-1}(H)$ is a divisor linearly equivalent to $\gamma^\ast \cL$, which is isomorphic to the canonical bundle of $X$. By construction, we have
\[ P \cap Y = L_x \cup L_y \cup L_z,\]
hence $\gamma^{-1}(H) = D_x \cup D_y \cup D_z$, which concludes the proof.
\end{proof}

\begin{corollary} For $U\subset X$ as above and $x, y \in U$, we have $D_x . D_y = 5$.
\end{corollary}

\begin{proof} By design, the divisors $D_x$ form a flat family over $U$, thus the intersection number $n=D_x.D_y$ is independent of $x$ and $y$. For $x, y, z \in U$ such that the lines $L_x$, $L_y$ and $L_z$ are pairwise distinct and span a plane, we thus have
\[ 9n = (D_x + D_y + D_z)^2 = c_1(X)^2 = 45,\]
because $D_x + D_y + D_z$ is a canonical divisor by \cref{Lemma:FormulaLinesCanonical}. Therefore $n = 5$.
\end{proof}

We can now prove that $\gamma$ is birational onto its image: The divisors $D_x$ are the pullback of Weil divisors on~$\gamma(X)$ because the incidence relation defining them comes from~$\Gr_2(\Lie A)$. The projection formula then implies that $\deg(\gamma)$ divides $D_x . D_y$ for $x, y \in U$. Since $\deg(\gamma)\in \{1,3\}$ and $D_x . D_y = 5$, it follows that~$\deg(\gamma)=1$. \qed

\subsection{Smoothness: Proof of \cref{Prop:E6CubicHypersurfaceIsSmooth}}  It follows from \cref{Prop:HilbertPolynomial} that the Hilbert polynomial of~$F$ with respect to the Pl\"ucker embedding is
\[ \chi(F, \cL^{\otimes i}) =  \textstyle 45 \binom{i+1}{2} - 45 i + 6,\]
where~$\cL$ is the line bundle on~$\Gr_2(\Lie A)$ obtained as restriction via the Pl\"ucker embedding of the line bundle~$\cO(1)$ on~$\bbP(\Alt^2 \Lie A)$. On the other hand, for all~$i \ge 0$,
\[ \chi(X, \gamma^\ast \cL^{\otimes i}) =  \textstyle 45 \binom{i+1}{2} - 45 i + 6.\]
To see this, first notice that~$\gamma^\ast \cL \iso \Omega^2_X$ by definition of the morphism~$\gamma$, thus by Serre's duality theorem, we have
\[\chi(X, \gamma^\ast \cL)  =  \chi(X, \cO_X) = 6.\]
On the other hand, the self-intersection of~$\gamma^\ast \cL$ is~$c_1(X)^2 = 45$, hence the Hirzebruch-Riemann-Roch theorem implies the wanted identity. Now, consider the short exact sequence of~$\cO_F$-modules
\[ 0 \too \cO_F \too \gamma_\ast \cO_X \too Q \too 0.\]
Since the morphism~$\gamma$ is finite, the~$\cO_F$-module~$\gamma_\ast \cO_X$ is coherent, thus so is~$Q$. Taking the tensor product with~$\cL^{\otimes i}$, we have the short exact sequence
\[ 0 \too \cL^{\otimes i} \too \gamma_\ast \gamma^\ast  \cL^{\otimes i} \too Q \otimes \cL^{\otimes i} \too 0\]
for all integer~$i \ge 0$. The finiteness of the morphism~$\gamma$ also implies the vanishing of the higher direct images~$\rR^q \gamma_\ast \gamma^\ast \cL^{\otimes i}$ for~$q \ge 1$, which yields the identity
\[ \chi(F, \gamma_\ast \gamma^\ast  \cL^{\otimes i}) = \chi(X, \gamma^\ast  \cL^{\otimes i})\]
Since Euler characteristics are additive in short exact sequences, this implies that the Hilbert polynomial of~$Q$ vanishes altogether, hence~$Q = 0$. The morphism~$\gamma$ being finite, this implies that~$\gamma \colon X\to F$ is an isomorphism. By hypothesis~$X$ is smooth, thus so is~$F$. Since~$F$ has dimension~$2$ we can apply \cite[th.~4.2]{AltmanKleiman} to derive that~$Y$ is smooth, too.
\qed

\section{How to rule out the Tannaka group \texorpdfstring{$E_7$}{E7}}
\label{sec:E7}

In this section, which is largely independent of the rest of the paper, we explain how to rule out the occurrence of the Tannaka group $E_7$.

\subsection{Main result} \label{sec:StatementsNoE7} We will rule out $E_7$ via a general criterion for a subvariety to have a big Tannaka group if the alternating or symmetric convolution square has few summands. We go back to the notation introduced in \cref{sec:TannakaGroupIntro}, so let~$A$ be an abelian variety of dimension $g$ over an algebraically closed field $k$ of characteristic zero, and consider a smooth integral subvariety $X \subset A$ of dimension~$d$. We again fix a fiber functor $\omega \colon \langle \delta_X \rangle \to \Vect(\bbF)$ and denote the corresponding Tannaka group by
\[ G_{X, \omega} \; \subset \; \GL(V) \quad \textup{where} \quad V \; := \; \omega(\delta_X).\]
The derived category of constructible sheaf complexes is a symmetric pseudoabelian tensor category with respect to the convolution product. So we have a decomposition
\[
 \delta_X * \delta_X \;=\; \Sym^2(\delta_X) \; \oplus \; \Alt^2(\delta_X)
\] 
where the involution of $\delta_X * \delta_X$ permuting the two factors acts as $\id$ on $\Sym^2(\delta_X)$ and $-\id$ on $\Alt^2(\delta_X)$. Define
\[
T_+(\delta_X) \;:=\;
 \begin{cases} 
 \Sym^2(\delta_X) & \text{if $d$ is even}, \\
 \Alt^2(\delta_X) & \text{if $d$ is odd}.
 \end{cases} 
\]
If $X\subset A$ is \emph{symmetric up to translation} in the sense that~$X = a - X$ for some point~$a \in A(k)$, then the fiber of the sum morphism $\sigma\colon X\times X \to A$ over this point is the antidiagonal $\sigma^{-1}(a)=\{(x, a-x)\mid x\in X\}$. By proper base change 
and the decomposition theorem~\cite[th.~6.2.5]{BBDG}, we then have a skyscraper direct summand $\delta_a\subset \delta_X*\delta_X$, and a closer look at the parity shows that $\delta_a \subset T_+(\delta_X)$; see~\cite[sect.~1.2]{JKLM}~\cite[lemma 2.1]{KWGeneric}. When $X$ is nondivisible, the point $a$ is unique and we define
\[ T_+(\delta_X) = S_+(\delta_X) \oplus \delta_a\]
for a unique complex $S_+(\delta_X)$. By the dimension formula~\eqref{eq:dim}, the skyscraper direct summand corresponds to a one-dimensional representation of the Tannaka group, and the projection onto it defines a bilinear form $\theta \colon  V \otimes V \to \omega(\delta_a)$ which is symmetric if $d$ is even and alternating if $d$ is odd. Hence the derived subgroup~$G_{X,\omega}^\ast$ of the connected component of $G_{X, \omega}$ is contained in $\SO(V, \theta)$ is $d$ is even and in~$\Sp(V, \theta)$ if $d$ is odd. To have a uniform notation, when~$X$ is not symmetric up to translation we set
\[ S_+(\delta_X) \; := \; T_+(\delta_X).\]

\begin{theorem} \label{th:LarsenAlternative} If $X\subset A$ is nondivisible with ample normal bundle and $d < g / 2$, then 
$ S_+(\delta_X)$ is a perverse sheaf without negligible direct summands. If~$S_+(\delta_X)$ is simple, then
\[
G_{X, \omega}^\ast = 
\begin{cases}
\SL(V) & \text{if $X$ is not symmetric up to translation}, \\
\SO(V, \theta) & \text{if $X$ is symmetric up to translation and $d$ is even}, \\
\Sp(V, \theta) & \text{if $X$ is symmetric up to translation and  $d$ is odd}.
\end{cases}
\]
\end{theorem}

Note that in comparison to Larsen's alternative in~\cite[lemma 3.7]{JKLM}, we here only need an assumption on the symmetric or alternating square rather than on the full tensor square. The rest of this section will be devoted to the proof of the above criterion, see \cref{sec:proof-of-LarsenAlternative}. But before coming to the proof, let us show how~\cref{th:LarsenAlternative} rules out the exceptional Tannaka group $E_7$:

\begin{corollary} \label{Cor:NoE7}
If $X\subset A$ has ample normal bundle and $d < g / 2$, then
\[G_{X, \omega}^\ast \; \not \simeq \; E_7.\]
\end{corollary}

\begin{proof}
Suppose $G_{X, \omega}^\ast \simeq E_7$. 
If $X\subset A$ descends under an isogeny $f\colon A\to B$, then~$f(X)\subset B$ is still smooth nondegenerate and irreducible of the same dimension with Tannaka group $E_7$ by~\cite[rem.~2.6]{JKLM}~\cite[cor.~1.6(a)]{KraemerMicrolocalI}. Hence in what follows we assume that~$X$ is nondivisible. Then by~\cite[cor.~1.10]{KraemerMicrolocalI} the representation $V= \omega(\delta_X)$ of $G_{X,\omega}^\ast$ is minuscule and hence isomorphic to the $56$-dimensional irreducible representation of $E_7$. By the highest weight theory of the exceptional group $E_7$ then
\[
\Alt^2(V) \;\simeq\; W \oplus \mathbf{1},\]
where $W$ is irreducible and $\mathbf{1}$ is the trivial representation of dimension one. In particular, there is an alternating bilinear form on $V$ which is invariant under the action of $E_7$. It follows from the parity discussion in~\cite[sect.~1.2]{JKLM} that $d$ is odd and, after a suitable translation, that $X$ is symmetric. Then
\[
 \Alt^2(\delta_X) \;=\; S_+(\delta_X) \oplus \delta_0 
\]
and $S_+(\delta_X)$ is perverse by the first part of \cref{th:LarsenAlternative} and corresponds to the representation $\omega(S_+(\delta_X))\simeq W$. The latter is irreducible, so $S_+(\delta_X)$ is a simple perverse sheaf up to negligible direct summands. The second part of \cref{th:LarsenAlternative} then says that $G_{X, \omega}^\ast$ is the full symplectic group $\Sp(V, \theta)$, a contradiction.
\end{proof}

\subsection{Topological preliminaries} 
In this section we work over $k=\bbC$. For the proof of~\cref{th:LarsenAlternative} we will detect perverse direct summands of a convolution product from the top cohomology of the fibers of the sum morphism, using the following well-known fact:

\begin{lemma} \label{lem:top-cohomology}
Let $Z$ be a proper complex variety of dimension $d$, and let $n$ be the number of its irreducible components of dimension $d$.
Then $\dim_\bbC \rH^{2d}(Z, \bbC)=n$.
\end{lemma}

\begin{proof} 
The intersection of distinct irreducible components has real codimension at least two, hence via the Mayer-Vietoris sequence one reduces to the case where~$Z$ is irreducible. The same argument allows us to pass to a resolution of singularities and thus the claim follows. 
\end{proof}

The above lemma is about the cohomology with constant coefficients, so we want a criterion to decide when the perverse intersection complex on a singular variety is the constant sheaf. A complex variety $W$ is a \emph{rational homology manifold} if
\[
 \rH^i_{\{x\}}(W, \bbC) \;\simeq\; 
 \begin{cases} 
 \bbC & \text{for $i=2\dim W$}, \\
 0 & \text{for $i\neq 2\dim W$}.
 \end{cases}
\]
Any smooth variety is a rational homology manifold. More generally we have:

\begin{lemma} \label{lem:homology-manifold}
For any rational homology manifold $X$ and any $n\in \bbN$, the symmetric power $W=\Sym^n X$ is again a rational homology manifold. In particular, if $X$ is irreducible, then 
\[
 \delta_{W} \;\simeq\; \bbC_{W}[m]
 \quad \text{where} \quad m = \dim W = n\dim X.
\]
\end{lemma} 

\begin{proof} 
Any quotient of a rational homology manifold by the action of a finite group is a rational homology manifold~\cite[prop.~A.1(iii)]{BrionRationalSmoothness}. So $\Sym^n X = X^n/\frS_n$ is a rational homology manifold. But the intersection complex on any irreducible rational homology manifold is the constant sheaf~\cite[prop.~8.2.21]{Hotta}. 
\end{proof} 

\subsection{Kashiwara's estimate} In this section we continue to work over $k=\bbC$. We will control negligible summands in convolution products in terms of characteristic varieties. 
As in~\cite[def.~2.2]{JKLM}, we attach to any subvariety $Y\subset A$ its projective conormal variety
\[
 \PLambda_Y \;\subset\; \bbP(\Omega^1_A) \;=\; A \times \bbP_A
\]
where $\bbP_A = \bbP(\Lie(A)^\vee)$.
The \emph{characteristic variety} of a perverse sheaf $P\in \Perv(A)$ is defined as the support $\Char(P)=\Supp \CC(P)$ of the characteristic cycle $\CC(P)$, including possible negligible components. Its projectivization 
\[
 \PChar(P) \;\subset\; \bbP(\Omega^1_A)
\]
is a finite union of projective conormal varieties. For a sheaf complex $K\in \Dbc(A, \bbC)$ we denote by
\[
 \PChar(K) \;=\; \bigcup_{i\in \bbZ} \PChar(^p\cH^i(K))
\]
the union of the characteristic varieties of all its perverse cohomology sheaves. We can then control the characteristic varieties of convolution products $K=P*P$ by the following incarnation of Kashiwara's estimate:

\begin{lemma} \label{lem:Kashiwara-estimate}
Let $P_1, P_2\in \Perv(A)$. If the Gauss maps $\PChar(P_i)\to \bbP_A$ are finite for both $i=1,2$, then so is the Gauss map 
\[ 
 \PChar(P_1*P_2) \;\longrightarrow\; \bbP_A
\]
and hence $\PChar(P_1*P_2)$ does not contain any negligible components.
\end{lemma} 

\begin{proof} 
By definition $P_1*P_2 = R\sigma_*(P_1\boxtimes P_2)$ where $\sigma\colon A\times A \to A$ denotes the sum morphism. Thus by Kashiwara's upper estimate for the characteristic variety of proper direct images~\cite[th.~4.2(b)]{KashiwaraBFunctionsAndHolonomic} we have
\[
 \PChar(P_1*P_2) 
 \;\subset\; \tilde{\sigma}(\PLambda_1\times_{\bbP_A} \PLambda_2),
\]
where $\PLambda_i = \PChar(P_i)$ and $\tilde{\sigma} = \sigma \times \id_{\bbP_A}$. Our finiteness assumption on the Gauss map implies that the projection $f \colon \PLambda_1\times_{\bbP_A} \PLambda_2 \to \bbP_A$ is a finite morphism. The morphism $f$ factors through the projection $g \colon \tilde{\sigma}(\PLambda_1 \times_{\bbP_A} \PLambda_2) \to \bbP_A$. Thus $g$ is finite and so is its restriction to the subvariety $\PChar(P_1*P_2)\subset \tilde{\sigma}(\PLambda_1\times_{\bbP_A} \PLambda_2)$. 
\end{proof} 

\begin{corollary} \label{cor:no-negligibles}
Let $P_1, P_2\in \Perv(A)$. If the Gauss maps $\PChar(P_i)\to \bbP_A$ are both finite, then for all $i\in \bbZ$ and any perverse direct summand 
$Q\subset ^p\cH^i(P_1*P_2)$, 
the support \[ Y \;=\; \Supp(Q) \] is either of general type (in which case $Q$ is not negligible), or $Y=A$.
\end{corollary} 

\begin{proof} 
For $Q \subset {}^p\cH^i(P_1*P_2)$ with support $Y$ we have $\PLambda_Y \subset \PChar(P_1*P_2)$.
\end{proof}

\subsection{Decomposition of the convolution square} 
The main point in the proof of~\cref{th:LarsenAlternative} will be to deduce from our assumption about the symmetric or alternating square the decomposition of the full tensor square. To explain this we go back to the notation of \cref{sec:StatementsNoE7}. Let $T_-(\delta_X)$ denote the direct summand complementary to $T_+(\delta_X)$ in the decomposition of the convolution square $\delta_X \ast \delta_X$, so that
\[ \delta_X*\delta_X \;=\; T_+(\delta_X) \oplus T_-(\delta_X). \]
We control the second direct summand in this decomposition as follows:
 
\begin{proposition} \label{prop:larsen-alternative}
Suppose that $X\subset A$ has ample normal bundle and $d < g/2$.\smallskip
\begin{enumerate} 
\item The convolution $\delta_X * \delta_X$ is a perverse sheaf without negligible summands.\smallskip 
\item If $S_+(\delta_X)$ is simple, then \smallskip
\begin{enumerate} 
\item the sum morphism $f\colon \Sym^2 X \to X + X$ is birational, and\smallskip
\item $T_-(\delta_X)$ is a simple perverse sheaf with support $X+X$.
\end{enumerate} 
\end{enumerate}
\end{proposition}

\begin{proof} We can assume $k= \bbC$ and work in the analytic framework.\medskip

(1) By construction $\delta_X*\delta_X$ is a sheaf complex with support $X+X$, which by our dimension assumption is a strict subvariety of $A$.  Since for any smooth subvariety $X\subset A$ with nontrivial ample normal bundle the Gauss map $\PLambda_X \to \bbP_A$ is finite, we see from~\cref{cor:no-negligibles} that the convolution $\delta_X * \delta_X$ cannot contain any negligible direct summand. The claim now follows from the general fact that the convolution of any two simple perverse sheaves is a direct sum of a semisimple perverse sheaf and a negligible complex, due to the decomposition theorem~\cite[th.~6.2.5]{BBDG} and generic vanishing
~\cite[lemma~4.3c]{KraemerSemiabelian}.

\medskip 

(2) Part (1) implies that $S_+(\delta_X)$ is perverse. Now, the sum $X \times X \to X + X$ is the composite of the quotient morphism $\pi \colon X \times X \to \Sym^2 X$ with a unique map $f \colon \Sym^2 X \to X + X$. Write
\[ \pi_*(\bbC_{X\times X}) \;\simeq\; L_+ \oplus L_- \]
where $L_+$ are the invariants under the action of the deck transformation group of the double cover~$\pi$ and where~$L_-$ are the anti-invariants. Both $L_+$ and $L_-$ are constructible sheaves of generic rank one on $\Sym^2 X$. In fact, by adjunction one sees that
\[
 L_+  \;\simeq\; \bbC_{\Sym^2 X}
\] 
and by~\cref{lem:homology-manifold} this sheaf is the perverse intersection complex on $\Sym^2 X$ up to a shift. 
The definition of the commutativity constraint for the convolution product in~\cite[sect.~2.1]{WeissauerBNSheaves} moreover shows
\[
 T_+(\delta_X) \;\simeq\; Rf_* L_+[2d]
 \quad \text{and} \quad 
 T_-(\delta_X) \;\simeq\; Rf_* L_-[2d].
\]
The decomposition theorem~\cite[th.~6.2.5]{BBDG} together with the fact that $f$ is generically finite implies that $\delta_{X+X} \subset T_+(\delta_X)$. Since by assumption $S_+(\delta_X)$ is a simple perverse sheaf, we see that 
\[
 S_+(\delta_X) \;=\;  \delta_{X+X}.
\]
Then
$f_* L_+ \simeq \mathcal{H}^{-2d}(T_+(\delta_X))$
has generic rank one on its support $X+X$. Since by base change the generic rank of this direct image is the degree of the generically finite morphism 
\[
 f\colon \quad \Sym^2 X \;\longrightarrow\; X+X,
\]
it follows that $f$ is birational. This implies that also $\mathcal{H}^{-2d}(T_-(\delta_X)) \simeq f_* L_-$ has generic rank one on its support $X+X$. The structure of perverse sheaves then forces
\[
 T_-(\delta_X) \;=\; Q \oplus R
\] 
where $Q$ is a simple perverse sheaf of generic rank one on its support $X+X$ and~$R$ is either zero or a semisimple perverse sheaf with strictly smaller support. Note that~$R$ cannot be a skyscraper sheaf: By adjunction~\cite[cor.~1]{WeissauerBNSheaves}, any skyscraper summand of $\delta_X \ast \delta_X$ is supported at a point $a \in A(\bbC)$ with $X = a - X$, but since $X$ is nondivisible, there is at most one such skyscraper sheaf and it enters $T_+(\delta_X)$. \medskip

To show that $R= 0$, we argue by contradiction: Suppose $R\neq 0$. The support $Y=\Supp(R)$ has dimension $e=\dim Y < 2d = \dim X + X$. For general $y\in Y$ we have
\[
 0 \;\neq\; 
 \cH^{-e}(R)_y 
 \;\subset\;
 \cH^{-e}(T_-(\delta_X))_y 
 \;=\;
 R^{2d-e}f_* (L_-)_y
 \;\simeq\; 
 \rH^{2d-e}(F, {L_-}_{\vert F})
\]
for the fiber $F = f^{-1}(y)$. 
This nonvanishing forces $2 \dim F \ge 2d-e$. On the other hand 
\[
 \cH^{2\dim F-2d}(T_+(\delta_X))_y \;=\; R^{2\dim F} f_*(\bbC_{\Sym^2 X}) \;\simeq\; \rH^{2\dim F}(F, \bbC) \;\neq\; 0
\] 
by~\cref{lem:top-cohomology}. It follows that we have
\[ 2\dim F = 2d-e. \]
Indeed, if $2\dim F > 2d-e$, then the above nontrivial stalk cohomology would come from a non-perverse direct summand in $T_+(\delta_X)$, contradicting (1). The identity $2\dim F = 2d-e$ implies the above nontrivial stalk cohomology comes from a perverse direct summand in $S_+(\delta_X)$ whose support $Y$ is strictly smaller than $X+X$. This contradicts the simplicity of $S_+(\delta_X)$.
\end{proof}

The above statement is false if we exchange the roles of $S_+(\delta_X)$ and $T_-(\delta_X)$. Indeed, if $X$ is the Fano surface of lines on a smooth cubic threefold embedded in $A = \Alb(X)$, then $T_-(\delta_X) = \Alt^2(V)$ is simple but $S_+(\delta_X) = \Sym^2 V$ is not.

\subsection{Proof of \cref{th:LarsenAlternative}} \label{sec:proof-of-LarsenAlternative} 
Part (1) of \cref{prop:larsen-alternative} says that $S_+(V)$ is a perverse sheaf without negligible direct summands, so it only remains to show that if $S_+(\delta_X)$ is a simple perverse sheaf, then
\[
G_{X, \omega}^\ast = 
\begin{cases}
\SL(V) & \text{if $X$ is not symmetric up to translation}, \\
\SO(V, \theta) & \text{if $X$ is symmetric up to translation and $d$ is even}, \\
\Sp(V, \theta) & \text{if $X$ is symmetric up to translation and  $d$ is odd}.
\end{cases}
\]
But this follows from part (2b) of \cref{prop:larsen-alternative} and from the version of Larsen's alternative given in \cite[lemma 3.7]{JKLM}. For convenience we recall the simple argument when the subvariety $X\subset A$ is symmetric up to translation, which is the only one needed for the application to $E_7$. The symmetry of $X$ implies that the representation $V$ is self-dual and we have an isomorphism $V \otimes V \iso  \End(V)$. Let $\mathfrak{h} \subset \End(V)$ be the Lie algebra of $\SO(V, \theta)$ if $d$ is even and of $\Sp(V, \theta)$ if $d$ is odd. Via the preceding isomorphism we have
\[ 
\mathfrak{h} \iso
\begin{cases}
\Alt^2(V) & \textup{if $d$ is even},\\
\Sym^2(V) & \textup{if $d$ is odd}.
\end{cases}
\]
The representation $\mathfrak{h}$ of $G_{X, \omega}$ therefore corresponds via Tannaka duality to the perverse sheaf $S_-(\delta_X)$, which is simple by part (2b) of~\cref{prop:larsen-alternative}. Thus $\mathfrak{h}$ is an irreducible representation of the group $G_{X,\omega}$. On the other hand, the adjoint representation $\mathfrak{g}$ of this group is a subrepresentation of $\mathfrak{h}$, thus $\mathfrak{g} = \mathfrak{h}$.
\qed

\appendix

\section{Hodge number estimates} \label{sec:HodgeEstimates}

Let~$X$ be a smooth complex projective variety of dimension~$d$. The arguments of Lazarsfeld-Popa~\cite[th.~C(iii)]{LazarsfeldPopa} and Lombardi~\cite[cor.~2.3]{Lombardi13} can be applied in a slightly more general context to get bounds on the Euler characteristics~$\chi(X, \Omega^p_X)$ with the Albanese variety replaced by an arbitrary ambient abelian variety. For a linear subspace~$V\subset \rH^0(X, \Omega^1_X)$ put
\[
 m(V) \;:=\; \min \bigl\{ \, \codim(X, Z(\omega)) \mid 0\neq \omega \in V \, \bigr\}, 
\]
where~$Z(\omega) \subset X$ denotes the zero locus of~$\omega \in \rH^0(X, \Omega^1_X)$ as in loc.~cit. We will be interested in the case~$m(V)=d$, which is to say that every nonzero differential form~$\omega \in V$ has at most isolated zeroes. In this case, we have: 

\begin{theorem} \label{thm:hodge-estimate-for-subspace}
If~$V \subset \rH^0(X, \Omega^1_X)$ satisfies~$m(V) = d$ and~$\dim V > d + 1$, then
\[
 (-1)^{d-p}\chi(X, \Omega^p_X) \;\ge\; 
 \begin{cases} 
  \dim V - d +1& \text{for~$p \,\in\, \{0,d\}$},\\[0.1em]
   2 & \text{for~$p \,\in\, \{ 1, d - 1 \}$}, \\[0.1em]
 1 & \text{for~$2 \le p \le d-2$}.
 \end{cases} 
\]
\end{theorem} 
\begin{proof} 
For~$0\le p\le d$ and any~$v\in V\smallsetminus \{0\}$, the cup product with~$v$ gives rise to a complex
\[
 \bbL^p_{X,v}\colon \; 
 \bigl[ 0 \;\longrightarrow\; \rH^0(X, \Omega^p_X) 
 \;\stackrel{v \smile}{\longrightarrow}\; \rH^1(X, \Omega^p_X) 
 \;\stackrel{v \smile}{\longrightarrow}\; \cdots 
 \;\stackrel{v \smile}{\longrightarrow}\; \rH^d(X, \Omega^p_X) 
 \;\longrightarrow\; 0
 \bigr].
\]
For~$m(V)=d$, this complex is exact in all degrees~$i \neq d-p$ by~\cite[prop. 3.4]{GreenLazarsfeld}, Hodge symmetry and Serre duality, see~\cite[prop.~2.1]{Lombardi13}. By varying the vector~$v$ we can view these complexes as the fibers of a complex of vector bundles on the projective space~$\bbP=\bbP(V)$, the Green-Lazarsfeld derivative complex
\[
 \bbL^p_{X,V}\colon \; 
 \bigl[ 
 \; \cdots \rightarrow \cO_{\bbP}(-d+i)\otimes \rH^i(X, \Omega^p_X)\stackrel{\delta_i}{\longrightarrow} \cO_{\bbP}(-d+i+1)\otimes \rH^{i+1}(X, \Omega^p_X) 
 \rightarrow \cdots \;\bigr]
\]
sitting in degrees~$i=0,1,\dots, d$. By the above, we have
\begin{enumerate} 
\item $\cH^i(\bbL^p_{X,V})=0$ for all~$i\neq d-p$, and
\item the morphisms~$\delta_i$ have constant rank for all~$i$.
\end{enumerate} 
Property (2) implies that the cokernel~$F=\coker(\delta_{d-p-1})$ is locally free and then also that
\[
 E \;:=\; \ker\left(F \;\rightarrow\; \cO_{\bbP}(-p+1)\otimes \rH^{d-p+1}(X, \Omega^p_X) \right)
\]
is locally free. Property (1) then shows that the rank of the vector bundle~$E$ is given by
\[
 \rk(E) \;=\; (-1)^{d-p} \chi(X, \Omega^p_X).
\]
This shows~$(-1)^{d-p} \chi(X, \Omega^p_X)\ge 1$, since necessarily~$E\neq 0$: Else the complex~$\bbL^p_{X,V}$ would be exact in all degrees, which is impossible because any exact sequence of direct sums of line bundles on~$\bbP=\bbP(V)$ with differentials of degree one has length at least~$\dim V+1$ by~\cite[lemma 2.2]{Lombardi13}, and we assumed~$\dim V>d$.
\medskip 

When~$p = d-1$, let~$G := \ker(\delta_1)$ so the complex~$\bbL^p_{X,V}$ becomes an exact sequence
\[ 0 \; \too \; G \; \too \; \cO_{\bbP}(-d + 1) \otimes \rH^1(X, \Omega^{d-1}_X) \; \too \; \cdots \too \; \cO_\bbP \otimes \rH^d(X, \Omega^{d-1}_X) \; \too \; 0.\]
The differential~$\delta_0$ factors through an injective map 
$ h \colon \cO_{\bbP}(-d) \otimes \rH^0(X, \Omega^1_X) \into G$. The cokernel of~$h$ is locally free and it is nonzero by \cite[lemma 2.2]{Lombardi13}. If it were of rank~$1$, it would be of the form~$\cO_{\bbP}(j)$ for some integer~$j$, hence we would have a short exact sequence
\[ 0 \too \cO_{\bbP}(-d) \otimes \rH^0(X, \Omega^1_X) \too G \too \cO_{\bbP}(j) \too 0.\]
The vanishing of~$\rH^1(\bbP, \cO_{\bbP}(d + j))$ implies that the previous short exact sequence is split, hence~$G$ is a sum of line bundles on~$\bbP$. By loc. cit. this is not possible, thus
\[ - \chi(X, \Omega^{d-1}_X) \; = \; \rk G \; \ge 2. \]

It remains to prove the estimate~$\chi(X, \Omega^d_X) \ge \dim V - d+1$. For this take~$p=d$, and let~$G := \ker(\delta_0)$. We then have an exact sequence of vector bundles
\[
 0 \;\too\; G \;\too\; \cO_\bbP(-d)\otimes \rH^0(X, \Omega^d_X) \;\too\;
 \cO_\bbP(-d+1)\otimes \rH^1(X, \Omega^d_X) \;\too\; \cdots
\]
Taking the dual of the above sequence one sees that 
\[
 \rH^i(\bbP, G^\vee(j)) \;=\; 0 \quad \text{for all~$j\in \bbZ$ and all~$i\in \{1,2,\dots, \dim V - d - 1\}$}.
\]
If~$\rk G \le \dim V - d$, then by the Evans-Griffiths theorem~\cite[ex.~7.3.10]{LazarsfeldPositivityII} it would follow that~$G$ is a direct sum of line bundles. This is seen to be impossible by applying \cite[lemma 2.2 (ii)]{Lombardi13} with $e = d$, $q = \dim V$ and $a = -j$. Therefore, we have 
$
 \rk G \ge \dim V - d +1
$
and the claim follows since from the above exact sequence~$\rk G = \chi(X, \Omega_X^d)$. 
\end{proof} 

The above in particular applies when~$X\subset A$ is a subvariety of an abelian variety and~$V=\rH^0(A, \Omega_A^1)$. For a smooth ample divisor~$D \subset A$ the Euler characteristics can be computed exactly:
\[ (-1)^{g-1-i} \chi(D, \Omega^i_D) \; = \;  n A(g, i),  \]
where~$n = D^g/g!$ and~$A(g, i)$ is the Eulerian number; see \cite[sect.~3.2]{LS20}. For subvarieties of codimension~$\ge 2$, we obtain:

\begin{corollary} \label{cor:hodge-number-estimates}
Let~$A$ be a complex abelian variety of dimension~$g$, and let~$X\subset A$ be a smooth subvariety of dimension~$d\le g -2$ with ample normal bundle. Then we have
\[
 (-1)^{d-p}\chi(X, \Omega^p_X) \;\ge\; 
 \begin{cases} 
  g - d +1& \text{for~$p \,\in\, \{0,d\}$},\\[0.1em]
   2 & \text{for~$p \,\in\, \{ 1, d - 1 \}$}, \\[0.1em]
 1 & \text{for~$2 \le p \le d-2$}.
 \end{cases} 
\]

\end{corollary} 

\begin{proof} 
Since~$X\subset A$ has ample normal bundle, we know by~\cite[prop.~6.3.10(i)]{LazarsfeldPositivityII} that for nonzero 
$
\omega \in V := \rH^0(A, \Omega^1_A) \subset \rH^0(X, \Omega^1_X),
$
the zero locus~$Z(\omega) \subset X$ is finite. Hence,~$m(V)=d$, and~\cref{thm:hodge-estimate-for-subspace} applies.
\end{proof}

\small

\bibliography{./../biblio}

\providecommand{\bysame}{\leavevmode\hbox to3em{\hrulefill}\thinspace}
\providecommand{\MR}{\relax\ifhmode\unskip\space\fi MR }
\providecommand{\MRhref}[2]{%
  \href{http://www.ams.org/mathscinet-getitem?mr=#1}{#2}
}
\providecommand{\href}[2]{#2}
\begin{thebibliography}{DMOS82}

\bibitem[AK77]{AltmanKleiman}
A.~B. Altman and S.~L. Kleiman, \emph{Foundations of the theory of {F}ano
  schemes}, Compositio Math. \textbf{34} (1977), no.~1, 3--47.

\bibitem[BBDG18]{BBDG}
A.~Beilinson, J.~Bernstein, P.~Deligne, and O.~Gabber, \emph{Faisceaux pervers.
  {Actes} du colloque ``{Analyse} et {Topologie} sur les {Espaces}
  {Singuliers}''. {Partie} {I}}, 2nd edition ed., Ast{\'e}risque, vol. 100,
  Paris: Soci{\'e}t{\'e} Math{\'e}matique de France (SMF), 2018.

\bibitem[Bea82]{BeauvilleIntermediateJacobian}
A.~Beauville, \emph{Les singularit\'es du diviseur theta de la jacobienne
  interm\'ediaire de l'hypersurface cubique dans {{$\mathbb{P}^4$}}.},
  Algebraic threefolds, {Proc}. 2nd 1981 {Sess}. {C}.{I}.{M}.{E}.,
  {Varenna}/{Italy} 1981, {Lect}. {Notes} {Math}. 947, 190-208 (1982)., 1982.

\bibitem[Bou02]{Bou02}
N.~Bourbaki, \emph{Lie groups and {L}ie algebras. {C}hapters 4--6}, Elements of
  Mathematics (Berlin), Springer-Verlag, Berlin, 2002, Translated from the 1968
  French original by Andrew Pressley.

\bibitem[Bri99]{BrionRationalSmoothness}
M.~Brion, \emph{Rational smoothness and fixed points of torus actions},
  Transform. Groups \textbf{4} (1999), no.~2-3, 127--156.

\bibitem[BW79]{BruceWall}
J.~W. Bruce and C.~T.~C. Wall, \emph{On the classification of cubic surfaces},
  J. London Math. Soc. (2) \textbf{19} (1979), no.~2, 245--256.

\bibitem[CG72]{ClemensGriffiths}
C.~H. Clemens and P.~A. Griffiths, \emph{The intermediate {J}acobian of the
  cubic threefold}, Ann. of Math. (2) \textbf{95} (1972), 281--356.

\bibitem[CMPS18]{CasalainaPopaSchreieder}
S.~Casalaina-Martin, M.~Popa, and S.~Schreieder, \emph{Generic vanishing and
  minimal cohomology classes on abelian fivefolds}, J. Algebraic Geom.
  \textbf{27} (2018), no.~3, 553--581.

\bibitem[DE21]{DAE20}
M.~D'Addezio and H.~Esnault, \emph{On the universal extensions in tannakian
  categories}, International Mathematics Research Notices (2021).

\bibitem[Deb95]{Deb95}
O.~Debarre, \emph{Fulton-{H}ansen and {B}arth-{L}efschetz theorems for
  subvarieties of abelian varieties}, J. Reine Angew. Math. \textbf{467}
  (1995), 187--197.

\bibitem[DMOS82]{DM82}
P.~Deligne, J.~S. Milne, A.~Ogus, and K.~Shih, \emph{Hodge cycles, motives, and
  {S}himura varieties}, Lecture Notes in Mathematics, vol. 900,
  Springer-Verlag, Berlin-New York, 1982.

\bibitem[Dol12]{Dol12}
I.~V. Dolgachev, \emph{Classical algebraic geometry: A modern view}, Cambridge
  University Press, Cambridge, 2012.

\bibitem[EH87]{EisenbudHarrisMinimalDegree}
D.~Eisenbud and J.~Harris, \emph{On varieties of minimal degree (a centennial
  account)}, Algebraic geometry, {B}owdoin, 1985 ({B}runswick, {M}aine, 1985),
  Proc. Sympos. Pure Math., vol. 46, Part 1, Amer. Math. Soc., Providence, RI,
  1987, pp.~3--13.

\bibitem[FK00]{FraneckiKapranov}
J.~Franecki and M.~Kapranov, \emph{The {G}auss map and a noncompact
  {R}iemann-{R}och formula for constructible sheaves on semiabelian varieties},
  Duke Math. J. \textbf{104} (2000), no.~1, 171--180.

\bibitem[FP01]{RuledVarieties}
G.~Fischer and J.~Piontkowski, \emph{Ruled varieties}, Advanced Lectures in
  Mathematics, Friedr. Vieweg \& Sohn, Braunschweig, 2001, An introduction to
  algebraic differential geometry.

\bibitem[GL87]{GreenLazarsfeld}
M.~Green and R.~Lazarsfeld, \emph{Deformation theory, generic vanishing
  theorems, and some conjectures of {E}nriques, {C}atanese and {B}eauville},
  Invent. Math. \textbf{90} (1987), no.~2, 389--407.

\bibitem[Gro65]{EGAIV2}
A.~Grothendieck, \emph{{\'E}l{\'e}ments de g{\'e}om{\'e}trie alg{\'e}brique.
  {IV}: {\'E}tude locale des sch{\'e}mas et des morphismes de sch{\'e}mas.
  ({S{\'e}conde} partie)}, Publ. Math., Inst. Hautes {\'E}tud. Sci. \textbf{24}
  (1965), 1--231.

\bibitem[Har92]{HarrisAlgebraicGeometry}
J.~Harris, \emph{Algebraic geometry}, Graduate Texts in Mathematics, vol. 133,
  Springer-Verlag, New York, 1992, A first course.

\bibitem[HTT08]{Hotta}
R.~Hotta, K.~Takeuchi, and T.~Tanisaki, \emph{{$\mathcal{D}$}-modules, perverse
  sheaves, and representation theory}, Progress in Mathematics, vol. 236,
  Birkh\"{a}user Boston, Inc., Boston, MA, 2008, Translated from the 1995
  Japanese edition by Takeuchi.

\bibitem[JKLM25]{JKLM}
A.~Javanpeykar, T.~Kr{\"a}mer, C.~Lehn, and M.~Maculan, \emph{The monodromy of
  families of subvarieties on abelian varieties}, Duke Math. J. \textbf{174}
  (2025), no.~6, 1045--1149.

\bibitem[Kas76]{KashiwaraBFunctionsAndHolonomic}
M.~Kashiwara, \emph{B-functions and holonomic systems}, Invent. Math.
  \textbf{38} (1976), 33--54.

\bibitem[KM23]{KM}
T.~Kr\"amer and M.~Maculan, \emph{Arithmetic finiteness of irregular
  varieties}, 2023,
  \href{https://arxiv.org/abs/2310.08485}{\texttt{arXiv:2310.08485}}.

\bibitem[Kr{\"a}14]{KraemerSemiabelian}
T.~Kr{\"a}mer, \emph{Perverse sheaves on semiabelian varieties}, Rend. Semin.
  Mat. Univ. Padova \textbf{132} (2014), 83--102.

\bibitem[Kr{\"a}16]{KraemerCubicThreefolds}
\bysame, \emph{Cubic threefolds, {Fano} surfaces and the monodromy of the
  {Gauss} map}, Manuscr. Math. \textbf{149} (2016), no.~3-4, 303--314.

\bibitem[Kr{\"a}22]{KraemerMicrolocalI}
\bysame, \emph{Characteristic cycles and the microlocal geometry of the {G}auss
  map, {I}}, Ann. Sci. \'{E}c. Norm. Sup\'{e}r. (4) \textbf{55} (2022),
  1475--1527.

\bibitem[KW15a]{KWGeneric}
T.~Kr{\"a}mer and R.~Weissauer, \emph{On the {Tannaka} group attached to the
  theta divisor of a generic principally polarized abelian variety}, Math. Z.
  \textbf{281} (2015), no.~3-4, 723--745.

\bibitem[KW15b]{KWVanishing}
\bysame, \emph{Vanishing theorems for constructible sheaves on abelian
  varieties}, J. Algebraic Geom. \textbf{24} (2015), no.~3, 531--568.

\bibitem[Laz04]{LazarsfeldPositivityII}
R.~Lazarsfeld, \emph{Positivity in algebraic geometry. {II}}, Ergebnisse der
  Mathematik und ihrer Grenzgebiete. 3. Folge. A Series of Modern Surveys in
  Mathematics, vol.~49, Springer-Verlag, Berlin, 2004, Positivity for vector
  bundles, and multiplier ideals.

\bibitem[Lom13]{Lombardi13}
L.~Lombardi, \emph{Inequalities for the {Hodge} numbers of irregular compact
  {K{\"a}hler} manifolds}, Int. Math. Res. Not. \textbf{2013} (2013), no.~1,
  63--83.

\bibitem[LP10]{LazarsfeldPopa}
R.~Lazarsfeld and M.~Popa, \emph{Derivative complex, {BGG} correspondence, and
  numerical inequalities for compact {K{\"a}hler} manifolds}, Invent. Math.
  \textbf{182} (2010), no.~3, 605--633.

\bibitem[LS20]{LS20}
B.~Lawrence and W.~Sawin, \emph{The {S}hafarevich conjecture for hypersurfaces
  in abelian varieties}, 2020,
  \href{https://arxiv.org/abs/2004.09046v3}{\texttt{arXiv:2004.09046v3}}.

\bibitem[MP12]{MarcucciPirola}
V.~O. Marcucci and G.~P. Pirola, \emph{Points of order two on theta divisors},
  Atti Accad. Naz. Lincei, Cl. Sci. Fis. Mat. Nat., IX. Ser., Rend. Lincei,
  Mat. Appl. \textbf{23} (2012), no.~3, 319--323.

\bibitem[Ran81]{Ran}
Z.~Ran, \emph{On subvarieties of abelian varieties}, Invent. Math. \textbf{62}
  (1981), no.~3, 459--479.

\bibitem[Rou09]{RoulleauManuscripta}
X.~Roulleau, \emph{Elliptic curve configurations on {Fano} surfaces}, Manuscr.
  Math. \textbf{129} (2009), no.~3, 381--399.

\bibitem[Sak10]{Sakamaki}
Y.~Sakamaki, \emph{Automorphism groups on normal singular cubic surfaces with
  no parameters}, Trans. Amer. Math. Soc. \textbf{362} (2010), no.~5,
  2641--2666.

\bibitem[Sch15]{SchnellHolonomic}
C.~Schnell, \emph{Holonomic {$\mathcal{D}$}-modules on abelian varieties},
  Publ. Math., Inst. Hautes {\'E}tud. Sci. \textbf{121} (2015), 1--55.

\bibitem[Sch17]{SchreiederIMRN}
S.~Schreieder, \emph{Decomposable theta divisors and generic vanishing}, Int.
  Math. Res. Not. IMRN (2017), no.~16, 4984--5009.

\bibitem[SS18]{MHMProject}
C.~Sabbah and C.~Schnell, \emph{The {MHM} project}, 2018,
  \url{https://perso.pages.math.cnrs.fr/users/claude.sabbah/MHMProject/mhm.pdf}.

\bibitem[SY23]{SchnellYang}
C.~Schnell and R.~Yang, \emph{Hodge modules and singular {Hermitian} metrics},
  Math. Z. \textbf{303} (2023), no.~2, 20.

\bibitem[Wei06]{WeissauerBNSheaves}
R.~Weissauer, \emph{Brill-noether sheaves}, ar{X}iv, 2006,
  \href{https://arxiv.org/abs/math/0610923}{\texttt{arXiv:math/0610923}}.

\bibitem[Wei15]{WeissauerAlmostConnected}
\bysame, \emph{Why certain {T}annaka groups attached to abelian varieties are
  almost connected}, ar{X}iv, 2015,
  \href{https://arxiv.org/abs/1207.4039}{\texttt{arXiv:1207.4039}}.

\end{thebibliography}

\bibliographystyle{amsalpha}

\end{document}